\pdfoutput=1
\documentclass[10pt]{article}
\usepackage{amssymb}
\usepackage{graphicx}
\usepackage{xcolor} 
\usepackage{tensor}
\usepackage{fullpage} 
\usepackage{amsmath}
\usepackage{amsthm}
\usepackage{verbatim}
\usepackage{mathrsfs}
\usepackage{hyperref}
\usepackage{authblk}

\newcommand*{\avint}{\mathop{\ooalign{$\int$\cr$-$}}}

\usepackage{scalerel,stackengine}
\stackMath
\newcommand\wwhat[1]{%
\savestack{\tmpbox}{\stretchto{%
  \scaleto{%
    \scalerel*[\widthof{\ensuremath{#1}}]{\kern-.6pt\bigwedge\kern-.6pt}%
    {\rule[-\textheight/2]{1ex}{\textheight}}
  }{\textheight}%
}{0.5ex}}%
\stackon[1pt]{#1}{\tmpbox}%
}
\usepackage{enumitem}
\setlist[enumerate]{leftmargin=1.5em}
\setlist[itemize]{leftmargin=1.5em}

\usepackage{seqsplit}

\definecolor{green}{rgb}{0,0.8,0} 

\newtheorem{maintheorem}{Theorem}

\newtheorem{theorem}{Theorem}[section]
\newtheorem{lemma}[theorem]{Lemma}
\newtheorem{proposition}{Proposition}[section]

\theoremstyle{definition}
\newtheorem{definition}[maintheorem]{Definition}

\theoremstyle{remark}
\newtheorem{remark}[proposition]{Remark}

\numberwithin{equation}{section}
\newcommand{\RN}[1]{%
  \textup{\uppercase\expandafter{\romannumeral#1}}%
}

\newcommand{\brk}[1]{\langle#1\rangle}
\newcommand{\R}{{ \mathbb{R}  }}

\newcommand{\bke}[1]{\left( #1 \right)}
\newcommand{\bkt}[1]{\left[ #1 \right]}
\newcommand{\bket}[1]{\left\{ #1 \right\}}
\newcommand{\norm}[1]{\left\Vert #1 \right\Vert}
\newcommand{\abs}[1]{\left| #1 \right|}

\newcommand{\nnrm}[1]{{\vert\kern-0.25ex\vert\kern-0.25ex\vert #1 
    \vert\kern-0.25ex\vert\kern-0.25ex\vert}}

\newcommand{\ud}{\mathrm{d}}
\newcommand{\rd}{\partial}

\newcommand{\ackn}[1]{
\addtocontents{toc}{\protect\setcounter{tocdepth}{1}}
\subsection*{Acknowledgements} {#1}
\addtocontents{toc}{\protect\setcounter{tocdepth}{2}} }

\begin{document}

\title{Stability and instability for the fully parabolic Keller-Segel system around constant equilibrium}
	
\author[1]{Jaewook Ahn\thanks{jaewookahn@dgu.ac.kr}}
\author[2]{Jae-Myoung Kim\thanks{jmkim02@anu.ac.kr}}
\author[3]{Junha Kim\thanks{junha02@ajou.ac.kr}}
\affil[1]{Department of Mathematics, Dongguk University, Seoul, 04620, Republic of Korea}
\affil[2]{Department of Mathematics Education, Gyeongkuk National University, Andong-si, Gyeongsangbuk-do, 36729, Republic of Korea}
\affil[3]{Department of Mathematics, Ajou University, Suwon-si, Gyeonggi-do, 16499, Republic of Korea}
\maketitle

\begin{abstract}
This paper studies the Cauchy problem for the fully parabolic Keller-Segel system. The main results show that there exists a critical threshold $A_{\rm crit}>0$ for steady states $(A,A)$ such that the steady states are nonlinearly stable when $A\le A_{\rm crit}$ and nonlinearly unstable when $A>A_{\rm crit}$. 
 We discuss asymptotic convergence rates as well.  In the subcritical case $A<A_{\rm crit}$, the rates correspond to those of the heat equation, and in the critical case $A=A_{\rm crit}$, the rates correspond to half those of the heat equation.\end{abstract}

\section{Introduction}
In the 1970s, Keller and Segel proposed a mathematical model describing chemotaxis \cite{KS71}.  The Keller-Segel model captures in principle the dynamic interplay between cell density and the concentration of a chemoattractant, providing a framework to explore the emergence of complex structures. 
 
This paper studies the Keller-Segel system
\begin{equation}\label{INTROKS}
\rd_t b -\Delta b= -\nabla \cdot (b \nabla a),\qquad\qquad
\tau\rd_t a -\Delta a  = b- \gamma a,
\end{equation} 
where $b$ and $a$ denote the cell density and the chemoattractant concentration, respectively.  One of the central questions in the study of the Keller-Segel system has been the global existence and blow-up of solutions. Depending on the initial conditions and parameters $\tau$ and $\gamma$, solutions may either remain globally bounded for all time or experience finite- or infinite-time blow-up. Indeed, extensive study has been devoted to establishing rigorous conditions for global existence and blow-up over the past decades \cite{B99,BKLN06,OY01,W10,W13}.  Classical results indicate that in two-dimensional space, if the total mass of the initial cell density is below or equal to a certain critical value, solutions exist globally, whereas for larger mass values, blow-up can occur in finite time \cite{BKLN06,BCM08,BDP06}. 
For more results on global existence and blow-up, including those in dimensions three and higher, see the review paper~\cite{BBTW15}. 
A related question in the mathematical analysis of the Keller-Segel system is the stability problem. Determining the conditions under which solutions remain stable or become unstable is crucial for predicting their behavior. This is also a classical topic in the context of the Keller-Segel model. In the very beginning, Keller and Segel regarded the onset of slime mold aggregation as an instability and analyzed the linear instability near constant steady states in a bounded domain~\cite{KS71}. Similar stability problems have been further investigated for more general Keller–Segel systems by Schaaf~\cite{S85}. For the Cauchy problem of \eqref{INTROKS} in $\R^{2}$ with $\tau=\gamma=0$, the asymptotic stability of steady states with a critical mass have been investigated in \cite{BCC12,CF13,LNY13}. We refer the reader to \cite{CGMN22, DDDMW24, GS24, RS14,V02} for related results on the stability and instability of singular solutions. 

The main purpose of this study is, first, to find a critical threshold that determines the Keller-Segel system’s transition between stable and unstable regimes, and second, to establish a complete dichotomy between nonlinear stability and instability, including the critical case.
  There have been several thresholds proposed that distinguish the stability and instability of the Keller-Segel system. For instance, Winkler~\cite{W18} considered
a parabolic-elliptic counterpart of \eqref{INTROKS} with the second equation 
$
-\Delta a-\avint_{B(0;R)} b=b$ in the ball $B(0;R)\subset\R^{d}$, $d\ge2$, and identified a critical mass threshold
$
 M_{\rm crit}
$
for radial initial data, beyond which constant steady states become unstable and eventually undergo finite-time blow-up. If the total mass is below $M_{\rm crit}$, then there exist radial initial data that are more concentrated than the constant steady states, but the system still admits globally bounded solutions. The question remains open when the total mass equals $M_{\rm crit}$.
On the other hand, a critical threshold
 \[
 A_{\rm crit}=1
 \] 
for
 the Cauchy problem 
 \eqref{INTROKS} of parabolic-elliptic type, $\tau=0$ and $\gamma=1$, 
 was found by 
 Cygan et al.~\cite{CKKW21}. The constant steady states $(A,A)$ of the Cauchy problem are nonlinearly stable in the subcritical case $A< A_{\rm crit}$, and nonlinearly unstable in the supercritical case  $A> A_{\rm crit}$.
  For the critical case $A= A_{\rm crit}$, the linear stability result was shown in \cite[Rmk 4.9]{CKKW21} but the nonlinear stability result remains open.
Compared to the parabolic-elliptic system, only particular results have been established for the fully parabolic case.  The result we are aware of in this direction is due to Guo-Hwang~\cite{GH10}, who obtained the nonlinear instability of constant steady states in a periodic domain setting. 

 In the present work, we study the 
stability problem 
 for the fully parabolic Keller-Segel system near a constant steady state. Since the values of $\tau>0$ and  $\gamma>0$ do not play any role in establishing a complete dichotomy, we
assume, for simplicity, that 
\[
\tau=1=\gamma,
\] and consider the Cauchy problem for \eqref{INTROKS} near $(A,A)$ with $A\in [0,\infty)$. Denoting the perturbations 
$
n=b-A$ and $c=a-A$,
 we reformulate it  as 
\begin{equation}\label{eq_crt}
\left\{
\begin{aligned} 
&\rd_tn -\Delta(n-Ac)   = -\nabla \cdot (n \nabla c), \\
& \rd_tc +(1-\Delta)c - n = 0,\\
& n(\cdot,0)=n_{0},\quad c(\cdot,0)=c_{0},
\end{aligned}
\right.\qquad x\in \R^{d},\,t>0.
\end{equation} 
Our notion of nonlinear (Lyapunov) stability and instability reads as follows (see, e.g., \cite[Def.~2.1]{FPS06}).
\begin{definition}
	Let $<\!\!X, Z\!\!>$ be a pair of Banach spaces. An equilibrium $(A,\, A)$, $A\in [0,\infty)$, of \eqref{INTROKS} is called $<\!\!X,Z\!\!>$ nonlinearly (Lyapunov) stable if for any $\varepsilon > 0$, there exists $\delta > 0$ such that  $(n_0,\, c_0)\in X$ and \begin{equation*}
		\begin{gathered}
			\| (n_0,c_0) \|_{Z} \le \delta
		\end{gathered}
	\end{equation*} imply that there exists a global-in-time solution to \eqref{eq_crt} such that $
			(n(t),c(t)) \in \mathcal{C}([0,\infty);X)$
 and  \[
			\sup_{t\ge0} \| (n(t), c(t)) \|_{Z} \le \varepsilon.
			\] 
The equilibrium $(A,\,A)$ that is not nonlinearly stable is called nonlinearly unstable.
\end{definition}  
Before stating our main results, let us summarize the contributions of this paper.
\begin{itemize}
\item {\bf Stable versus unstable dichotomy}.  A dichotomy between nonlinear stability and instability, including the critical case, is established.
 It will turn out that   steady states $(A,A)$ of  \eqref{INTROKS} with $\tau=\gamma =1$
 are nonlinearly stable if  $A\le   A_{\rm crit}(=1)$ and nonlinearly  unstable if $A>   A_{\rm crit}$. 
\item {\bf Nonexistence of a critical mass threshold}. The Cauchy problem  \eqref{INTROKS} with $\tau > 0$, $\gamma = 0$ in $\R^{2}$ has
a critical threshold 
$
M_{\rm crit}(\tau)>0
$
such that all solutions emanating from smooth radial initial data with  supercritical mass
$\int_{\R^{2}}b(\cdot,0)>M_{\rm crit}(\tau)$ blow up in finite time \cite{M20} (see also \cite{M16,M20A}). It will turn out that
such a critical mass phenomenon does not hold whenever $\gamma>0$. Indeed, there is a global bounded solution of \eqref{INTROKS} at any mass level. See Remark~\ref{RMK12}.
\item {\bf Estimation of convergence rates}. The asymptotic convergence rates of $b$ and $a$ are obtained. If $A<   A_{\rm crit}$, the convergence rates correspond to those of the heat equation and if 
$A=  A_{\rm crit}$, the convergence rates correspond to half those of the heat equation. It might also be interesting to note that $b$ and $a$ get closer to each other at a faster rate than their convergence rates.
\end{itemize}
Throughout this paper, $C$ denotes a generic constant that may change from line to
line, $\mathcal{S}(\R^{d})$ denotes the Schwartz class. We denote $D^{k}=(-\Delta)^{\frac{k}{2}}$ for integer $k\ge0$, $\|f\|_{X\cap Y}:=\|f\|_{X}+\|f\|_{ Y}$, and  
\[
\mathscr{F}f(\xi)=\widehat{f}(\xi)=\int_{\mathbb{R}^{d}}e^{-ix\cdot\xi}f(x) \,\ud x,\qquad \mathscr{F}^{-1}f(x)=\frac{1}{(2\pi)^{d}}\int_{\mathbb{R}^{d}}e^{ix\cdot\xi}f(\xi)\,\ud \xi.
\]
We use the notation $g\lesssim h$ to indicate the estimate $g\le Ch$ for a universal constant $C>0$.
We mostly suppress  $\R^{d}$ in $L^{p}(\R^{d})$.

The first main result shows nonlinear stability in the subcritical case 
$ A<1(=A_{\rm crit})$.
 \begin{theorem}\label{thm1}
Let  $ A\in[0,1)$, $d \geq 2$,  $(n_0, c_0) \in  (H^{s}\times  H^{s+1})(\R^d)$ for some integer $s>\frac{d}{2}$, and let $j=0$ or $1$. Then, there exist $\delta_{j}>0$ such that if 
	\begin{equation}\label{cond1}
	\begin{aligned}
		\|\widehat{n_0}\|_{L^{1}(\R^d)} +  \||\xi|^{j}\widehat{c_0}\|_{L^{1}(\R^d)} \le \delta_{j},
	\end{aligned}
\end{equation}
then \eqref{eq_crt} possesses a unique global and smooth solution $(n,c)$ satisfying for some $C=C(A)>0$ 
\[
\|\widehat{n}\|_{L^{\infty}(0,\infty;L^{1}(\R^d))} +  \||\xi|^{j}\widehat{c}\|_{L^{\infty}(0,\infty;L^{1}(\R^d))} \leq C \delta_j,
\]
\[
 (n,c) \in \mathcal{C}([0,t);(H^{s}\times H^{s+1})(\R^d))\cap L^{2}(0,t;(H^{s+1}\times H^{s+2})(\R^d))\,\,\mbox{ for all }\,\,t>0.
\]
Moreover, the solution $(n,c)$  decays in time, in the sense that for  any integer $m\ge0$, there exists $C>0$ satisfying 
\[
\| D^{m} n(\cdot,t) \|_{L^2(\R^{d})} + \| D^{m} c (\cdot,t)\|_{L^2(\R^{d})}  \le Ct^{-\frac{m}{2}}\qquad\mbox{ for all }\,\, t>1, 
\]
\[
 \| D^m(n-c)(\cdot,t) \|_{L^2(\R^{d})} \le C(t^{-(\frac m2 +\frac d4+1)} + At^{-(\frac{m}{2}+1)})\qquad\mbox{ for all }\,\, t>1. 
\]
\end{theorem}
\begin{remark}
If $d\ge2$ and $s>\frac{d}{2}$, then 
$\|\widehat{f} \|_{L^{1}(\R^d)}\lesssim \| f \|_{H^{s}(\R^d)}$ for all $f\in H^{s}(\R^d)$.\end{remark}
\begin{remark}\label{RMKL1}
When $A=0$, if we further assume in Theorem~\ref{thm1} that $n_{0}, c_{0}$ are in $L^{1}(\R^{d})$ and nonnegative, then it holds that $\|n(\cdot,t)\|_{L^{1}(\R^{d})}= \int_{\R^{d}} n_{0}$,  $ \|c(\cdot,t)\|_{L^{1}(\R^{d})}=e^{-t} \int_{\R^{d}} c_{0} +(1-e^{-t}) \int_{\R^{d}} n_{0}$, and $n$ and $c$ are non-negative for all $t>0$. Moreover, these $L^{1}$-bounds yield better decay rates,  
$\| D^{m} n(\cdot, t)\|_{L^2(\R^{d})}+\| D^{m} c(\cdot, t) \|_{L^2(\R^{d})} \lesssim  t^{-(\frac{m}{2}+\frac{d}{4})}$, $\| D^{m} (n-c)(\cdot, t)\|_{L^2(\R^{d})} \lesssim  t^{-(\frac{m}{2}+\frac{3d}{4}+1)}$ for all $t\ge1$ and integer $m\ge0$. For the proof, see Appendix~\ref{secA2}.
\end{remark}
\begin{remark}\label{RMK12} Consider the case $A=0$, $d=2$. According to
Theorem~\ref{thm1} and Remark~\ref{RMKL1}, for any $M>0$ there exists $(n_{0},c_{0})$ with $\int_{\R^{2}}n_{0}=M$ such that   \eqref{eq_crt} has a global bounded solution. For instance, if   
 $n_{0}(x)=\mu e^{-\sigma |x|^2}$, $\mu,\sigma>0$,  i.e., $\widehat{n_{0}}(\xi)= \pi(\frac{\mu}{\sigma}) e^{- |\xi|^{2}/(4\sigma)}$, then   we can take small 
$\|\widehat{n_0}\|_{L^{1}(\R^2)}=  4\pi^{2}\mu$ while $\int_{\R^{2}}n_{0}=\pi\frac{\mu}{\sigma}$
to be arbitrarily large.
\end{remark}
Under a stronger assumption on $(n_0,c_0)$, 
nonlinear stability  is shown in the critical case $A=1$.
\begin{theorem}\label{thm_crt}
Let $A= 1$, $d \geq 2$, and   $(n_0, c_0) \in  (H^{s}\times H^{s+1}) (\R^d)$ for some  integer $s>\frac{d}{2}$. Then, 
 there exists $\delta >0$ such that if
$n_{0}$ and $c_{0}$ further satisfy  $n_{0},\nabla c_{0}\in L^{1}(\R^d)$, $\widehat{c_0}\in L^{\infty}(\R^d)$, and  
\begin{equation}\label{cond2}
 \norm{    \widehat{n_0}  }_{L^{1}\cap L^{\infty}(\R^{d})}  + \norm{|\xi|\widehat{c_0} }_{L^{1}\cap L^{\infty}(\R^{d})}\le \delta,
\end{equation}
then \eqref{eq_crt} possesses a unique global and smooth solution $(n,c)$ satisfying for some $C>0$ independent of $\delta$  
\[
  \norm{    \widehat{n}   }_{L^{\infty}(0,\infty;L^{1}\cap L^{\infty}(\R^{d}))}  + \norm{|\xi|\widehat{c} }_{L^{\infty}(0,\infty;L^{1}\cap L^{\infty}(\R^{d}))} \le C\delta, 
 \]
\[
 (n,c) \in \mathcal{C}([0,t);(H^{s}\times H^{s+1})(\R^d))\cap L^{2}(0,t;(H^{s+1}\times H^{s+2})(\R^d))\,\,\mbox{ for all }\,\, t>0.
\]
Moreover, the solution $(n,c)$  decays in time, in the sense that for  any integer $m\ge0$,  there exists $C>0$ satisfying
\[
\| D^{m} n(\cdot,t) \|_{L^2(\R^{d})} + \| D^{m} c (\cdot,t)\|_{L^2(\R^{d})}  \le Ct^{-\frac{1}{2}(\frac{m}{2} + \frac{d}{4})}\qquad\mbox{ for all }\,\, t>1, 
\]
\[
\| D^m(n-c)(\cdot,t) \|_{L^2(\R^{d})} \le C t^{-\frac{1}{2}(\frac{m}{2} + \frac{d}{4}+1)}\qquad\mbox{ for all }\,\, t>1. 
\]
\end{theorem}
The last result shows nonlinear instability in the supercritical case $A>1$.
\begin{theorem}\label{thm3}
Let $A>1$, $d\ge2$, and let  $ s>\frac{d}{2}$ be an integer. Then, the  steady state $(0,0)$ of the system \eqref{eq_crt} is nonlinearly unstable; that is, there exist $C_{0}>1$ and $\eta_0>0$ such that for all $\eta\in(0,\eta_{0})$ there exist $T_{\eta}=\frac{4A}{(A-1)^{2}}\ln \frac{C_{0}\eta_{0}}{\eta} <T_{\rm max}$ and
 $(n_{0,\eta},c_{0,\eta})\in  (H^{s}\times H^{s+1}) (\R^d)$ with the property that
$
\|(n_{0,\eta},c_{0,\eta})\|_{(H^{s}\times H^{s+1})(\R^{d})}= \eta  
$
but the solution $(n,c)$ to \eqref{eq_crt} with $(n_{0},c_{0})=(n_{0,\eta},c_{0,\eta})$  satisfies
\begin{equation}\label{THM13GOAL}
\min\{\| n(\cdot,T_{\eta})\|_{L^{2}}, \| c(\cdot,T_{\eta})\|_{L^{2}},\|\nabla  c(\cdot,T_{\eta})\|_{L^{2}}\}\ge \frac{\eta_{0}}{4}.
\end{equation}
\end{theorem}
We briefly describe the overall proof strategy. 
The starting point is the spectral analysis of the linearized operator, which plays a central role in both the stability and instability analyses.
In Section~\ref{SEC:PRE}, we reformulate the problem in the frequency domain and analyze the associated linear operator, which turns out to be stable for $A\le 1$ and unstable for $A>1$.
In the subcritical (Section~\ref{SEC11}) and critical (Section~\ref{SEC12}) cases, we obtain global-in-time bounds and establish nonlinear Lyapunov stability by controlling the nonlinear term $-\nabla\cdot(n\nabla c)$ under suitable smallness assumptions on the initial perturbations. Subsequently, we derive temporal decay estimates, revealing that differing low-frequency structures account for the heat-type decay rate when $A<1$ and half that rate when $A=1$.
  In the supercritical case (Section~\ref{SEC13}), we construct initial data concentrated near an unstable Fourier mode. The growing linear part is shown to dominate the nonlinear remainder up to a suitable escape time, which proves nonlinear instability.

\section{Preliminary}\label{SEC:PRE}
The local existence result for \eqref{eq_crt} reads as follows. We omit the proof since it can be shown by the standard fixed point argument. See, e.g., \cite[Theorem~2.1]{AKKL17}.
\begin{lemma}\label{LEMLWP}
Let $A\ge0$, $d\ge2$, and $(n_0,c_0) \in (H^{s}\times H^{s+1})(\mathbb{R}^d)$ for some integer $s > \frac d2$. Then, there exists the maximal time of existence $T_{\rm max}\le\infty$, such that a unique regular solution $(n,c)$ of \eqref{eq_crt} exists and satisfies for all $t<T_{\rm max}$
\[
(n,c) \in \mathcal{C}([0,t);(H^{s}\times H^{s+1})(\R^d))\cap L^{2}(0,t;(H^{s+1}\times H^{s+2})(\R^d)).
\]
Moreover, if $ T_{\rm max}<\infty$, then
$
\displaystyle{\lim_{t\nearrow T_{\rm max}}} (\|n(\cdot,t)\|_{H^{s}(\R^{d})}+\|c(\cdot,t)\|_{H^{s+1}(\R^{d})})=\infty$.
\end{lemma}
\begin{remark}
Note that the blow-up criteria,
 the last line of Lemma~\ref{LEMLWP}, can be refined to 
\begin{equation}\label{BUCRITREF}
\lim_{t\nearrow T_{\rm max}} \|n(\cdot,t)\|_{H^{s}(\R^{d})}=\infty
\end{equation}
because    
$
 \frac{d}{dt}\|c\|_{H^{s+1}(\R^{d})}^{2}+\|c\|_{H^{s+2}(\R^{d})}^{2}+\|c\|_{H^{s+1}(\R^{d})}^{2}\lesssim \|n\|_{H^{s}(\R^{d})}^{2}$  for all $t<T_{\rm max}$. Note also that, by a bootstrapping argument, the solution constructed in Lemma~\ref{LEMLWP} is smooth for all $t\in(0,T_{\rm max})$.
\end{remark}

In the following Lemma~\ref{KATOTYPELEM}, we introduce a Kato-Ponce type inequality which is useful for our purpose. We omit the proof since
it is similar to that in  \cite[Lemma~3.4]{MB03}, where the classical H\"older and  Gagliardo-Nirenberg inequalities are the main ingredients.
\begin{lemma}\label{KATOTYPELEM}
Let $\alpha\ge0$, $\beta\ge0$ be integers. There exists $C>0$ such that for all $f,g\in (L^{\infty}\cap H^{\alpha+\beta})(\R^{d})$
\begin{equation}\label{KATO1}
\|D^{\alpha}(fD^{\beta}g)\|_{L^{2}(\R^{d})}\le C(\|D^{\alpha+\beta}f\|_{L^{2}(\R^{d})}\|g\|_{L^{\infty}(\R^{d})}+\|f\|_{L^{\infty}(\R^{d})}\|D^{\alpha+\beta}g\|_{L^{2}(\R^{d})}).
\end{equation}
\end{lemma}
We close this section by deriving the Duhamel formula for \eqref{eq_crt}. 
Applying the Fourier transform to  \eqref{eq_crt} leads to  
\[
\partial_{t}\widehat{n}+|\xi|^{2}\widehat{n}-A|\xi|^{2}\widehat{c}= -i\xi\cdot\widehat{ (n \nabla c)},\qquad \partial_{t}\widehat{c}-\widehat{n}+(1+|\xi|^{2})\widehat{c}=0,
\]
 or equivalently,
\begin{equation} \label{FTROFEQ}
\left\{
\begin{aligned} 
&\widehat{n}(t) = e^{-|\xi|^2 t} \widehat{n_0} - \int_0^t e^{-|\xi|^2(t-\tau)} i \xi \cdot (\widehat{n \nabla c}+Ai\xi \,\widehat{c}) \,\ud \tau, \\
&\widehat{c}(t) = e^{-(1+|\xi|^2) t} \widehat{c_0} + \int_0^t e^{-(1+|\xi|^2)(t-\tau)} \widehat{ n} \,\ud \tau.
\end{aligned}
\right.
\end{equation}
By denoting 
$
\widehat{u}:=(\widehat{n},\widehat{c})^{\mathsf{T}}$, 
we can rewrite it  as 
\begin{equation}\label{FTROFEQU}
\partial_t \widehat{u} + L_A(\xi) \widehat{u}  =  -i\xi\cdot\widehat{ (n \nabla c)}e_1,\,\,\mbox{ where }\,\, L_A(\xi) := \begin{pmatrix}
	|\xi|^2 & -A|\xi|^2 \\
	-1 & 1+|\xi|^2
\end{pmatrix}. 
\end{equation}
Note that solving the characteristic equation,
$
\left| L_A(\xi) - \lambda I \right| = (|\xi|^2 - \lambda)(1+|\xi|^2 - \lambda) - A|\xi|^2$,
we have two eigenvalues
\[
\lambda_{\pm}=\lambda_{\pm}(|\xi|):= \frac {1+2|\xi|^2 \pm \sqrt{1+4A|\xi|^2}}{2},
\]
and its complex conjugates, $\overline{\lambda_{\pm}}$, are eigenvalues of $L_A^{*}$, the conjugate transpose of $L_A$. Thus, for given $A>0$, we have $L_A^{*}a_{\pm}=\overline{\lambda_{\pm}}a_{\pm}$ with eigenvectors  
\[
a_{\pm}(\xi):= \bke{\frac{-2}{1\pm\sqrt{1+4A|\xi|^{2}}},1}^{\mathsf{T}}.
\] 
Moreover, it holds that
$
	  \sum_{l=+,-}b_{l}\,\overline{a_{l}}^{\mathsf{T}}
	= I
$
with
\[
b_{\pm}(\xi) := \frac{\pm 1}{\sqrt{1+4A|\xi|^{2}}}
\bke{
-A|\xi|^{2},
\frac{1\pm\sqrt{1+4A|\xi|^{2}}}{2}
}^{\mathsf{T}}
\]
and thus, we have the decomposition formula
\begin{equation*}
  w =I  w= \sum_{l=+,-} b_l\overline{\langle  a_l, w  \rangle} =\sum_{l=+,-} \langle  w,a_l \rangle b_l\quad\mbox{ for all } w\in\mathbb{C}^{2}.
\end{equation*}
Here,  $\langle v,w\rangle$ denotes the inner product of two vectors $v,w \in\mathbb{C}^{2}$ and its value is  $v^{\mathsf{T}} \overline{w}$. 
Now, from   
  \[
 \langle L_{A} w,  a_\pm \rangle = \langle w, L_{A}^{*} a_\pm \rangle= \lambda_{\pm} \langle  w, a_\pm \rangle \quad\mbox{ for all } w\in\mathbb{C}^{2}
\]
and \eqref{FTROFEQU}, 
we have
$
\rd_t \langle\widehat{u}, a_\pm \rangle b_\pm+ \lambda_{\pm} \langle   \widehat{ u}, a_\pm \rangle  b_\pm =-\langle i\xi\cdot \widehat{(n \nabla c)}e_1, a_\pm \rangle b_\pm$,
which entails the Duhamel formula for $ \widehat{u} =\sum_{l=+,-} \langle  \widehat{u},a_l \rangle b_l$, 
\begin{equation}\label{PREdf_u}
	\begin{aligned}
		\langle \widehat{u},a_{\pm} \rangle b_{\pm} = e^{-\lambda_{\pm}t} \langle \widehat{u_0},a_{\pm} \rangle b_{\pm} - \int_0^t e^{-\lambda_{\pm}(t-\tau)} \langle i\xi \cdot\widehat{  (n \nabla c)}e_1, a_{\pm} \rangle b_{\pm} \,\ud \tau.
	\end{aligned}
\end{equation}
\begin{remark}\label{RMK22}
Note that $\lambda_{-}$
	is   non-negative for $A \leq 1$, but is negative at some $\xi$  for $A > 1$. Thus, it is expected that \eqref{eq_crt} is nonlinear stable for $A \leq 1$ and unstable for $A > 1$. This corresponds to the result of the parabolic-elliptic counterpart in \cite{CKKW21}, where the cases $A < 1$ and $A > 1$ are considered. 
\end{remark}
\section{Subcritical Case: Proof of Theorem~\ref{thm1}}\label{SEC11}
Throughout this section, we assume the hypotheses of Theorem~\ref{thm1}. Let $(n,c)$ be the solution given by Lemma~\ref{LEMLWP}. 
\subsection{A Priori Estimates}
We first show that $(n,c)$ or $(n,\nabla c)$ is small 
in $L^{\infty}_{x,t}$. Since $\|f\|_{L^{\infty}(\R^d)}\lesssim \|\widehat{f}\|_{L^{1}(\R^d)}$ for all $f\in \mathcal{S}(\mathbb{R}^{d})$, it is enough to show the smallness of either $(\widehat{n}, \widehat{c})$ or $(\widehat{n}, \xi\widehat{c})$ in $(L^{1}(\R^d))^{2}$ as long as the solution exists. 
To this end, we begin with a priori estimates for $\widehat{n}$ and $\widehat{c}$, established in the following lemma. For the proof, see Appendix~\ref{secA1}.
\begin{lemma}\label{REVLEM31}
Let the same assumptions as in Theorem~\ref{thm1} be satisfied, and let $j=0$ or $1$.    It holds that for $T<T_{\rm max}$ 
\begin{equation}\label{est_NABLAn_L2+}
\int_{\R^{d}} \| |\xi|\widehat{n}\|_{L^2(0,T)} \,\ud \xi \leq  \frac{\| \widehat{n_0} \|_{L^1(\R^{d})} }{1-A} +\frac{A}{1-A}    \||\xi|^{j} \widehat{c_0} \|_{L^1(\R^{d})}+  \frac{\int_{\R^{d}} \| \widehat{n \nabla c} \|_{L^2(0,T)} \,\ud \xi}{1-A},
\end{equation}
\begin{equation}\label{est_NABLAc_L2+}
\int_{\R^{d}} \| |\xi|\widehat{c} \|_{L^2(0,T)} \,\ud \xi \leq  \frac{\| \widehat{n_0}\|_{L^1(\R^{d})} }{1-A} +\frac{  \| |\xi|^{j}\widehat{c_0} \|_{L^1(\R^{d})}}{1-A}+  \frac{\int_{\R^{d}} \| \widehat{n \nabla c} \|_{L^2(0,T)}\,\ud \xi}{1-A},
\end{equation}
 \begin{equation}\label{est_nc_L2+}
	\begin{aligned}
		\int_{\R^{d}} \| \widehat{n \nabla c} \|_{L^2(0,T)} \,\ud \xi  
		&\leq  \| \widehat{n_0} \|_{L^1(\R^{d})} \||\xi|^{j} \widehat{c_0} \|_{L^1(\R^{d})} 
		+  \| \widehat{n_0} \|_{L^1(\R^{d})} \int_{\R^{d}} |\xi|\| \widehat{ n} \|_{L^2(0,T)} \,\ud \xi \\
		&\quad+  \| |\xi|^{j}\widehat{c_0} \|_{L^1(\R^{d})}  \biggr{(}\int_{\R^{d}} \| \widehat{n \nabla c} \|_{L^2(0,T)} \,\ud \xi+A\int_{\R^{d}} \| |\xi|\widehat{c}  \|_{L^2(0,T)} \,\ud \xi \biggr{)} \\
		&\quad+ \biggr{(}\int_{\R^{d}} \| \widehat{n \nabla c} \|_{L^2(0,T)} \,\ud \xi+A\int_{\R^{d}} \||\xi| \widehat{c}  \|_{L^2(0,T)} \,\ud \xi \biggr{)} \int_{\R^{d}} \||\xi| \widehat{n} \|_{L^2(0,T)} \,\ud \xi,
	\end{aligned}
\end{equation}
and 
\begin{equation}\label{est_L1+}
	\begin{aligned}
		\sup_{t\le T} \| \widehat{n}(\cdot,t) \|_{L^1(\R^{d})} \leq  \| \widehat{n_0} \|_{L^1(\R^{d})} + \int_{\R^{d}} \| \widehat{n \nabla c} \|_{L^2(0,T)} \,\ud \xi+A\int_{\R^{d}} \|\xi|| \widehat{ c}  \|_{L^2(0,T)} \,\ud \xi,
	\end{aligned}
\end{equation}
\begin{equation}\label{est_L1++}
	\begin{aligned}
 &\sup_{t \le T} \| \widehat{c}(\cdot,t) \|_{L^1(\R^{d})} \le \| \widehat{c_0} \|_{L^1(\R^{d})} 
		+ \sup_{t\le T} \| \widehat{n}(\cdot,t) \|_{L^1(\R^{d})},\\
 &\sup_{t \le T} \| |\xi|\widehat{c}(\cdot,t) \|_{L^1(\R^{d})} \le \||\xi| \widehat{c_0} \|_{L^1(\R^{d})} 
		+ \int_{\R^{d}} \||\xi| \widehat{n} \|_{L^2(0,T)} \,\ud \xi.
	\end{aligned}
\end{equation}
\end{lemma}
Next, we show that either $(\widehat{n}, \widehat{c})$ or $(\widehat{n}, \xi\widehat{c})$ remains small in $(L^{1}(\R^d))^{2}$.
\begin{proposition}\label{NCSMALLLEM}
Let the same assumptions as in Theorem~\ref{thm1} be satisfied, and let $j=0$ or $1$.  
There exists $\delta_{j}>0$ such that if 
\begin{equation}\label{NCSMALLLEMASS}
	\begin{aligned}
		\|\widehat{n_0}\|_{L^{1}(\R^d)} +  \||\xi|^{j}\widehat{c_0}\|_{L^{1}(\R^d)} \leq  \delta_j,
	\end{aligned}
\end{equation}
then for some $C=C(A)>0$, 
\[
		\sup_{t<T_{\rm max}}(\|\widehat{n}(\cdot,t)\|_{L^{1}(\R^d)} +  \||\xi|^{j}\widehat{c}(\cdot,t)\|_{L^{1}(\R^d)}) \leq C \delta_j.
\]
\end{proposition}
\begin{proof}
Denote 
\[
X(t):= \int_{\R^{d}} \| \widehat{n \nabla c} (\xi,\cdot)\|_{L^2(0,t)} \,\ud \xi,\qquad t<T_{\rm max}.
\] 
We first show the continuity of $X$ in $[0,T_{\rm max})$.
Let $s>\frac{d}{2}$ be integer and $j=0$ or $1$. For $0\le t<T<T_{\rm max}$, we use $X'\ge0$, H\"older's inequality,   $(a_1-a_2)^{2}\le a_1^{2}-a_2^{2}$ for all $a_{1}\ge a_{2}\ge0$, $  (1+|\xi|^2)^{-s}\in L^{1}(\R^{d})$, and Fubini's theorem to obtain
\[
\begin{aligned}
X(T)-X(t)
& = \int_{\R^{d}}\bket{ \| \widehat{n \nabla c}(\xi,\cdot)(1+|\xi|^2)^{\frac{s}{2}} \|_{L^2(0,T)}- \| \widehat{n \nabla c}(\xi,\cdot)(1+|\xi|^2)^{\frac{s}{2}} \|_{L^2(0,t)}} (1+|\xi|^2)^{-\frac{s}{2}} \,\ud \xi 
\\& \le  \bke{\int_{\R^{d}}\bket{ \| \widehat{n \nabla c}(\xi,\cdot)(1+|\xi|^2)^{\frac{s}{2}} \|_{L^2(0,T)}- \| \widehat{n \nabla c}(\xi,\cdot)(1+|\xi|^2)^{\frac{s}{2}} \|_{L^2(0,t)}}^{2}\,\ud \xi}^{\frac{1}{2}}
\\&\qquad\cdot \bke{\int_{\R^{d}}  (1+|\xi|^2)^{-s}\,\ud \xi }^{\frac{1}{2}}
\\&\lesssim \bke{\int_{t}^{T} \int_{\R^{d}} |\widehat{n \nabla c}(\xi,\tau)|^{2}(1+|\xi|^2)^{s}  \,\ud \xi  \,\ud \tau }^{\frac{1}{2}}
\\& \lesssim (T-t)^{\frac{1}{2}}\| n\nabla c\|_{L^{\infty}(t,T;H^{s}(\R^{d}))}.
\end{aligned}
\]
By H\"older's inequality, the embedding inequality and Lemma~\ref{LEMLWP}, we can see that
\[
 \| n\nabla c\|_{L^{\infty}(t,T;H^{s}(\R^{d}))}\lesssim \| n\|_{L^{\infty}(t,T;H^{s}(\R^{d}))}\| c\|_{L^{\infty}(t,T;H^{s+1}(\R^{d}))}<\infty
 \] 
 and thus, $X$ is continuous in $[0,T_{\rm max})$.
In particular,  
\begin{equation}\label{XINISMALL}
\lim_{t\rightarrow 0}X(t)=0.
\end{equation}
We also note that substituting \eqref{est_NABLAn_L2+}--\eqref{est_NABLAc_L2+} into \eqref{est_nc_L2+} gives
for some $C=C(A)>0$
\begin{equation}\label{XESTIMATE}
	\begin{aligned}
		X(t)& \leq  \| \widehat{n_0} \|_{L^1}   \||\xi|^{j} \widehat{c_0} \|_{L^1} + \frac{\| \widehat{n_0} \|_{L^1}}{1-A} \left(   \| \widehat{n_0} \|_{L^1}  +A  \| |\xi|^{j}\widehat{c_0} \|_{L^1}+  X(t) \right) 
		\\ &\qquad+   \frac{ \||\xi|^{j} \widehat{c_0} \|_{L^1} }{1-A} ( A\| \widehat{n_0} \|_{L^1} +A \||\xi|^{j} \widehat{c_0} \|_{L^1}+  X(t) )
		\\ &\qquad
		+ \frac{1}{(1-A)^{2}}( \| \widehat{n_0} \|_{L^1}  + A  \||\xi|^{j} \widehat{c_0} \|_{L^1}+   X(t) ) 
		(  A\| \widehat{n_0} \|_{L^1} + A \||\xi|^{j} \widehat{c_0} \|_{L^1}+ X(t))
		\\&\leq  C(\| \widehat{n_0} \|_{L^1}^{2}+  \||\xi|^{j} \widehat{c_0} \|_{L^1}) + C(\| \widehat{n_0} \|_{L^1} + \||\xi|^{j} \widehat{c_0} \|_{L^1}) X(t)+\biggr{(}\frac{X(t)}{1-A}\biggr{)}^{2}.
	\end{aligned}
\end{equation}
Then, due to   \eqref{XINISMALL}--\eqref{XESTIMATE} and the continuity of $X$, there exists  $ \delta_{j}>0$ such that under \eqref{NCSMALLLEMASS}, 
\[
\sup_{t<T_{\rm max}}X(t)\le C \delta_{j}
\] 
for some $C=C(A)>0$.
Plugging it into \eqref{est_NABLAn_L2+} and \eqref{est_NABLAc_L2+} shows that there exists  $C=C(A)>0$ satisfying 
\[
\sup_{t<T_{\rm max}}\biggr{(}\int_{\R^{d}} \| |\xi|\widehat{n}\|_{L^2(0,t)} \,\ud \xi+\int_{\R^{d}} \| |\xi|\widehat{c}\|_{L^2(0,t)} \,\ud \xi  \biggr{)}\le C \delta_{j}.
\]
Hence, by \eqref{est_L1+}--\eqref{est_L1++}, we can deduce the desired result.
\end{proof}
Next, to prove Theorem~\ref{thm1} we prepare some  $L^{2}$-type a priori estimates.
\begin{lemma}
Let the same assumptions as in Theorem~\ref{thm1} be satisfied. For integer $m\ge1$, there exists $C>0$ such that
\begin{equation}\label{eng_each}  
	\begin{aligned}
		\frac {1}{2}&\frac {\ud}{\ud t} ( \|D^{m-1} n\|_{L^{2}(\R^{d})}^2+\|D^{m} c\|_{L^{2}(\R^{d})}^2) \\&+ \Bigr{(}\frac{1-A}{2}-C(\| n \|_{L^{\infty}(\R^{d})}+\| c \|_{L^{\infty}(\R^{d})}) \Bigr{)} (\|D^{m} n\|_{L^{2}(\R^{d})}^2+\|D^{m} c\|_{L^{2}(\R^{d})}^2)   
		+\|D^{m+1} c\|_{L^{2}(\R^{d})}^2  \leq 0,
	\end{aligned}
\end{equation}
\begin{equation}\label{eng_n} 
	\begin{aligned}
\frac {1}{2}&\frac {\ud}{\ud t} \left( \| n \|_{H^m(\R^{d})}^2 + \| D c \|_{H^{m-1}(\R^{d})}^2 \right) \\&
+ \Bigr{(}\frac{1-A}{2} - C(\| n \|_{L^{\infty}(\R^{d})} + \| c \|_{L^{\infty}(\R^{d})}) \Bigr{)} (\| Dn \|_{H^m(\R^{d})}^2 + \| Dc \|_{H^{m }(\R^{d})}^2 )  \leq 0,
	\end{aligned}
\end{equation}
and
\begin{equation}\label{eng_nn} 
	\begin{aligned}
\frac {1}{2}&\frac {\ud}{\ud t} \left( \| n \|_{H^m(\R^{d})}^2 + \| D c \|_{H^{m-1}(\R^{d})}^2 \right) \\&
+ \Bigr{(}\frac{1-A}{2} - C(\| n \|_{L^{\infty}(\R^{d})} + \| \nabla c \|_{L^{\infty}(\R^{d})})\Bigr{)} (\| Dn \|_{H^m(\R^{d})}^2 + \| Dc \|_{H^{m}(\R^{d})}^2 )  \leq 0.
	\end{aligned}
\end{equation}
\end{lemma}
\begin{proof}
Taking the $L^{2}$ scalar inner product of $\eqref{eq_crt}_{1}$ with $n$ yields
\begin{equation*}
	\begin{aligned}
		\frac {1}{2}\frac {\ud}{\ud t} \|n\|_{L^{2}}^{2} +\|D n\|_{L^{2}}^{2} = \int_{\R^{d}} n \nabla n \cdot \nabla c \,\ud x+ A\int_{\R^{d}} \nabla n \cdot \nabla c \,\ud x,
	\end{aligned}
\end{equation*} 
which gives, after applying H\"older's inequality and Young's inequality, that
\begin{equation}\label{LEM32_0}
\frac {1}{2}\frac {\ud}{\ud t}  \|n\|_{L^{2}}^{2} + \|D n\|_{L^{2}}^{2} \le  (\frac{1}{2}\|n\|_{L^{\infty}}+\frac{A}{2})(\|D n \|_{L^{2}}^{2}+\|D c \|_{L^{2}}^{2}).
\end{equation} 
Next,  let $k\ge1$ be integer. Using $\eqref{eq_crt}_{1}$ and the H\"older inequality, we compute 
 \begin{equation}\label{LEM32_-1}
\begin{aligned}
		\frac {1}{2}\frac {\ud}{\ud t}  \|D^k n\|_{L^{2}}^{2} + \|D^{k+1} n\|_{L^{2}}^2 
	&\le \int |D^{k+1} n||D^{k-1} [\nabla \cdot (n \nabla c)]|  \,\ud x + A\int_{\R^{d}} |D^{k+1}n|  |D^{k+1}c| \,\ud x\\
		&\le \| D^{k+1} n \|_{L^2}  \| D^{k}(n \nabla c)  \|_{L^2} +A  \| D^{k+1} n \|_{L^2} \| D^{k+1} c \|_{L^2}. 		
\end{aligned}
\end{equation}
If we further estimate the right-hand side using Lemma~\ref{KATOTYPELEM} with $(\alpha,\beta)=(k,1)$ and Young's inequality, we can find $C>0$ such that
 \begin{equation}\label{LEM32_1}
	\begin{aligned}
\frac {1}{2}\frac {\ud}{\ud t}  &\|D^k n\|_{L^{2}}^2  + \|D^{k+1} n\|_{L^{2}}^2  
		\\&\le C\| D^{k+1} n \|_{L^2}(  \| D^{k+1}  n  \|_{L^2} \|  c  \|_{L^\infty} + \|   n  \|_{L^\infty} \| D^{k+1}c  \|_{L^2} )+ A \| D^{k+1} n \|_{L^2} \| D^{k+1} c \|_{L^2} \\
		&\leq \Bigr{(}2C (\| n \|_{L^{\infty}} + \| c \|_{L^{\infty}})+\frac{A}{2}\Bigr{)} (\| D^{k+1} n \|_{L^2}^2 + \| D^{k+1} c \|_{L^2}^2).
\end{aligned}
\end{equation}
On the other hand,  we use Lemma~\ref{KATOTYPELEM} with $(\alpha,\beta)=(k,0)$ and Young's inequality to compute the right-hand side of \eqref{LEM32_-1}  as
 \begin{equation}\label{LEM32_11}
	\begin{aligned}
		\frac {1}{2}&\frac {\ud}{\ud t}\|D^k n\|_{L^{2}}^2  + \|D^{k+1} n\|_{L^{2}}^2  
		\\
		&\leq C\| D^{k+1} n \|_{L^2}(  \| D^{k }  n  \|_{L^2} \| \nabla c  \|_{L^\infty} + \|   n  \|_{L^\infty} \| D^{k+1}c  \|_{L^2} )+ A\| D^{k+1} n \|_{L^2}  \| D^{k+1} c \|_{L^2} \\
		&\leq \Bigr{(}C(\| n  \|_{L^\infty} + \| \nabla c  \|_{L^\infty}) +\frac{A}{2}\Bigr{)}  (\| D^{k+1} n \|_{L^2}^2 + \| D^{k+1} c \|_{L^2}^2)+C\| \nabla c  \|_{L^\infty}\| D^{k} n \|_{L^2}^2.
\end{aligned}
\end{equation}
Similarly, for any integer $k\ge0$, appying Young's inequality to the $c$-equation $\eqref{eq_crt}_{2}$  yields that
\begin{equation}\label{LEM32_2}
	\begin{aligned}
		\frac {1}{2}\frac {\ud}{\ud t} \|D^k c\|_{L^{2}}^2  + \|D^k c\|_{L^{2}}^2 + \|D^{k+1} c\|_{L^{2}}^2  &\le \int_{\R^{d}} |D^k n| |D^k c| \,\ud x 
		\\&\le \frac {1}{2} (\| D^k n \|_{L^2}^2 + \| D^k c \|_{L^2}^2).
	\end{aligned}
\end{equation}
Now, using the above estimates, we show \eqref{eng_each}--\eqref{eng_nn}. 
Adding  \eqref{LEM32_0} and \eqref{LEM32_2} with $k=1$ shows  \eqref{eng_each} for $m=1$. To obtain \eqref{eng_each} for $m\ge2$, 
we add \eqref{LEM32_1} with $k=m-1$ and \eqref{LEM32_2} with $k=m$. 
If we add  \eqref{LEM32_0} and \eqref{LEM32_1} for $k=1,\cdots, m$ as well as \eqref{LEM32_2} for $k=1,\cdots, m$, then 
 \eqref{eng_n} is obtained. Adding \eqref{LEM32_0}, \eqref{LEM32_11} for $k=1,\cdots,m$, and \eqref{LEM32_2} for $k=1,\cdots,m$ gives \eqref{eng_nn}. 
\end{proof}
\subsection{Main Proofs}
Based on the preceding a priori estimates, we now prove Theorem~\ref{thm1}.
\subsubsection{Global Existence.}
We show
$T_{\rm max}=\infty$.
In the case of $j=0$,   \eqref{eng_n} with $m=s>\frac{d}{2}$ yields that for some $C>0$
\begin{equation}\label{BDDTHM1}
\sup_{t<T_{\rm max}}(\| n(\cdot,t)\|_{H^s }+\| D c(\cdot,t) \|_{H^{s-1} })\le C
\end{equation}
if $\| n \|_{L^{\infty}_{x,t}}+\| c \|_{L^{\infty}_{x,t}}$ is small.
This smallness is guaranteed by Proposition~\ref{NCSMALLLEM} under \eqref{cond1} with sufficiently small $ \delta_0$  since $\|f\|_{L^{\infty} }\lesssim \|\widehat{f}\|_{L^{1} }$ for all $f\in \mathcal{S}(\R^{d})$. Then, by  \eqref{BUCRITREF}, $T_{\rm max}=\infty$. 
Similarly, when $j=1$, we use \eqref{eng_nn} with $m=s$ to show \eqref{BDDTHM1} if $\| n \|_{L^{\infty}_{x,t}}+\|\nabla  c \|_{L^{\infty}_{x,t}}$ is  small. 
Again, this smallness is guaranteed under \eqref{cond1}, by Proposition~\ref{NCSMALLLEM}, with sufficiently small $\delta_1$. Thus, $T_{\rm max}=\infty$. Moreover, Proposition~\ref{NCSMALLLEM} gives
\[
\sup_{t>0}(\|\widehat{n}(\cdot,t)\|_{L^{1}(\R^d)} +  \||\xi|^{j}\widehat{c}(\cdot,t)\|_{L^{1}(\R^d)}) \leq C \delta_j
\]
for some $C=C(A)>0$.

\subsubsection{Smallness of $n$ and  $c$ in $L^{\infty}$.}
 
 We show   for any $\varepsilon>0$ there exists $T_{\varepsilon}<\infty$ such that
\begin{equation}\label{STEP2GOAL}
\sup_{t\ge T_{\varepsilon}}(\|n(\cdot,t)\|_{L^{\infty}}+\|c(\cdot,t)\|_{L^{\infty}})\le\varepsilon.
\end{equation}
To this end, we first obtain the temporal decay of $\|\nabla n\|_{L^{2}}$.
 Multiplying both sides of $\eqref{eq_crt}_{1}$ by $-\Delta n$ and integrating over $\R^{d}$, we have
 \[
 \frac{1}{2}\frac{d}{dt}\|D n\|_{L^{2}}^{2}+\|D^2 n\|_{L^{2}}^{2}=\int_{\R^{d}}\Delta n  (\nabla n\cdot \nabla c)\,\ud x+\int_{\R^{d}}n \Delta n  \Delta c\,\ud x+A\int_{\R^{d}}\Delta n \Delta c \,\ud x.
 \]
To estimate the first term on the right-hand side, we use integration by parts as
 \[
 \begin{aligned}
 \int_{\R^{d}}\Delta n  (\nabla n\cdot \nabla c)  \,\ud x&=\sum_{l=1}^{d}\sum_{k=1}^{d} \int_{\R^{d}} \partial_{l}^{2}n \partial_{k}n \partial_{k}c  \,\ud x
 \\&=-\sum_{l=1}^{d}\sum_{k=1}^{d} \int_{\R^{d}} \partial_{l}n \partial_{l}\partial_{k}n \partial_{k}c  \,\ud x-\sum_{l=1}^{d}\sum_{k=1}^{d} \int_{\R^{d}} \partial_{l}n \partial_{k}n \partial_{l}\partial_{k}c  \,\ud x
    \\&= \frac{1}{2}\sum_{l=1}^{d}\sum_{k=1}^{d} \int_{\R^{d}} (\partial_{l}n)^{2}  \partial_{k}^{2}c  \,\ud x-\sum_{l=1}^{d}\sum_{k=1}^{d} \int_{\R^{d}} \partial_{l}n \partial_{k}n \partial_{l}\partial_{k}c  \,\ud x.
  \end{aligned}
 \]
If we apply the H\"older inequality and the Gagliardo–Nirenberg inequality of type,
\[
\|D f\|_{L^{4}}\lesssim\|  f\|_{L^{\infty}}^\frac{1}{2}\|D^2 f\|_{L^{2}}^\frac{1}{2}\,\,\mbox{ for all }\,\,f\in (L^{\infty}\cap H^{2})(\R^{d}),
\]
then it follows that for some $C>0$
\[
 \int_{\R^{d}}\Delta n  (\nabla n\cdot \nabla c)  \,\ud x\le C\|n\|_{L^{\infty}}\|D^2 n\|_{L^{2}}\|D^2 c\|_{L^{2}}.
\]
Combining the above estimates, using Young's inequality, and adding 
 \eqref{LEM32_2} with $k=2$ shows that
\begin{equation}\label{YTINEQ}
 \frac{1}{2}\frac{d}{dt}(\|D n\|_{L^{2}}^{2}+ \| D^2 c\|_{L^{2}}^2)+ \frac{1-A}{4} (\|D^2 n\|_{L^{2}}^{2}+\|D^{3} c\|_{L^{2}}^{2}) \le0
\end{equation}
if   $\| n \|_{L^{\infty}}$ is sufficiently small. To guarantee this smallness from Proposition~\ref{NCSMALLLEM}, we pick sufficiently small  $\delta_{j} $  of \eqref{cond1}.  Since $\|n\|_{L^{2}}$ and $\|D c\|_{L^{2}}$ are uniformly bounded due to \eqref{BDDTHM1}, we can see from \eqref{YTINEQ}, after employing 
\[
\|D f\|_{L^{2}}^{2}\lesssim\|f\|_{L^{2}}\|D^{2}f\|_{L^{2}}\,\,\mbox{ for all }\,\, f\in H^{2}(\R^{d}),
\]
 that 
$
y(t):=(\|D n(\cdot,t)\|_{L^{2}}^{2}+ \| D^2 c(\cdot,t)\|_{L^{2}}^2)
$
satisfies 
$
\frac{d}{dt}y(t)+Cy^{2}(t)\le 0$ for some $C>0$.
 Solving this ordinary differential inequality (ODI) yields   there exists $C>0$ satisfying
\[
\|D n(\cdot,t)\|_{L^{2}}^{2}\le C(1+t)^{-1}\qquad\mbox{for all}\quad t>0.
\]
Thus, along with the uniform bound \eqref{BDDTHM1}, we can conclude \eqref{STEP2GOAL} since
\[
 \begin{aligned}
\|n(\cdot,t)\|_{L^{\infty}}
&\lesssim  \|n(\cdot,t)\|_{L^{2}}^{\frac{1}{2}\cdot\frac{2s-d}{2s-1}} \|D n(\cdot,t)\|_{L^{2}}^{\frac{1}{2}\cdot\frac{2s-d}{2s-1}}\|D^{s} n(\cdot,t)\|_{L^{2}}^{\frac{d-1}{2s-1}}
\lesssim  (1+t)^{-\frac{1}{4}\cdot\frac{2s-d}{2s-1}},
 \end{aligned}
\] 
 and
\[
\begin{aligned}
\|c(\cdot,t)\|_{L^{\infty}}
&\lesssim \|e^{t(\Delta-1)}c_0\|_{L^{\infty}}+\int_{0}^{\frac{t}{2}}\|e^{(t-\tau)(\Delta-1)}n(\tau)\|_{L^{\infty}}\,\ud\tau+\int_{\frac{t}{2}}^{t}\|e^{(t-\tau)(\Delta-1)}n(\tau)\|_{L^{\infty}}\,\ud\tau
\\&\lesssim  e^{-t}\|c_0\|_{L^{\infty}}+\Bigr{(}\sup_{t>0}\| n(t)\|_{L^{\infty} }\Bigr{)}\int_{0}^{\frac{t}{2}}e^{-(t-\tau)} \,\ud\tau+ \int_{\frac{t}{2}}^{t}  e^{-(t-\tau)}\|n(\tau)\|_{L^{\infty}}\,\ud\tau
\\&\lesssim  e^{-t}\|c_0\|_{H^{s}}+\Bigr{(}\sup_{t>0}\| n(t)\|_{H^{s} }\Bigr{)}e^{-\frac{t}{2}}\frac{t}{2}+ \Bigr{(}1+\frac{t}{2}\Bigr{)}^{-\frac{1}{4}\cdot\frac{2s-d}{2s-1}}\int_{\frac{t}{2}}^{t}  e^{-(t-\tau)}\,\ud\tau
\\&\lesssim  e^{-t}+ e^{-\frac{t}{2}}\frac{t}{2}+ \Bigr{(}1+\frac{t}{2}\Bigr{)}^{-\frac{1}{4}\cdot\frac{2s-d}{2s-1}}.
\end{aligned}
\]
\subsubsection{Temporal Decay Rates of $n$ and $c$.}\label{SEC323}
 Note that the solution $(n,c)$ given
  is global and smooth. In particular,  for any  $m\ge1$ and $t>1$, by the bootstrapping argument,
\begin{equation}\label{FEB282}
 (n,c)\in (L^{\infty}([1,t];H^{m}(\R^{d}))^{2}.
 \end{equation}
 We claim that for any given integer $m\ge1$, 
\begin{equation}\label{DECAYNC}
 \| D^{m} n (\cdot,t)\|_{L^2} + \| D^{m}c (\cdot,t)\|_{L^2} \lesssim  t^{-\frac{m}{2}}\,\,\mbox{ for all }\,\, t>1.
\end{equation}
 By \eqref{eng_each} and \eqref{STEP2GOAL}, there exists a sufficiently large $T$ satisfying
\[
\frac {1}{2}\frac {\ud}{\ud t} ( \|D^{m-1} n\|_{L^{2} }^2+\|D^{m} c\|_{L^{2} }^2)  + \frac{1-A}{4} (\|D^{m} n\|_{L^{2} }^2+\|D^{m+1} c\|_{L^{2} }^2)  \leq 0\qquad\mbox{for all}\quad t\ge T.
\]
Since $\|n\|_{L^{2}}$ and $\|D c\|_{L^{2}}$ are uniformly bounded by \eqref{BDDTHM1}, 
 we obtain, after using
\begin{equation}\label{SEP121}
\| D^{m-1} f \|_{L^2}\lesssim \| f \|_{L^2}^{\frac{1}{m}}\| D^{m} f \|_{L^2}^{1-\frac{1}{m}} \quad\mbox{ for all }\quad f\in H^{m}(\mathbb{R}^{d}),
\end{equation} 
 that $y(t) := ( \|D^{m-1} n(\cdot,t)\|_{L^{2}}^2+\|D^{m} c(\cdot,t)\|_{L^{2}}^2)$
satisfies   with some $C>0$,
\[
y'(t)+Cy(t)^{\frac{m}{m-1}}\le0\quad\mbox{ for all }\quad t\ge T.
\]
Solving this ODI, due to \eqref{FEB282} we can find  $C>0$ such that
\begin{equation}\label{SEP122}
  \| D^{m-1} n (\cdot,t)\|_{L^2} + \| D^{m}c (\cdot,t)\|_{L^2} \le Ct^{-\frac{ m-1}{2}}\qquad\mbox{for all}\quad t>1.
\end{equation}
The above decay rate of $n$ improves the decay rate of $c$ since for all $t>2$ and integer $l\ge1$, 
\[
\begin{aligned}
\|D^{l}c(t)\|_{L^{2}}
&\lesssim \|D^{l}e^{t(\Delta-1)}c_0\|_{L^{2}}+\int_{0}^{\frac{t}{2}}\|D^{l}e^{(t-\tau)(\Delta-1)}n(\tau)\|_{L^{2}}\,\ud\tau+\int_{\frac{t}{2}}^{t}\|D^{l}e^{(t-\tau)(\Delta-1)}n(\tau)\|_{L^{2}}\,\ud\tau
\\&\lesssim  t^{-\frac{l}{2}}e^{-t}\|c_0\|_{L^{2}}+\Bigr{(}\sup_{t>0}\| n(t)\|_{L^{2} }\Bigr{)}\int_{0}^{\frac{t}{2}}(t-\tau)^{-\frac{l}{2}}e^{-(t-\tau)} \,\ud\tau+ \int_{\frac{t}{2}}^{t}  e^{-(t-\tau)}\|D^{l }n(\tau)\|_{L^{2}}\,\ud\tau
\\&\lesssim  t^{-\frac{l}{2}}e^{-t}\|c_0\|_{L^{2}}+\int_{0}^{\frac{t}{2}}e^{-(t-\tau)} \,\ud\tau\Bigr{(}\frac{t}{2}\Bigr{)}^{-\frac{l}{2}}+ \int_{\frac{t}{2}}^{t}  e^{-(t-\tau)}\,\ud\tau\Bigr{(}\frac{t}{2}\Bigr{)}^{-\frac{l}{2}} 
\\&\lesssim  t^{-\frac{l}{2}}.
\end{aligned}
\]
Hence, \eqref{DECAYNC} is obtained.
\subsubsection{Temporal Decay Rates of $n-c$.}

Moreover, using 
$\rd_t(n-c) -\Delta(n-c) + (n-c) = -\nabla \cdot ( n  \nabla c)-A\Delta c$,  
we can derive decay rates for $n-c$ that are faster than those for $n$ or $c$, \eqref{DECAYNC}. Indeed, for all $t>2$ and integer $m\ge1$ we have    
 \[
 \begin{aligned}
\|&D^{m}(n-c)(t)\|_{L^{2}}
\\&\lesssim \|D^{m}e^{t(\Delta-1)} (n_{0}-c_{0})\|_{L^{2}}+\int_{0}^{\frac{t}{2}} \|D^{m}e^{(t-\tau)(\Delta -1)} (\nabla \cdot (n\nabla c)+A\Delta c)(\tau)\|_{L^{2}} \,\ud \tau 
\\&\quad+\int_{\frac{t}{2}}^{t} \|D^{m}e^{(t-\tau)(\Delta -1)} (\nabla \cdot (n\nabla c)+A\Delta c)(\tau)\|_{L^{2}} \,\ud \tau
\\&\lesssim  t^{-\frac{m}{2}}e^{-t}\|n_{0}-c_{0}\|_{L^{2}}
\\&\quad + \int_{0}^{\frac{t}{2}}  e^{-(t-\tau)} (t-\tau)^{-\frac{m+1}{2}}( \| n(\tau)\|_{L^{\infty}}\|D c (\tau)\|_{L^{2}}+A\|D c(\tau)\|_{L^{2}})d\tau
\\&\quad +\int_{\frac{t}{2}}^{t}e^{-(t-\tau)}( \|  D^{m+1}  n(\tau)\|_{L^{2}}\|D c(\tau)\|_{L^{\infty}}+   \| n(\tau)\|_{L^{\infty}}\|D^{m+2}  c(\tau)\|_{L^{2}}  +A\|D^{m+2}c (\tau)\|_{L^{2}})\,\ud \tau,
 \end{aligned}
 \]
 where we used Lemma~\ref{KATOTYPELEM} with $(\alpha,\beta)=(m+1,0)$ in the last line.
Using \eqref{BDDTHM1}, \eqref{DECAYNC} and
$
\|  f \|_{L^{\infty}}\lesssim \| f \|_{L^2}^{\frac{1}{2}}\| D^{d} f \|_{L^2}^{\frac{1}{2}}$ for all   $f\in H^{d}(\mathbb{R}^{d})$,
 we can further compute the second and third terms on the right hand side as 
  \[
 \begin{aligned}
 \int_{0}^{\frac{t}{2}}  &e^{-(t-\tau)} (t-\tau)^{-\frac{m+1}{2}}( \| n(\tau)\|_{L^{\infty}}\|D c (\tau)\|_{L^{2}}+A\|D c(\tau)\|_{L^{2}})d\tau
\\&\lesssim    \Bigr{(}\sup_{t>0}( \| n(t)\|_{H^{s}}\| Dc(t)\|_{L^{2}}+A  \| Dc(t)\|_{L^{2}}) \Bigr{)}\int_{0}^{\frac{t}{2}}  e^{-(t-\tau)} (t-\tau)^{-\frac{m+1}{2}} d\tau 
\\&\lesssim (1+A)\Bigr{(}\frac{t}{2} \Bigr{)}^{-\frac{m+1}{2}+1}e^{-\frac{t}{2}},
 \end{aligned}
\]
\[
	\begin{aligned}
\int_{\frac{t}{2}}^{t}&e^{-(t-\tau)}( \|  D^{m+1}  n(\tau)\|_{L^{2}}\|D c(\tau)\|_{L^{\infty}}+   \| n(\tau)\|_{L^{\infty}}\|D^{m+2}  c(\tau)\|_{L^{2}}  +A\|D^{m+2}c (\tau)\|_{L^{2}})\,\ud \tau,
\\&\lesssim \int_{\frac{t}{2}}^{t}e^{-(t-\tau)}( \|  D^{m+1}  n(\tau)\|_{L^{2}}\|D c(\tau)\|_{L^{2}}^{ \frac{1}{2}}\|D^{d+1} c(\tau)\|_{L^{2}}^{\frac{1}{2}}
\\&\qquad+   \| n(\tau)\|_{L^{2}}^{ \frac{1}{2}} \| D^{d}n(\tau)\|_{L^{2}}^{\frac{1}{2}}\|D^{m+2}  c(\tau)\|_{L^{2}}  +A\|D^{m+2}c (\tau)\|_{L^{2}}) \,\ud \tau 
\\&\lesssim \bket{\Bigr{(} \frac{t}{2}\Bigr{)}^{-\frac{m}{2}-\frac{d}{4}-1}+A\Bigr{(} \frac{t}{2}\Bigr{)}^{-\frac{m}{2}-1}}\int_{\frac{t}{2}}^{t} e^{-(t-\tau)}d\tau
\\&\lesssim t^{-\frac{m}{2}-\frac{d}{4}-1}+At^{-\frac{m}{2}-1}.
	\end{aligned}
\]
Combining the above estimates and \eqref{FEB282} gives
\[
\|D^{m}(n-c)(t)\|_{L^{2}}
\lesssim t^{-\frac{m}{2}-\frac{d}{4}-1}+At^{-\frac{m}{2}-1}\quad\mbox{ for all }t>1.
\]
\qed
\section{Critical Case: Proof of Theorem~\ref{thm_crt}}\label{SEC12}
In this section, we assume the hypotheses of Theorem~\ref{thm_crt}. Let $(n,c)$ be the solution given by Lemma~\ref{LEMLWP}. 
\subsection{A Priori Estimates}
We first prepare the $L^{2}$-type a priori estimates.
\begin{lemma}\label{L2ESTLEM}
Let the same assumptions as in Theorem~\ref{thm_crt} be satisfied. For any integer $m\ge1$, there exists $C>0$ such that
\begin{equation}\label{A1HM}
	\begin{aligned}
		\frac {1}{2}\frac {\ud}{\ud t} ( \| n\|_{H^{m}(\R^{d})}^2+&\| \nabla c\|_{H^{m}(\R^{d})}^2  ) + \frac{1}{8}\|D^{2} n\|_{H^{m-1}(\R^{d})}^2  +\frac{1}{8}    \|D^{2} c\|_{H^{m}(\R^{d})}^2\\   &\leq  C(\| n\|_{L^{\infty}(\R^{d})}^{2}+\| \nabla c\|_{L^{\infty}(\R^{d})}^{2})( \| n\|_{H^{m}(\R^{d})}^2+\| \nabla c\|_{H^{m}(\R^{d})}^2  ).
	\end{aligned}
\end{equation}
\end{lemma}
\begin{proof}
From the equation, after using Young's inequality and integration by parts, we have 
\[
	\begin{aligned}
		\frac {1}{2}\frac {\ud}{\ud t} \int_{\R^{d}} |n|^2 \,\ud x + \int_{\R^{d}} |\nabla n|^2 \,\ud x  &= \int_{\R^{d}}  \nabla n \cdot \nabla c \,\ud x+ \int_{\R^{d}} n \nabla c \cdot \nabla n \,\ud x\\
		&\le \frac {1}{2} (\| \nabla  n \|_{L^2}^2 +\| \nabla  c \|_{L^2}^2)-  \frac{1}{2}\int_{\R^{d}}   n^{2}  \Delta c \,\ud x,
			\end{aligned}
\]
and
\[
	\begin{aligned}
		\frac {1}{2}\frac {\ud}{\ud t} \int_{\R^{d}} |\nabla  c|^2 \,\ud x + \int_{\R^{d}} (|\nabla  c|^2 + |D^{2} c|^2) \,\ud x &= \int_{\R^{d}} \nabla n \cdot \nabla  c \,\ud x \\
	 	&\leq \frac {1}{2} (\| \nabla  n \|_{L^2}^2 +   \| \nabla  c \|_{L^2}^2).
	\end{aligned}
\]
Thus,  we have
\begin{equation}\label{O0}
\frac {1}{2}\frac {\ud}{\ud t} (\|n\|_{L^{2}}^2+\|\nabla c\|_{L^{2}}^2)+\|D^{2} c\|_{L^{2}}^2 \le  \frac{1}{2}\int_{\R^{d}}   n^{2} | \Delta c|\,\ud x.
\end{equation}
Similarly, for any integer $k>0$ we can see that
\begin{equation*}
	\begin{aligned}
		\frac {1}{2}\frac {\ud}{\ud t} \int_{\R^{d}} |D^k n|^2 \,\ud x + \int_{\R^{d}} |D^{k+1} n|^2 \,\ud x &\le \int_{\R^{d}}       | D^{k+1} n |  | D^{k+1} c | \,\ud x+ \int_{\R^{d}} |D^{k} (n \nabla c)||  D^{k} \nabla n |\,\ud x\\
		&\le   \frac {1}{4} \| D^{k+1} n \|_{L^2}^2 + \| D^{k+1} c \|_{L^2}^2+\| D^{k}  (n \nabla c) \|_{L^2} \| D^{k+1} n \|_{L^2},
	\end{aligned}
\end{equation*}
\begin{equation*}
	\begin{aligned}
		\frac {1}{2}\frac {\ud}{\ud t} \int_{\R^{d}} |D^{k+1} c|^2 \,\ud x + \int_{\R^{d}} (|D^{k+1} c|^2 + |D^{k+2} c|^2) \,\ud x &\le  \int_{\R^{d}} | D^{k+1} n||D^{k+1} c| \,\ud x\\ & \leq \frac {1}{2} (\| D^{k+1} n \|_{L^2}^2 + \| D^{k+1} c \|_{L^2}^2),
	\end{aligned}
\end{equation*}
and thus,
\begin{equation}\label{Ok}
	\begin{aligned}
		\frac {1}{2}\frac {\ud}{\ud t}( \|D^k n\|_{L^{2}}^2+\|D^{k+1} c\|_{L^{2}}^2) +   \frac{1}{4}\|D^{k+1} n\|_{L^{2}}^2  &-\frac{1}{2}    \|D^{k+1} c\|_{L^{2}}^2+ \|D^{k+2} c\|_{L^{2}}^2   \\ &\qquad \leq   \| D^{k}  (n \nabla c) \|_{L^2} \| D^{k+1} n \|_{L^2}.
	\end{aligned}
\end{equation}
If we add \eqref{O0} and \eqref{Ok} with $k=1$, then
\[
	\begin{aligned}
		\frac {1}{2}\frac {\ud}{\ud t} ( \| n\|_{H^{1}}^2+\| \nabla c\|_{H^{1}}^2  )  &+\frac{1}{4}\|D^{2} n\|_{L^{2}}^2 +\frac{1}{2}    \|D^{2} c\|_{L^{2}}^2+ \|D^{3} c\|_{L^{2}}^2\\ &\qquad \leq  \frac{1}{2}\int_{\R^{d}}   n^{2} | \Delta c|\,\ud x+   \| D  (n \nabla c) \|_{L^2} \| D^{2} n \|_{L^2}.
	\end{aligned}
\]
Adding \eqref{Ok} for $k=2,\cdots,m$ to the above estimate gives
\[
	\begin{aligned}
		\frac {1}{2}\frac {\ud}{\ud t} ( \| n\|_{H^{m}}^2+\| \nabla c\|_{H^{m}}^2  ) &+\frac{1}{4}\|D^{2} n\|_{H^{m-1}}^2  +\frac{1}{2}    \|D^{2} c\|_{H^{m-1}}^2+ \|D^{m+2} c\|_{L^{2}}^2\\ &\qquad \leq \frac{1}{2}\int_{\R^{d}}   n^{2} | \Delta c|\,\ud x+  \sum_{k=1}^{m} \| D^{k}  (n \nabla c) \|_{L^2} \| D^{k+1} n \|_{L^2}.
	\end{aligned}
\]
Using H\"older's and Young's inequalities, we estimate the first term on the right-hand side as
\[
\frac{1}{2}\int_{\R^{d}}   n^{2} | \Delta c| \,\ud x\le \frac{1}{2}\|n\|_{L^{2}}^{2}\|n\|_{L^{\infty}}^{2}+\frac{1}{8}\|D^{2}c\|_{L^{2}}^{2}.
\]
Using Lemma~\ref{KATOTYPELEM} with $(\alpha,\beta)=(k,0)$, we compute the rightmost term as
\[
	\begin{aligned}
\sum_{k=1}^{m} \| D^{k}  (n \nabla c) \|_{L^2} \| D^{k+1} n \|_{L^2} &\lesssim \sum_{k=1}^{m} ( \| D^{k} n \|_{L^2} \| \nabla c \|_{L^\infty}+\|   n \|_{L^\infty}\| D^{k+1} c \|_{L^2}        )\| D^{k+1} n \|_{L^2}\\
&\lesssim (\| n\|_{H^{m}}\| \nabla c\|_{L^{\infty}} +  \| n\|_{L^{\infty}}\| \nabla c\|_{H^{m}}) \| D^{2} n \|_{H^{m-1}}.
	\end{aligned}
\]
Combining the above estimates, after using Young's inequality, we have \eqref{L2ESTLEM}.
\end{proof}
The proof of Theorem~\ref{thm_crt} relies on the following two lemmas, whose proofs are deferred to Appendices~\ref{secA3} and~\ref{secA4}.
Note that 
\[
\| (1+t)^{\frac{d}{4}}\hat{n} (t) \|_{L^{1}_{\xi}L^{\infty}_{t}}:=\|(1+t)^{\frac{d}{4}} \hat{n}(\xi,t) \|_{L^{\infty}_{t}(0,T;L^{1}_{\xi}(\R^{d}))},\quad \| \hat{n} \|_{L^{\infty}_{\xi,t}}:=\| \hat{n}(\xi,t) \|_{L^{\infty}_{t}(0,T;L^{\infty}_{\xi}(\R^{d}))},\] 
\[
\|(1+t)^{\frac{d+1}{4}} |\xi|\hat{c} (t)\|_{L^{1}_{\xi}L^{\infty}_{t}}:=\| (1+t)^{\frac{d+1}{4}}|\xi|\hat{c}(\xi,t) \|_{L^{\infty}_{t}(0,T;L^{1}_{\xi}(\R^{d}))}
\]
 are well-defined for all $T<T_{\rm max}$  due to the regularity properties of $(n,c)$ in Lemma~\ref{LEMLWP}. 
\begin{lemma}\label{LEM42}
Let the same assumptions as in Theorem~\ref{thm_crt} be satisfied.
There exists $C>0$ such that for all $T<T_{\rm max}$
\[
\begin{aligned}
&\|(1+t)^{\frac{d}{4}} \hat{n}(t)\|_{L^{1}_{\xi}L^{\infty}_{t}}
\\&
\le C (\norm{    \widehat{n_0}  }_{L^{1}\cap L^{\infty}}  + \norm{|\xi|\widehat{c_0} }_{L^{1}\cap L^{\infty}})+C(\norm{\hat{n} }_{L^{\infty}_{\xi,t}}+\|(1+t)^{\frac{d}{4}} \hat{n}(t)\|_{L^{1}_{\xi}L^{\infty}_{t}})\|(1+t)^{\frac{d+1}{4}}|\xi| \hat{c}(t)\|_{L^{1}_{\xi}L^{\infty}_{t}},
\end{aligned}
\]
\[
\begin{aligned}
\norm{ \hat{n}(t)}_{L^{\infty}_{\xi,t}}
\le C(\norm{    \widehat{n_0}  }_{ L^{\infty}}  + \norm{|\xi|\widehat{c_0} }_{ L^{\infty}})+ C\norm{\hat{n} }_{L^{\infty}_{\xi,t}}\|(1+t)^{\frac{d+1}{4}}|\xi| \hat{c}(t)\|_{L^{1}_{\xi}L^{\infty}_{t}}.
\end{aligned}
\]
\end{lemma}

\begin{lemma}\label{LEM43}
Let the same assumptions as in Theorem~\ref{thm_crt} be satisfied.
There exists $C>0$ such that for all $T<T_{\rm max}$
\[
\begin{aligned}
&\|(1+t)^{\frac{d+1}{4}} |\xi|\hat{c}(t)\|_{L^{1}_{\xi}L^{\infty}_{t}}
\\&
\le C \bke{\norm{    \widehat{n_0}  }_{L^{1}\cap L^{\infty}}  + \norm{|\xi|\widehat{c_0} }_{L^{1}\cap L^{\infty}}}+C\bke{\norm{\hat{n} }_{L^{\infty}_{\xi,t}}+\|(1+t)^{\frac{d}{4}} \hat{n}(t)\|_{L^{1}_{\xi}L^{\infty}_{t}}}\|(1+t)^{\frac{d+1}{4}}|\xi| \hat{c}(t)\|_{L^{1}_{\xi}L^{\infty}_{t}}.
\end{aligned}
\]
\end{lemma}

\subsection{Main Proofs}
We are in position to prove Theorem~\ref{thm_crt}.
\subsubsection{Global Existence and Long Time Behaviors.}
Assume that condition \eqref{cond2}, given by 
\[\norm{\widehat{n_0}}_{L^{1}\cap L^{\infty}} + \norm{|\xi|\widehat{c_0}}_{L^{1}\cap L^{\infty}}\le \delta,
\] holds, where $\delta>0$ will be specified later.
We claim that $T_{\rm max}=\infty$ and there exists $C>0$ independent of $\delta$ such that
\begin{equation}\label{GOAL42MAIN}
\norm{    \widehat{n}   }_{L^{\infty}(0,\infty;L^{1}\cap L^{\infty}(\R^{d}))}  + \norm{|\xi|\widehat{c} }_{L^{\infty}(0,\infty;L^{1}\cap L^{\infty}(\R^{d}))} \le C\delta.
\end{equation}

  First, we show the continuity of 
\[
Y(T):=
\sup_{t\le T}\Bigr{(}\norm{\hat{n}(\cdot,t)}_{L^{\infty} }+\|(1+t)^{\frac{d}{4}} \hat{n}(\cdot,t)\|_{L^{1} }+\|(1+t)^{\frac{d+1}{4}}|\xi| \hat{c}(\cdot,t)\|_{L^{1} }\Bigr{)}
\]
in $[0,T_{\rm max})$. 
The continuity of $\|(1+t)^{\frac{d}{4}} \hat{n}(\cdot,t)\|_{L^{1} }$ and $\|(1+t)^{\frac{d+1}{4}}|\xi| \hat{c}(\cdot,t)\|_{L^{1} }$ are guaranteed by the regularity properties of $(n,c)$ in Lemma~\ref{LEMLWP}.  It remains to show the continuity of $\norm{\hat{n}(\cdot,t)}_{L^{\infty} }$. Note that for $0\le t<T<T_{\rm max}$
\begin{equation}\label{DIFFNHAT}
\begin{aligned}
\|\hat{n}(\cdot,T) - \hat{n}(\cdot,t)\|_{L^{\infty} }
& \le \sum_{j=+,-}\|(e^{-\lambda_{j}T} -e^{-\lambda_{j}t})  \brk{ \widehat{u_0},a_{j} } \brk{ b_{j}, e_{1}}\|_{L^{\infty}}
\\&\,\, + \sum_{j=+,-}\norm{(e^{-\lambda_{j} t }-e^{-\lambda_{j} T })\int_0^t e^{\lambda_{j} \tau } \brk{ i\xi \cdot\widehat{  (n \nabla c)}e_1, a_{j} } \brk{ b_{j},e_{1}} \,\ud \tau}_{L^{\infty}}
\\&\,\, + \sum_{j=+,-}\norm{e^{-\lambda_{j} T } \int_t^T e^{ \lambda_{j} \tau } \brk{ i\xi \cdot\widehat{  (n \nabla c)}e_1, a_{j} } \brk{ b_{j},e_{1}} \,\ud \tau}_{L^{\infty}},
\end{aligned}
\end{equation}
where
\[
 \brk{ \widehat{u_0},a_{\pm} } \brk{ b_{\pm}, e_{1}}= \bke{ \widehat{n_0}\frac{-2}{1\pm\sqrt{1+4|\xi|^{2}}} +\widehat{c_0} }\bke{\frac{\mp|\xi|^{2}}{\sqrt{1+4|\xi|^{2}}} }, 
\]
\[
\brk{ i\xi \cdot\widehat{  (n \nabla c)}e_1, a_{\pm} } \brk{ b_{\pm},e_{1}} =  \bke{ i\xi \cdot\widehat{  (n \nabla c)}\frac{-2}{1\pm\sqrt{1+4|\xi|^{2}}}  }\bke{\frac{\mp|\xi|^{2}}{\sqrt{1+4|\xi|^{2}}} }.
\]
Note also that a direct computation gives 
\[
\lambda_{\pm}\ge0,
\qquad
|\brk{ \widehat{u_0},a_{\pm} } \brk{ b_{\pm}, e_{1}} |\le |\widehat{n_{0}}|+|\xi||\widehat{c_{0}}|,
\]
  Minkowski's integral inequality shows for $0\le t_{1}<t_{2}<T_{\rm max}$
\[
\|\|\brk{ i\xi \cdot\widehat{  (n \nabla c)}e_1, a_{\pm} } \brk{ b_{\pm},e_{1}} \|_{L^{\infty}(t_{1},t_{2})} \|_{L^{\infty}_{\xi}}
\le\|\|\brk{ i\xi \cdot\widehat{  (n \nabla c)}e_1, a_{\pm} } \brk{ b_{\pm},e_{1}} \|_{L^{\infty}_{\xi}}\|_{L^{\infty}(t_{1},t_{2})},
\]
and  Young's convolution inequality and the Plancherel theorem yields
\[
\begin{aligned}
\|\brk{ i\xi \cdot\widehat{  (n \nabla c)}e_1, a_{\pm} } \brk{ b_{\pm},e_{1}} \|_{L^{\infty}}
&\lesssim  \|i\xi \cdot\widehat{  (n \nabla c)}\|_{L^{\infty}}
\\& \lesssim  \|(|\xi| |\hat{n}|)  * (|\xi| |\hat{c}|) \|_{L^{\infty}}+ \| |\hat{n}|   * (|\xi|^{2} |\hat{c}|) \|_{L^{\infty}}
\\& \lesssim \| |\xi|  \hat{n} \|_{L^{2}}  \| |\xi|  \hat{c} \|_{L^{2}} + \|   \hat{n} \|_{L^{2}}  \| |\xi|^{2}  \hat{c} \|_{L^{2}}
\\& \lesssim \|  n  \|_{H^{s}}  \|  c  \|_{H^{s+1}}.
\end{aligned}
\]
We compute the first term on the right-hand side of \eqref{DIFFNHAT}  as 
\[
\begin{aligned}
\sum_{j=+,-}&\|(e^{-\lambda_{j}T} -e^{-\lambda_{j}t})  \brk{ \widehat{u_0},a_{j} } \brk{ b_{j}, e_{1}}\|_{L^{\infty}} 
\\&\quad\le \sum_{j=+,-}\| e^{-\lambda_{j}t}\|_{L^{\infty}}\|(1-e^{-\lambda_{j}(T-t)} ) \brk{ \widehat{u_0},a_{j} } \brk{ b_{j}, e_{1}}\|_{L^{\infty}}
\\& \quad\le   \sum_{j=+,-}\| (1-e^{-\lambda_{j}(T-t)} ) (|\widehat{n_{0}}|+|\xi||\widehat{c_{0}}|)\|_{L^{\infty}}.
\end{aligned}
\]
Using the Riemann-Lebesgue lemma and $n_0,\nabla c_0\in L^{1}(\R^{d})$, we observe that  for any  $\varepsilon>0$, there exists $R>0$ such that
\[
\|(|\widehat{n_{0}}|+|\xi||\widehat{c_{0}}|)\|_{L^{\infty}(\R^{d}\setminus B(0;R))}\le \varepsilon.
\]
With such $R>0$, we can compute 
\[
\begin{aligned}
\|& (1-e^{-\lambda_{\pm}(T-t)} ) (|\widehat{n_{0}}|+|\xi||\widehat{c_{0}}|)\|_{L^{\infty}}
\\&\quad\le \| (1-e^{-\lambda_{\pm}(T-t)} ) (|\widehat{n_{0}}|+|\xi||\widehat{c_{0}}|)\|_{L^{\infty}(\R^{d}\setminus B(0;R))}+\| (1-e^{-\lambda_{\pm}(T-t)}  ) (|\widehat{n_{0}}|+|\xi||\widehat{c_{0}}|)\|_{L^{\infty}(B(0;R))}
\\&\quad\le \|   (|\widehat{n_{0}}|+|\xi||\widehat{c_{0}}|)\|_{L^{\infty}(\R^{d}\setminus B(0;R))} + \|(|\widehat{n_{0}}|+|\xi||\widehat{c_{0}}|)\|_{L^{\infty}(B(0;R))}\| (1-e^{-\lambda_{\pm}(T-t)}  )  \|_{L^{\infty}(B(0;R))}
\\&\quad\le \varepsilon  + \|(|\widehat{n_{0}}|+|\xi||\widehat{c_{0}}|)\|_{L^{\infty} }\| (1-e^{-\lambda_{\pm}(T-t)}  )  \|_{L^{\infty}(B(0;R))}.
\end{aligned}
 \]
Since $\| (1-e^{-\lambda_{\pm}(T-t)}  )  \|_{L^{\infty}(B(0;R))}$ tends to $0$ as $T\rightarrow t$ and $\varepsilon>0$ is arbitrary,
  the first term on the right-hand side of  \eqref{DIFFNHAT} also  approaches $0$ as $T\rightarrow t$. 
Similarly, we can show the second term on the right-hand side of  \eqref{DIFFNHAT} tends to $0$ as $T\rightarrow t$. Indeed, 
\[
\begin{aligned}
\sum_{j=+,-}&\norm{(e^{-\lambda_{j} t }-e^{-\lambda_{j} T })\int_0^t e^{\lambda_{j} \tau } \brk{ i\xi \cdot\widehat{  (n \nabla c)}e_1, a_{j} } \brk{ b_{j},e_{1}} \,\ud \tau}_{L^{\infty}}
\\ &\quad\lesssim \|  n  \|_{L^{\infty}(0,t;H^{s})}  \|  c  \|_{L^{\infty}(0,t;H^{s+1})} \sum_{j=+,-}\norm{(e^{-\lambda_{j} t }-e^{-\lambda_{j} T })\int_{0}^{t}e^{\lambda_{j}\tau}d\tau}_{L^{\infty}}
 \end{aligned}
 \]
 and observe that for any $\varepsilon>0$, there exists $R>0$ such that 
 \[
 \norm{ 1/\lambda_{\pm}}_{L^{\infty}(\R^{d}\setminus B(0;R))}\le \varepsilon.
 \]
With such $R>0$, we have
\[
\begin{aligned}
&\norm{(e^{-\lambda_{\pm} t }-e^{-\lambda_{\pm} T })\int_{0}^{t}e^{\lambda_{\pm}\tau}d\tau}_{L^{\infty}}
\\&\quad\le \norm{(e^{-\lambda_{\pm} t }-e^{-\lambda_{\pm} T })\int_{0}^{t}e^{\lambda_{\pm}\tau}d\tau}_{L^{\infty}(\R^{d}\setminus B(0;R))}+\norm{(e^{-\lambda_{\pm} t }-e^{-\lambda_{\pm} T })\int_{0}^{t}e^{\lambda_{\pm}\tau}d\tau}_{L^{\infty}( B(0;R))}
\\&\quad\le \norm{(e^{-\lambda_{\pm} t }-e^{-\lambda_{\pm} T })\frac{1}{\lambda_{\pm}}(e^{\lambda_{\pm}t}-1)}_{L^{\infty}(\R^{d}\setminus B(0;R))}+\|(e^{-\lambda_{\pm} t }-e^{-\lambda_{\pm} T })te^{\lambda_{\pm}t}\|_{L^{\infty}( B(0;R))}
\\&\quad\le \norm{1/\lambda_{\pm}}_{L^{\infty}(\R^{d}\setminus B(0;R))}+t\|(1-e^{-\lambda_{\pm} (T-t) })\|_{L^{\infty}( B(0;R))}
\\&\quad\le \varepsilon +t\|(1-e^{-\lambda_{\pm} (T-t) })\|_{L^{\infty}( B(0;R))},
\end{aligned}
\]
which entails that the second term on the right-hand side of  \eqref{DIFFNHAT}  approaches $0$ as $T\rightarrow t$. The rightmost term on the right-hand side of \eqref{DIFFNHAT} also approaches $0$ as $T\rightarrow t$ because 
\[
\begin{aligned}
\sum_{j=+,-}&\norm{e^{-\lambda_{j} T } \int_t^T e^{ \lambda_{j} \tau } \brk{ i\xi \cdot\widehat{  (n \nabla c)}e_1, a_{j} } \brk{ b_{j},e_{1}} \,\ud \tau}_{L^{\infty}}
\\& \lesssim \sum_{j=+,-}\|  n  \|_{L^{\infty}(0,T;H^{s})}  \|  c  \|_{L^{\infty}(0,T;H^{s+1})} \norm{e^{-\lambda_{j} T } \int_t^T e^{ \lambda_{j} \tau }   \,\ud \tau}_{L^{\infty}}
\\& \lesssim \sum_{j=+,-}\|  n  \|_{L^{\infty}(0,T;H^{s})}  \|  c  \|_{L^{\infty}(0,T;H^{s+1})}| T-t|.
\end{aligned}
\]
Therefore, $Y$ is continuous in $[0,T_{\rm max})$.  

Note from Lemma~\ref{LEM42} and Lemma~\ref{LEM43} that  
\[
Y(T)\lesssim \norm{    \widehat{n_0}  }_{L^{1}\cap L^{\infty}}  + \norm{|\xi|\widehat{c_0} }_{L^{1}\cap L^{\infty}}+Y^{2}(T) .
\]
Note also that $
Y(0)\le 
\norm{    \widehat{n_0}  }_{L^{1}\cap L^{\infty}}  + \norm{|\xi|\widehat{c_0} }_{L^{1}\cap L^{\infty}}$. 
Thus, choosing $\delta>0$ sufficiently small in \eqref{cond2} yields  
\begin{equation}\label{YBOUND}\sup_{T<T_{\rm max}}Y(T)\le C\delta
\end{equation}
for some $C>0$ independent of $\delta$.
In particular,  
$
\|f\|_{L^{\infty} }\lesssim \|\widehat{f}\|_{L^{1} }$,   $\|f\|_{L^{2} }^{2}=\|\widehat{f}\|_{L^{2} }^{2}\lesssim \|\widehat{f}\|_{L^{1} }\|\widehat{f}\|_{L^{\infty} }$ for all $f\in \mathcal{S}(\R^{d})$
 shows that
\begin{equation}\label{SEP4}
\|n(\cdot,t)\|_{L^{2}}^{2}+\|n(\cdot,t)\|_{L^{\infty}}\lesssim (1+t)^{-\frac{d}{4}},\qquad \|\nabla c(\cdot,t)\|_{L^{\infty}}\lesssim  (1+t)^{-\frac{d+1}{4}}\,\,\mbox{ for all }\,\, t<T_{\rm max}.
\end{equation}
Using 
\eqref{A1HM} with $m=s>\frac{d}{2}$, Gr\"onwall's inequality, and \eqref{SEP4}, we have $T_{\rm max}=\infty$ by   \eqref{BUCRITREF}. 

Next, we show \eqref{GOAL42MAIN}. As in the Section~\ref{SEC323},
  due to  \eqref{SEP4}, we have   
\begin{equation}\label{FEB19}
\begin{aligned}
\|c(\cdot,t)\|_{L^{\infty}}
&\lesssim \|e^{t(\Delta-1)}c_0\|_{L^{\infty}}+\int_{0}^{\frac{t}{2}}\|e^{(t-\tau)(\Delta-1)}n(\tau)\|_{L^{\infty}}\,\ud\tau+\int_{\frac{t}{2}}^{t}\|e^{(t-\tau)(\Delta-1)}n(\tau)\|_{L^{\infty}}\,\ud\tau
\\&\lesssim  e^{-t}\|c_0\|_{L^{\infty}}+\Bigr{(}\sup_{t>0}\| n(t)\|_{L^{\infty} }\Bigr{)}\int_{0}^{\frac{t}{2}}e^{-(t-\tau)} \,\ud\tau+ \int_{\frac{t}{2}}^{t}  e^{-(t-\tau)}\|n(\tau)\|_{L^{\infty}}\,\ud\tau
\\&\lesssim  e^{-t}\|c_0\|_{H^{s}}+e^{-\frac{t}{2}}\frac{t}{2}+ \Bigr{(}1+\frac{t}{2}\Bigr{)}^{-\frac{d}{4} }\int_{\frac{t}{2}}^{t}  e^{-(t-\tau)}\,\ud\tau
\\&\lesssim (1+t )^{-\frac{d}{4} },
\end{aligned}
\end{equation}
\begin{equation}\label{327PM}
\begin{aligned}
\|c(\cdot,t)\|_{L^{2}}
&\lesssim \|e^{t(\Delta-1)}c_0\|_{L^{2}}+\int_{0}^{\frac{t}{2}}\|e^{(t-\tau)(\Delta-1)}n(\tau)\|_{L^{2}}\,\ud\tau+\int_{\frac{t}{2}}^{t}\|e^{(t-\tau)(\Delta-1)}n(\tau)\|_{L^{2}}\,\ud\tau
\\&\lesssim  e^{-t}\|c_0\|_{L^{2}}+\Bigr{(}\sup_{t>0}\| n(t)\|_{L^{2} }\Bigr{)}\int_{0}^{\frac{t}{2}}e^{-(t-\tau)} \,\ud\tau+ \int_{\frac{t}{2}}^{t}  e^{-(t-\tau)}\|n(\tau)\|_{L^{2}}\,\ud\tau
\\&\lesssim  e^{-t}\|c_0\|_{L^{2}}+e^{-\frac{t}{2}}\frac{t}{2}+ \Bigr{(}1+\frac{t}{2}\Bigr{)}^{-\frac{d}{8} }\int_{\frac{t}{2}}^{t}  e^{-(t-\tau)}\,\ud\tau
\\&\lesssim (1+t )^{-\frac{d}{8} }.
\end{aligned}
\end{equation}
On the other hand, by  $\eqref{FTROFEQ}_{2}$ and \eqref{YBOUND},
we have for $j=0,1$ that
\[
\||\xi|^{j}\widehat{c}(t)\|_{L^{\infty}} \lesssim   \||\xi|^{j}\widehat{c_0} \|_{L^{\infty}}+ \int_0^t (t-\tau)^{-\frac{j}{2}}e^{-(t-\tau)} \|\widehat{ n}(\tau)\|_{L^{\infty}} \,\ud \tau\,\,\mbox{ for all }\,\, t>0.
\]
Thus, along with \eqref{YBOUND}, we can 
conclude \eqref{GOAL42MAIN} and
\begin{equation}\label{328PM}
\sup_{t>0}(\|\widehat{n}(\cdot,t)\|_{L^{\infty}}+\|\widehat{c}(\cdot,t)\|_{L^{\infty}})\le C.
\end{equation}

\subsubsection{Temporal Decay Rates of $n$ and $c$.}  
Note that the solution $(n,c)$ 
is global and smooth. In particular, by a bootstrapping argument, for any $m\ge1$ and $t>1$, 
\begin{equation}\label{FEB28249}
(n,c)\in (L^{\infty}([1,t];H^{m}(\R^{d})))^{2}.
\end{equation}
We claim that for any integer $m\ge0$,
\begin{equation}\label{WEBFEB}
\| D^m n(\cdot,t)\|_{L^{2}}+\| D^m c(\cdot,t)\|_{L^{2}}   \lesssim (1+t)^{-\frac{1}{2}(\frac{m}{2} + \frac{d}{4})}\quad \mbox{ for all }\,\,t> 1.
\end{equation}
Since the case $m=0$ follows from \eqref{SEP4} and \eqref{327PM}, it suffices to consider $m\ge1$.
As in the proof of Lemma~\ref{L2ESTLEM}, we have 
\[
		\frac {1}{2}\frac {\ud}{\ud t} \int_{\R^{d}} |D^m n|^2 \,\ud x + \int_{\R^{d}} |D^{m+1} n|^2 \,\ud x \le   \frac {1}{2} \| D^{m+1} n \|_{L^2}^2 + \frac{1}{2}\| D^{m+1} c \|_{L^2}^2+ \| D^{m}  (n \nabla c) \|_{L^2}\| D^{m+1} n \|_{L^2},
\]
\[		\frac {1}{2}\frac {\ud}{\ud t} \int_{\R^{d}} |D^{m+1} c|^2 \,\ud x + \int_{\R^{d}} (|D^{m+1} c|^2 + |D^{m+2} c|^2) \,\ud x \leq \frac {1}{2} (\| D^{m+1} n \|_{L^2}^2 + \| D^{m+1} c \|_{L^2}^2),
\]
and thus, we can see that
 \begin{equation*}
		\begin{aligned}
			&\frac {\ud}{\ud t} \left( \int_{\R^{d}} |D^m n|^2  +   |D^{m+1} n|^2 +   |D^{m+1} c|^2 \,\ud x \right) + \int_{\R^{d}} |D^{m+2} n|^2  +  |D^{m+2} c|^2 \,\ud x \\
			&\qquad\qquad\qquad\qquad\qquad\leq \| D^{m}  (n \nabla c) \|_{L^2}\| D^{m+1} n \|_{L^2} + \| D^{m+1}  (n \nabla c) \|_{L^2}\| D^{m+2} n \|_{L^2}.
		\end{aligned}
	\end{equation*} 
If we further estimate	 the first term on the right-hand side
using Lemma~\ref{KATO1} with $(\alpha,\beta)=(m,1)$, and interpolation and Young's inequalities,  then for some $C>0$ 
	 \begin{equation*}
		\begin{aligned}
			&\| D^{m}  (n \nabla c) \|_{L^2}\| D^{m+1} n \|_{L^2} \\
			&\lesssim ( \| D^{m+1} n \|_{L^2} \| c \|_{L^\infty} + \| D^{m+1} c \|_{L^2} \| n \|_{L^\infty})\| D^{m+1} n \|_{L^2} \\
			&\lesssim  \| D^{m+2} n \|_{L^2}^{\frac{2(m+1)}{m+2}} \| n \|_{L^2}^{\frac{2}{m+2}} \| c \|_{L^\infty} +  \| D^{m+2} c \|_{L^2}^{\frac{m+1}{m+2}} \| c \|_{L^2}^{\frac{1}{m+2}} \| n \|_{L^\infty} \| D^{m+2} n \|_{L^2}^{\frac{m+1}{m+2}} \| n \|_{L^2}^{\frac{1}{m+2}} \\
			&\leq \frac{1}{8} \| D^{m+2} n \|_{L^2}^2 + \frac{1}{8} \| D^{m+2} c \|_{L^2}^2 + C(\| n \|_{L^2}^2 \| c \|_{L^{\infty}}^{m+2} +   \| n \|_{L^2} \| c \|_{L^2} \| n \|_{L^{\infty}}^{m+2}).
		\end{aligned}
	\end{equation*} 
 Similarly, using Lemma~\ref{KATO1} with $(\alpha,\beta)=(m+1,1)$, 
\eqref{SEP4}--\eqref{FEB19} and Young's inequality,
we can compute the rightmost term to find a sufficiently large $T$ satisfying
\begin{equation*}
		\begin{aligned}
			 \| D^{m+1}  (n \nabla c) \|_{L^2}\| D^{m+2} n \|_{L^2} 
			& \lesssim ( \| D^{m+2} n \|_{L^2} \| c \|_{L^\infty} + \| D^{m+2} c \|_{L^2} \| n \|_{L^\infty})\| D^{m+2} n \|_{L^2}
			\\ &\le \frac{1}{8} (\| D^{m+2} n \|_{L^2}^2 +\| D^{m+2} c \|_{L^2}^2)\,\,\mbox{ for all }\,\, t\ge T.
		\end{aligned}
	\end{equation*}
Combining the above estimates gives, due to \eqref{SEP4}--\eqref{327PM}, that   \begin{equation*}
		\begin{aligned}
			 \frac {\ud}{\ud t}  ( \|D^m n\|_{L^{2}}^{2} + \|D^{m+1} n\|_{L^{2}}^{2}  +   \|D^{m+1}& c\|_{L^{2}}^{2} ) +\frac{1}{2} ( \|D^{m+2} n\|_{L^{2}}^{2}+  \|D^{m+2} c\|_{L^{2}}^{2} ) 
			 \\&\lesssim (\| n \|_{L^2}^2 \| c \|_{L^{\infty}}^{m+2} +   \| n \|_{L^2} \| c \|_{L^2} \| n \|_{L^{\infty}}^{m+2})
			 \\&\lesssim (1+t)^{-(\frac{m}{2} + \frac{d}{4})-1} \,\,\mbox{ for all }\,\,t\ge T.
		\end{aligned}
	\end{equation*} 
It follows from  \eqref{328PM} and the interpolation inequality,	  
\[
  \| |\xi|^{m+1}  \widehat{f} \|_{L^2}\!\!\lesssim\| \widehat{f} \|_{L^\infty}^{\omega}\| |\xi|^{m+2} \widehat{f} \|_{L^2}^{1-\omega}, \quad  \| |\xi|^{m}  \widehat{f} \|_{L^2} \!\!\lesssim \| \widehat{f} \|_{L^\infty}^{2\omega} \| |\xi|^{m+2} \widehat{f} \|_{L^2}^{1-2\omega}\,\,\mbox{ for all }f\in\mathcal{S}(\R^{d}),
\]
where $\omega:=(m+2+\frac{d}{2})^{-1}\in(0,\frac{1}{2})$, 
 that 
$y(t) :=  \|D^m n(\cdot,t)\|_{L^{2}}^{2} + \|D^{m+1} n(\cdot,t)\|_{L^{2}}^{2}  +   \|D^{m+1} c(\cdot,t)\|_{L^{2}}^{2} $ satisfies, with some $C>0$,
\[
y'+\frac{1}{C}\biggr{(} \|D^m n\|_{L^{2}}^{2\frac{1}{1-2\omega}}\!\!\!\!+\|D^{m+1} n\|_{L^{2}}^{2\frac{1}{1-\omega}}  \!\!+\|D^{m+1} c\|_{L^{2}}^{2\frac{1}{1-\omega}} \biggr{)}\!\!\le C (1+t)^{-(\frac{m}{2} + \frac{d}{4})-1}\,\,\mbox{ for all }\,\,t\ge T.
\]	
This gives $ y(t)\lesssim 1$ uniformly in $t\in [T,\infty)$
and thus, with some $C>0$,
\[y'+ \frac{1}{C} y^{\frac{1}{1-2\omega}} \leq C(1+t)^{-(\frac{m}{2} + \frac{d}{4})-1}\,\,\mbox{ for all }\,\,t\ge T.
\]
Equivalently,  $z(t):= (1+t)^{\frac{m}{2}+\frac{d}{4}}y(t)$ satisfies
\[
z'\le (1+t)^{-1}\bkt{ \left(\frac{m}{2}+\frac{d}{4}\right)z +C   -\frac{1}{C}z^{\frac{1}{1-2\omega}}}\,\,\mbox{ for all }\,\,t\ge T
\]
and thus, $ z(t)\lesssim 1$ uniformly in $t\in [T,\infty)$. In particular, with \eqref{FEB28249} we have
\[
\| D^m n(\cdot,t)\|_{L^{2}}   \lesssim (1+t)^{-\frac{1}{2}(\frac{m}{2} + \frac{d}{4})}\,\,\mbox{ for all }\,\,t> 1.
\]
 This entails, due to \eqref{SEP4}, that  for all $t>2$
\[
\begin{aligned}
\|&D^{m} c(\cdot,t)\|_{L^{2}} 
\\&\lesssim  \|  D^m e^{t(\Delta-1)} c_0\|_{L^{2}}  + \int_0^{\frac{t}{2}}\| D^{m}  e^{-(1-\Delta)(t-\tau)} n(\tau)\|_{L^{2}} \,\mathrm{d}\tau+\int_{\frac{t}{2}}^{t} \| D^{m}  e^{-(1-\Delta)(t-\tau)} n(\tau)\|_{L^{2}} \,\mathrm{d}\tau
\\&\lesssim   t^{-\frac{m}{2}}   e^{-t} \|c_0\|_{L^{2}}  + \int_0^{\frac{t}{2}}(t-\tau)^{-\frac{m}{2}}e^{-(t-\tau)}\| n(\tau)\|_{L^{2}} \,\mathrm{d}\tau+\int_{\frac{t}{2}}^{t}  e^{-(t-\tau)}\|    D^{m}n(\tau)\|_{L^{2}} \,\mathrm{d}\tau
\\&\lesssim   t^{-\frac{m}{2}}   e^{-t} \|c_0\|_{L^{2}}  +\Bigr{(} \frac{t}{2}\Bigr{)}^{1-\frac{m}{2}}e^{-\frac{t}{2}}  +\Bigr{(}1+\frac{t}{2}\Bigr{)}^{-\frac{1}{2}(\frac{m}{2} + \frac{d}{4})}\int_{\frac{t}{2}}^{t}  e^{-(t-\tau)} \,\mathrm{d}\tau
\\&\lesssim   (1+t)^{-\frac{1}{2}(\frac{m}{2} + \frac{d}{4})}.
\end{aligned}
\] 
Thus, together with \eqref{FEB28249}, we obtain \eqref{WEBFEB}.

\subsubsection{Temporal Decay Rates of $n-c$.} 
Similarly, for any integer $m\ge0$ and for all $t>2$, using \eqref{SEP4}--\eqref{327PM}, \eqref{WEBFEB}, and 
 Lemma~\ref{KATO1} with $(\alpha,\beta)=(m+1,1)$ we can compute
\[
\begin{aligned}
&\|D^m (n-c)(\cdot,t)\|_{L^{2}}
\\&\lesssim \|D^{m}e^{t(\Delta-1)}(n_0-c_0)\|_{L^{2}}+\int_{0}^{t}\|D^{m} e^{(t-\tau)(\Delta-1)}(\nabla \cdot(n\nabla c)+A\Delta c)(\tau)\|_{L^{2}}\,\ud\tau
\\&\lesssim t^{-\frac{m}{2}}e^{-t}\|n_0-c_0\|_{L^{2}}+\int_{0}^{\frac{t}{2}}e^{-(t-\tau)}\bkt{(t-\tau)^{-\frac{m+1}{2}}\|  n(\tau)\|_{L^{2}}\|\nabla c(\tau)\|_{L^{\infty}}+(t-\tau)^{-\frac{m+2}{2}}\|c(\tau)\|_{L^{2}}}\,\ud\tau
\\&\quad+\int_{\frac{t}{2}}^{t}  e^{-(t-\tau)}\bkt{\|D^{m+1} (n\nabla c)(\tau)\|_{L^{2}}+\|D^{m+2} c(\tau)\|_{L^{2}}}\,\ud\tau
\\&\lesssim t^{-\frac{m}{2}}e^{-t}+e^{-\frac{t}{2}}\int_{0}^{\frac{t}{2}} (t-\tau)^{-\frac{m+1}{2}} +(t-\tau)^{-\frac{m+2}{2}}\,\ud\tau
\\&\quad+\int_{\frac{t}{2}}^{t}  e^{-(t-\tau)}\bkt{\|D^{m+2}  n(\tau)\|_{L^{2}} \| c(\tau)\|_{L^{\infty}}+\|n(\tau)\|_{L^{\infty}} \| D^{m+2}  c(\tau)\|_{L^{2}}+\| D^{m+2}  c(\tau)\|_{L^{2}}}\,\ud\tau
\\&\lesssim t^{-\frac{m}{2}}e^{-t}+e^{-\frac{t}{2}}\Bigr{(} \frac{t}{2}\Bigr{)}^{1-\frac{m+1}{2}}+e^{-\frac{t}{2}}\Bigr{(} \frac{t}{2}\Bigr{)}^{1-\frac{m+2}{2}}
+\Bigr{(} 1+\frac{t}{2}\Bigr{)}^{-\frac{1}{2}(\frac{m+2}{2} + \frac{d}{4})}\int_{\frac{t}{2}}^{t}  e^{-(t-\tau)}\,\ud\tau
\\&\lesssim (1+ t )^{-\frac{1}{2}(\frac{m+2}{2} + \frac{d}{4})}.
\end{aligned}
\]
This shows 
\[
\| D^m (n-c)(\cdot,t)\|_{L^{2}}  \lesssim  (1+ t )^{-\frac{1}{2}(\frac{m+2}{2} + \frac{d}{4})}\quad \mbox{ for all }t> 1.
\]
\qed
\section{Supercritical Case: Proof of Theorem~\ref{thm3}}\label{SEC13}
 The proof relies on a specific choice of parameters and initial conditions. We outline these preliminary definitions below.

Let $A>1$, $d\ge2$, and let  $ s>\frac{d}{2}$ be an integer. We often suppressed $|\xi|$-dependence of $\lambda_{\pm}$ unless any confusion is to be expected. Note that 
\begin{equation}\label{LAMML}
 \begin{aligned}
 &\frac{d}{d\sigma}\lambda_{-}(\sigma)<0\mbox{ for } \sigma< \sqrt{\frac{A^{2}-1}{4A}}\quad\mbox{ and } \quad\frac{d}{d\sigma}\lambda_{-}(\sigma)\ge0 \mbox{ otherwise},
   \\&\lambda_{-}\ge   \lambda_{-}\bigr{|}_{|\xi|= \sqrt{\frac{A^{2}-1}{4A}}}= -\frac{ (A-1)^2 }{4A  },\,\,\mbox{ and }\,\, \lambda_{+}\ge1.
    \end{aligned}
 \end{equation}
 Denote by positive numbers
\begin{equation}\label{DEFTMU}
\begin{aligned}
\lambda_{0}&:= \lambda_{-}(\sqrt{2(A-1)}),\qquad
C_{0} :=  2\sqrt{9A^{2}-5} \bke{1+\frac{9(A^{2}-1)}{16A}}^{\frac{s}{2}}\max\bket{1,\frac{2}{\sqrt{A-A^{-1}}}}, 
\\
C_{1}&:= \int_{0}^{\infty}    \frac{e^{-\frac{1}{2}\lambda_{0}\sigma}}{\sqrt{\sigma}} \,\ud\sigma,
\qquad
C_{2}:=\sup_{\sigma>\sqrt{2(A-1)}}\frac{\sigma}{\sqrt{\lambda_{-}(\sigma)}},
\\
\eta_{0}&:= \frac{1}{6}\biggr{(} \frac{(2A-1)^{\frac{s}{2}}\sqrt{2(A-1)}  C_{d,s}}{\frac{(A-1)^{2}}{4A}\frac{2s-d}{2s}}3^{\frac{4s-d}{2s}}C_{0}^{\frac{2s-d}{2s}}(2\sqrt{2}C_{0})^{\frac{d}{2s}}+9\frac{C_{2}}{\sqrt{e}}C_{1}\tilde{C}_{d,s} C_{0}\biggr{)}^{-1}.
\end{aligned}
\end{equation}
 Here, $C_{d,s}\ge1$ and $\tilde{C}_{d,s}\ge1$  are numbers satisfying
\begin{equation}\label{FGHS}
\| D  f\|_{L^{\infty}}\le  C_{d,s}\|D   f\|_{L^{2}}^{\frac{2s-d}{2s}}\|   f\|_{H^{s+1}}^{\frac{d}{2s}}\quad\mbox{ for all }\,\, f\in H^{s+1}(\R^{d}),
\end{equation}
\begin{equation}\label{FGHS1}
\|f\nabla g\|_{H^{s}}\le \tilde{C}_{d,s}\|f\|_{H^{s}} \|g\|_{H^{s+1}}\quad\mbox{ for all }\,\, f\in H^{s}(\R^{d}), \,\,g\in H^{s+1}(\R^{d}).
\end{equation}
Let $\eta\in(0,\eta_0)$. Define $T_{\eta}:= \frac{4A}{(A-1)^{2}}\ln \frac{C_{0}\eta_{0}}{\eta}$, i.e.,  $\eta e^{\frac{(A-1)^{2}}{4A}T_{\eta}}=C_{0}\eta_{0}$.

We pick the initial data by considering the linear counterpart of \eqref{PREdf_u} as follows.  
\begin{lemma}\label{LININSTA}
There exist 
  $(n_{0,\eta},c_{0,\eta})\in(H^{s}\times H^{s+1})(\R^{d})$ and $\theta\in(0,1)$ such that $\| n_{0,\eta}\|_{ H^{s}(\R^{d})}= \eta$,  $\| c_{0,\eta}\|_{ H^{s+1}(\R^{d})}= 0$,  
\begin{equation}\label{THETATETA}
e^{\theta\frac{(A-1)^{2}}{4A}T_{\eta}}\ge \frac{1}{2}e^{\frac{(A-1)^{2}}{4A}T_{\eta}},
\end{equation}
and for all $t>0$
 \[
\begin{aligned}
 &e^{\theta\frac{(A-1)^{2}}{4A}t} \|\langle \widehat{u_{0,\eta}},a_{-} \rangle \langle b_{-}, e_{1} \rangle\|_{L^{2}(\R^{d})}
\\&\qquad \le \|e^{-\lambda_{-}t} \langle \widehat{u_{0,\eta}},a_{-} \rangle \langle b_{-},e_{1} \rangle  \|_{L^{2}(\R^{d})}
\\&\qquad\qquad\le 
e^{\frac{(A-1)^{2}}{4A}t} \|\langle \widehat{u_{0,\eta}},a_{-} \rangle \langle b_{-}, e_{1} \rangle \|_{L^{2}(\R^{d})},
\end{aligned}
\]
where $\widehat{u_{0,\eta}}:=(\widehat{n_{0,\eta}},\widehat{c_{0,\eta}})^{\mathsf{T}}$.
\end{lemma}
For the proof, see Appendix~\ref{secA5}. 

We assume the hypotheses of Theorem~\ref{thm3}. In the subsequent analysis, let $(n,c)$ be the solution provided by Lemma~\ref{LEMLWP} corresponding to the initial data $(n_{0},c_{0})=(n_{0,\eta},c_{0,\eta})$ constructed in Lemma~\ref{LININSTA}. We define 
\[
\begin{aligned}
T_{*}:=\sup\{t<T_{\rm max}\,|\, \|n (\cdot,t)\|_{H^{s}}\le    2C_{0}\eta_{0}\}, \qquad
T_{**}:=\sup\{t<T_{\rm max}\,|\, \|n (\cdot,t)\|_{L^{2}}\le   3\eta e^{\frac{ (A-1)^2 }{4A  }t}\}.
\end{aligned}
\]
\subsection{A Priori Estimates}
  To prove Theorem~\ref{thm3}, we first establish the following lemma.\begin{lemma}\label{LEMHSHS1}
Let  $(n,c)$ be the solution given by Lemma~\ref{LEMLWP} under $(n_{0},c_{0})=(n_{0,\eta},c_{0,\eta})$. 
For any $T<T_{\rm max}$, it holds that
\begin{equation}\label{SEP010}
\sup_{t\le T} \|D  c(\cdot,t)\|_{L^{2}(\R^{d})} \le \sup_{t\le T}\|n(\cdot,t)\|_{L^{2}(\R^{d})} ,
\end{equation}
\begin{equation}\label{SEP0010}
 \sup_{t\le T}\|c(\cdot,t)\|_{H^{s+1}(\R^{d})} \le \sqrt{2} \sup_{t\le T}\|n(\cdot,t)\|_{H^{s}(\R^{d})}.
\end{equation}
\end{lemma}
\begin{proof}
Let $ t<T_{\rm max}$.
For integer $k\in[1,s+1]$, using $\eqref{eq_crt}_{2}$ and Young's inequality, we have
\[
\begin{aligned}
		\frac {1}{2}\frac {\ud}{\ud t} \|D^k c\|_{L^{2}}^2  + \|D^k c\|_{L^{2}}^2 + \|D^{k+1} c\|_{L^{2}}^2  &\le \int_{\R^{d}} |D^{k-1} n| |D^{k+1} c| \,\ud x \\&\le \frac{1}{2} \|D^{k-1} n\|_{L^{2}}^2+\frac{1}{2} \|D^{k+1} c\|_{L^{2}}^2.
\end{aligned}
\]
This inequality for $k=1$ gives 
	 $\frac {\ud}{\ud t} \|D  c\|_{L^{2}}^2  + \|D  c\|_{L^{2}}^2      \le   \|  n\|_{L^{2}}^2$ and thus, solving this ODI, we have \eqref{SEP010}. 
Similarly,  adding the above inequality for $k=1,\cdots,s+1$ yields
$
	 \frac {\ud}{\ud t} \|D  c\|_{H^{s}}^2  + \|D  c\|_{H^{s}}^2      \le   \|  n\|_{H^{s}}^2$ 
and thus, solving this ODI, we obtain  
\begin{equation}\label{SEP10}
\sup_{t\le T} \|D  c(\cdot,t)\|_{H^{s}}^2\le \sup_{t\le T}\|n(\cdot,t)\|_{H^{s}}^2\quad\mbox{for all}\quad T<T_{\rm max}.
\end{equation}
If we take the $L^{2}$ scalar product of the $c$-equation with $c$ and apply Young's inequality, then  
$
		\frac {1}{2}\frac {\ud}{\ud t} \|  c\|_{L^{2}}^2  + \| c\|_{L^{2}}^2  = - \|D  c\|_{L^{2}}^2+\int_{\R^{d}} nc \,\ud x  
		 \le \frac{1}{2} \|  n\|_{L^{2}}^2+\frac{1}{2} \|  c\|_{L^{2}}^2$.
Solving this ODI, we obtain
\[
\sup_{t\le T} \|   c(\cdot,t)\|_{L^{2}}^2\le \sup_{t\le T}\|n(\cdot,t)\|_{L^{2}}^2\quad\mbox{for all}\quad T<T_{\rm max}.
\]
Adding it with \eqref{SEP10} and square rooting both sides,
we have the desired result, \eqref{SEP0010}.
\end{proof}
\subsection{Main Proofs}
We prove Theorem~\ref{thm3}.  Note from  Lemma~\ref{LEMLWP}, Lemma~\ref{LININSTA}, and a direct computation that
$0<T_{*}<T_{\rm max}$, $T_{**}>0$, $T_{\eta}\in(0,\infty)$.
We begin by establishing
\[
\min\{T_{\eta},T_{*},T_{**}\}= T_{\eta}.
\]
To establish this, we show that $\min\{T_{\eta},T_{*},T_{**}\}$ equals neither $T_{*}$ nor $T_{**}$, treating the two cases in turn.
\subsubsection{$\min\{T_{\eta},T_{*},T_{**}\}\not= T_{*}$.}
\noindent Assume to the contrary that 
$
\min\{T_{\eta},T_{*},T_{**}\}=T_{*}$.
Note that
\[
\|n (\cdot,T_{*})\|_{H^{s}}=\|(1+|\xi|^{2})^{\frac{s}{2}}\widehat{n}(\cdot,T_{*})\|_{L^{2}}=   2C_{0}\eta_{0}.
\]
From \eqref{PREdf_u}, we can compute
\[
\|(1+|\xi|^{2})^{\frac{s}{2}}\widehat{n}(\cdot,T_{*})\|_{L^{2}}\le  J^{*}_{1}+J^{*}_{2}+J^{*}_{3}+J^{*}_{4},
\]
where
 \[
J^{*}_{1}:=  \|e^{-\lambda_{-}T_{*}} (1+|\xi|^{2})^{\frac{s}{2}}\langle \widehat{u_{0,\eta}},a_{-} \rangle \langle b_{-},e_{1} \rangle  \|_{L^{2}},\qquad
J^{*}_{2}:=\|e^{-\lambda_{+}T_{*}} (1+|\xi|^{2})^{\frac{s}{2}}\langle \widehat{u_{0,\eta}},a_{+} \rangle \langle b_{+},e_{1} \rangle  \|_{L^{2}},
 \]
\[
J^{*}_{3}:=\norm{\int_{0}^{T_{*}}e^{-\lambda_{-}(T_{*}-\tau)}(1+|\xi|^{2})^{\frac{s}{2}}  i\xi\cdot\widehat{n \nabla c}\frac{1+\sqrt{1+4A|\xi|^{2}}}{2\sqrt{1+4A|\xi|^{2}}} \,\ud\tau}_{L^{2}},
\]
\[
J^{*}_{4}:=\norm{\int_{0}^{T_{*}}e^{-\lambda_{+}(T_{*}-\tau)}(1+|\xi|^{2})^{\frac{s}{2}} i\xi\cdot\widehat{n \nabla c}\frac{-2}{1+\sqrt{1+4A|\xi|^2{}}}\frac{- A|\xi|^{2}}{\sqrt{1+4A|\xi|^{2}}}
\,\ud\tau}_{L^{2}}.
\]
By a direct computation, \eqref{LAMML}, \eqref{NOHSNORM},  and the definition of $T_{\eta}$,
we have
\[
J_{1}^{*} =\norm{ e^{-\lambda_{-}T_{*}}\frac{1+\sqrt{1+4A|\xi|^{2}}}{2\sqrt{1+4A|\xi|^{2}}} (1+|\xi|^{2})^{\frac{s}{2}} \widehat{n_{0,\eta}} }_{L^{2}}
  \le    e^{ \frac{ (A-1)^2 }{4A  }T_{*}} \|n_{0,\eta}\|_{H^{s}}
  \le\eta e^{ \frac{ (A-1)^2 }{4A  }T_{\eta}}
 = C_{0}\eta_{0}.
\]
Since $-\lambda_{+}\le -\lambda_{-}$, we also obtain  
\[
J_{2}^{*} =\norm{ e^{-\lambda_{+}T_{*}}\frac{-2}{1+\sqrt{1+4A|\xi|^2{}}}\frac{-A|\xi|^{2}}{\sqrt{1+4A|\xi|^{2}}} (1+|\xi|^{2})^{\frac{s}{2}}\widehat{n_{0,\eta} } }_{L^{2}}
 \le \frac{1}{2}     e^{ \frac{ (A-1)^2 }{4A  }T_{*}} \|n_{0,\eta}\|_{H^{s}}
 \le \frac{1}{2}C_{0}\eta_{0}.
\]
To compute $J_{3}^{*}$, we note from \eqref{DEFTMU} that  
\[
\lambda_{-}(|\xi|)> \lambda_{0},
\quad
\frac{|\xi|}{\sqrt{\lambda_{-}}}\le C_{2}\quad \mbox{ for all } |\xi|>\sqrt{2(A-1)}.
\]
We further use
\eqref{LAMML}, the Plancherel theorem, H\"older's inequality, \eqref{FGHS}, and $e^{-x}\sqrt{x}\le \frac{1}{\sqrt{2 e}}$ for all  $x\ge0$ to show
\[
\begin{aligned}
J_{3}^{*}&\le  \norm{\int_{0}^{T_{*}}e^{-\lambda_{-}(T_{*}-\tau)}(1+|\xi|^{2})^{\frac{s}{2}}  |\xi||\widehat{n \nabla c}| \,\ud\tau}_{L^{2}(\{|\xi|\le\sqrt{2(A-1)}\}\cup \{|\xi|>\sqrt{2(A-1)}\})}
\\&\le   (2A-1)^{\frac{s}{2}}\sqrt{2(A-1)}  \int_{0}^{T_{*}} e^{\frac{(A-1)^{2}}{4A}(T_{*}-\tau)}   \|  n \nabla c   \|_{L^{2}}\,\ud\tau  
\\&\quad+\sqrt{2}\norm{ \int_{0}^{T_{*}} e^{-\lambda_{-}(T_{*}-\tau)}\sqrt{\frac{\lambda_{-}(T_{*}-\tau)}{2}} \frac{(1+|\xi|^{2})^{\frac{s}{2}}|\xi|  |\widehat{n \nabla c}|}{\sqrt{\lambda_{-}(T_{*}-\tau)}} \,\ud\tau  }_{L^{2}(|\xi|>\sqrt{2(A-1)})}
\\&\le   (2A-1)^{\frac{s}{2}}\sqrt{2(A-1)}  C_{d,s} \int_{0}^{T_{*}} e^{\frac{(A-1)^{2}}{4A}(T_{*}-\tau)}   \|  n \|_{L^{2}} 
 \|D  c\|_{L^{2}}^{\frac{2s-d}{2s}}\|   c\|_{H^{s+1}}^{\frac{d}{2s}}
\,\ud\tau  
\\&\quad+\frac{C_{2}}{\sqrt{e}}\norm{ \int_{0}^{T_{*}}   \frac{e^{-\frac{1}{2}\lambda_{0}(T_{*}-\tau)}}{\sqrt{ T_{*}-\tau }}(1+|\xi|^{2})^{\frac{s}{2}}  |\widehat{n \nabla c}| \,\ud\tau  }_{L^{2}(|\xi|>\sqrt{2(A-1)})}.
\end{aligned}
\]
Using Lemma~\ref{LEMHSHS1}, the definitions of $T_{**}$ and $T_{\eta}$, we compute the integral in the first term on the right-hand side as
\[
\begin{aligned}
  \int_{0}^{T_{*}} e^{\frac{(A-1)^{2}}{4A}(T_{*}-\tau)}   \|  n \|_{L^{2}} 
 \|D  c\|_{L^{2}}^{\frac{2s-d}{2s}}\|   c\|_{H^{s+1}}^{\frac{d}{2s}}
\,\ud\tau  
  &\le   \int_{0}^{T_{*}} e^{\frac{(A-1)^{2}}{4A}(T_{*}-\tau)} (3\eta e^{\frac{(A-1)^{2}}{4A}\tau})^{\frac{4s-d}{2s}} (2\sqrt{2}C_{0}\eta_{0})^{\frac{d}{2s}}d\tau
\\&  \le  e^{\frac{(A-1)^{2}}{4A}T_{*}} (3\eta   )^{\frac{4s-d}{2s}}(2\sqrt{2}C_{0}\eta_{0})^{\frac{d}{2s}}\int_{0}^{T_{*}}  ( e^{\frac{(A-1)^{2}}{4A}\tau})^{\frac{2s-d}{2s}} d\tau
\\&  \le   (3\eta e^{\frac{(A-1)^{2}}{4A}T_{*}} )^{\frac{4s-d}{2s}}(2\sqrt{2}C_{0}\eta_{0})^{\frac{d}{2s}}\frac{1}{\frac{(A-1)^{2}}{4A}\frac{2s-d}{2s}}
\\&  \le   (3C_{0})^{\frac{4s-d}{2s}}(2\sqrt{2}C_{0})^{\frac{d}{2s}}\frac{1}{\frac{(A-1)^{2}}{4A}\frac{2s-d}{2s}}\eta_{0}^{2}.
\end{aligned}
\]
On the other hand, we use 
\[
\int_{0}^{T}    \frac{e^{-\frac{1}{2}\lambda_{0}(T-\tau)}}{\sqrt{ T-\tau }} \,\ud\tau=\int_{0}^{T}    \frac{e^{-\frac{1}{2}\lambda_{0}\sigma}}{\sqrt{\sigma}} \,\ud\sigma\le  C_{1} \quad \mbox{ for all }\,\,T>0, 
\]
the Plancherel theorem, \eqref{FGHS1}, and \eqref{SEP0010} to compute the rightmost term as 
\[
\begin{aligned}
&\frac{C_{2}}{\sqrt{e}}\norm{ \int_{0}^{T_{*}}   \frac{e^{-\frac{1}{2}\lambda_{0}(T_{*}-\tau)}}{\sqrt{ T_{*}-\tau }}(1+|\xi|^{2})^{\frac{s}{2}}  |\widehat{n \nabla c}| \,\ud\tau  }_{L^{2}(|\xi|>\sqrt{2(A-1)})}
\\&\quad \le \frac{C_{2}}{\sqrt{e}} \int_{0}^{T_{*}}    \frac{e^{-\frac{1}{2}\lambda_{0}(T_{*}-\tau)}}{\sqrt{ T_{*}-\tau }} \,\ud\tau (\sup_{\sigma\le T_{*}}\|( n \nabla c)(\cdot,\sigma)   \|_{H^{s}})
\\&\quad \le \sqrt{2} \frac{C_{2}}{\sqrt{e}}C_{1}\tilde{C}_{d,s}( \sup_{\sigma\le T_{*}}\| n(\cdot,\sigma) \|_{H^{s}})^{2}
\\&\quad \le 9\frac{C_{2}}{\sqrt{e}}C_{1}\tilde{C}_{d,s} C_{0}^{2}\eta_{0}^{2}.
\end{aligned}
\]
Combining the above estimates gives
\[
\begin{aligned}
J_{3}^{*}\le  \biggr{(} \frac{(2A-1)^{\frac{s}{2}}\sqrt{2(A-1)}  C_{d,s}}{\frac{(A-1)^{2}}{4A}\frac{2s-d}{2s}}(3C_{0})^{\frac{4s-d}{2s}}(2\sqrt{2}C_{0})^{\frac{d}{2s}}+9\frac{C_{2}}{\sqrt{e}}C_{1}\tilde{C}_{d,s}    C_{0}^{2}\biggr{)}\eta_{0}^{2}.
\end{aligned}
\]
Similarly, due to $-\lambda_{+}\le -\lambda_{-}$, we can obtain
\[
\begin{aligned}
J_{4}^{*}&\le \frac{1}{2}\norm{\int_{0}^{T_{*}}e^{-\lambda_{-}(T_{*}-\tau)}(1+|\xi|^{2})^{\frac{s}{2}} |\xi| |\widehat{n \nabla c}|
\,\ud\tau}_{L^{2}(\{|\xi|\le\sqrt{2(A-1)}\}\cup \{|\xi|>\sqrt{2(A-1)}\})}
\\&\le \frac{1}{2} \biggr{(} \frac{(2A-1)^{\frac{s}{2}}\sqrt{2(A-1)}  C_{d,s}}{\frac{(A-1)^{2}}{4A}\frac{2s-d}{2s}}(3C_{0})^{\frac{4s-d}{2s}}(2\sqrt{2}C_{0})^{\frac{d}{2s}}+9\frac{C_{2}}{\sqrt{e}}C_{1}\tilde{C}_{d,s}    C_{0}^{2}\biggr{)}\eta_{0}^{2}.
\end{aligned}
\]
Thus, by the definition of $\eta_{0}$ in \eqref{DEFTMU}, we have a contradiction since 
\[
2C_{0}\eta_{0}  =\|n (\cdot,T_{*})\|_{H^{s}}
 \le J_{1}^{*}+J_{2}^{*}+J_{3}^{*}+J_{4}^{*}
 \le   \frac{3}{2}C_{0}\eta_{0}+ \frac{1}{4}C_{0}\eta_{0}=\frac{7}{4}C_{0}\eta_{0}.
\]
\subsubsection{$
\min\{T_{\eta},T_{*},T_{**}\}\not= T_{**}$.} 
\noindent Turning to the second case, suppose to the contrary that $\min\{T_{\eta},T_{*},T_{**}\}=T_{**}$. Note that  
\[
\|n (\cdot,T_{**})\|_{L^{2}}=\|\widehat{n}(\cdot,T_{**})\|_{L^{2}}=   3\eta e^{\frac{ (A-1)^2 }{4A  }T_{**}}
\]
and that by \eqref{PREdf_u},
 \[
 \|\widehat{n}(\cdot,T_{**})\|_{L^{2}}\le  J^{**}_{1}+J^{**}_{2}+J^{**}_{3}+J^{**}_{4},
 \]
 where
 \[
J^{**}_{1}:=  \|e^{-\lambda_{-}T_{**}} \langle \widehat{u_{0,\eta}},a_{-} \rangle \langle b_{-},e_{1} \rangle  \|_{L^{2}},
\qquad
J^{**}_{2}:=\|e^{-\lambda_{+}T_{**}} \langle \widehat{u_{0,\eta}},a_{+} \rangle \langle b_{+},e_{1} \rangle  \|_{L^{2}},
 \]
\[
J^{**}_{3}:=\norm{\int_{0}^{T_{**}}e^{-\lambda_{-}(T_{**}-\tau)}i\xi\cdot\widehat{n \nabla c}\frac{1+\sqrt{1+4A|\xi|^{2}}}{2\sqrt{1+4A|\xi|^{2}}} \,\ud\tau}_{L^{2}},
\]
\[
J^{**}_{4}:=\norm{\int_{0}^{T_{**}}e^{-\lambda_{+}(T_{**}-\tau)}i\xi\cdot\widehat{n \nabla c}\frac{-2}{1+\sqrt{1+4A|\xi|^2{}}}\frac{- A|\xi|^{2}}{\sqrt{1+4A|\xi|^{2}}}
\,\ud\tau}_{L^{2}}.
\]
The computations for $J^{**}_{k}$, $k=1,2,3,4$, are similar to those of $J^{*}_{k}$. Indeed, we have  
\[
J^{**}_{1}=\norm{ e^{-\lambda_{-}T_{**}}\frac{1+\sqrt{1+4A|\xi|^{2}}}{2\sqrt{1+4A|\xi|^{2}}} \widehat{n_{0,\eta}} }_{L^{2}}
  \le    e^{ \frac{ (A-1)^2 }{4A  }T_{**}} \|\widehat{n_{0,\eta}}\|_{L^{2}}
\le \eta e^{ \frac{ (A-1)^2 }{4A  }T_{**}},
\]
\[
J^{**}_{2} =\norm{ e^{-\lambda_{+}T_{**}}\frac{-2}{1+\sqrt{1+4A|\xi|^2{}}}\frac{-A|\xi|^{2}}{\sqrt{1+4A|\xi|^{2}}}\widehat{n_{0,\eta}} }_{L^{2}}
 \le \frac{1}{2}    e^{ \frac{ (A-1)^2 }{4A  }T_{**}}\|n_{0,\eta}\|_{L^{2}}
  \le \frac{1}{2}\eta e^{ \frac{ (A-1)^2 }{4A  }T_{**}},
\]
\[
\begin{aligned}
J^{**}_{3}& \le \norm{\int_{0}^{T_{**}}e^{-\lambda_{-}(T_{**}-\tau)}  |\xi||\widehat{n \nabla c}| \,\ud\tau}_{L^{2}(\{|\xi|\le\sqrt{2(A-1)}\}\cup \{|\xi|>\sqrt{2(A-1)}\})}
\\&\le   \sqrt{2(A-1)}  \int_{0}^{T_{**}} e^{\frac{(A-1)^{2}}{4A}(T_{**}-\tau)}   \|  n \nabla c   \|_{L^{2}}\,\ud\tau  
\\&\quad+\sqrt{2}\norm{ \int_{0}^{T_{**}} e^{-\lambda_{-}(T_{**}-\tau)}\sqrt{\frac{\lambda_{-}(T_{**}-\tau)}{2}} \frac{|\xi|  |\widehat{n \nabla c}|}{\sqrt{\lambda_{-}(T_{**}-\tau)}} \,\ud\tau  }_{L^{2}(|\xi|>\sqrt{2(A-1)})}
\\&\le  \sqrt{2(A-1)}  C_{d,s} \int_{0}^{T_{**}} e^{\frac{(A-1)^{2}}{4A}(T_{**}-\tau)}   \|  n \|_{L^{2}} 
 \|D  c\|_{L^{2}}^{\frac{2s-d}{2s}}\|   c\|_{H^{s+1}}^{\frac{d}{2s}}
\,\ud\tau  
\\&\quad+\frac{C_{2}}{\sqrt{e}}\norm{ \int_{0}^{T_{**}}   \frac{e^{-\frac{1}{2}\lambda_{0}(T_{**}-\tau)}}{\sqrt{ T_{**}-\tau }}   |\widehat{n \nabla c}| \,\ud\tau  }_{L^{2}(|\xi|>\sqrt{2(A-1)})}
\\&\le   \frac{ \sqrt{2(A-1)}  C_{d,s}}{\frac{(A-1)^{2}}{4A}\frac{2s-d}{2s}} (3\eta e^{\frac{(A-1)^{2}}{4A}T_{**}} )^{\frac{4s-d}{2s}}(2\sqrt{2}C_{0}\eta_{0})^{\frac{d}{2s}}+\frac{C_{2}}{\sqrt{e}} C_{1}  (3\eta e^{\frac{(A-1)^{2}}{4A}T_{**}} )^{2},
\end{aligned}
\]
\[
\begin{aligned}
J^{**}_{4}&\le \frac{1}{2}   \norm{\int_{0}^{T_{**}}e^{-\lambda_{-}(T_{**}-\tau)}  |\xi||\widehat{n \nabla c}| \,\ud\tau}_{L^{2}(\{|\xi|\le\sqrt{2(A-1)}\}\cup \{|\xi|>\sqrt{2(A-1)}\})}
\\&\le \frac{1}{2}\Biggr{(} \frac{ \sqrt{2(A-1)}  C_{d,s}}{\frac{(A-1)^{2}}{4A}\frac{2s-d}{2s}} (3\eta e^{\frac{(A-1)^{2}}{4A}T_{**}} )^{\frac{4s-d}{2s}}(2\sqrt{2}C_{0}\eta_{0})^{\frac{d}{2s}}+\frac{C_{2}}{\sqrt{e}} C_{1}  (3\eta e^{\frac{(A-1)^{2}}{4A}T_{**}} )^{2}\Biggr{)}.
\end{aligned}
\]
Combining the above estimates, after using $ 3\eta e^{\frac{(A-1)^{2}}{4A}T_{**}}  \le 3\eta e^{\frac{(A-1)^{2}}{4A}T_{\eta}} =3C_{0}\eta_{0} $ and the definition of $\eta_{0}$ in \eqref{DEFTMU},
leads to a contradiction since 
\[
3\eta e^{\frac{ (A-1)^2 }{4A  }T_{**}}=\|n (\cdot,T_{**})\|_{L^{2}}\le J^{**}_{1}+J^{**}_{2}+J^{**}_{3}+J^{**}_{4}
\le  \frac{1}{2}\eta e^{ \frac{ (A-1)^2 }{4A  }T_{**}}+ \frac{1}{4}\eta e^{ \frac{ (A-1)^2 }{4A  }T_{**}}= \frac{3}{4}\eta e^{ \frac{ (A-1)^2 }{4A  }T_{**}}.
\]

\noindent 
Therefore, we have the desired result, $\min\{T_{\eta},T_{*},T_{**}\}= T_{\eta}$.

Finally, we establish the following lower bounds for the $L^{2}$-norms of $n$, $c$, and $\nabla c$ at time $t=T_{\eta}$, which confirm the escape of the solution from the $\eta$-scale neighborhood, as formulated in \eqref{THM13GOAL}. 
\subsubsection{Lower Estimate for $n$.}
 
\noindent We show 
$
\| n(\cdot,T_{\eta})\|_{L^{2}}\ge \frac{\eta_{0}}{4}$.
 Note from \eqref{PREdf_u}  that  
$
 \|\widehat{n}(\cdot,T_{\eta})\|_{L^{2}}\ge  L_{1}-L_{2}-L_{3}-L_{4}$,
 where
  \[
L_{1}:=  \|e^{-\lambda_{-}T_{\eta}} \langle \widehat{u_{0,\eta}},a_{-} \rangle \langle b_{-},e_{1} \rangle  \|_{L^{2}},
\qquad
L_{2}:=  \|e^{-\lambda_{+}T_{\eta}} \langle \widehat{u_{0,\eta}},a_{+} \rangle \langle b_{+},e_{1} \rangle  \|_{L^{2}},
\]
\[
L_{3}:=\norm{\int_{0}^{T_{\eta}}e^{-\lambda_{-}(T_{\eta}-\tau)}i\xi\cdot\widehat{n \nabla c}\frac{1+\sqrt{1+4A|\xi|^{2}}}{2\sqrt{1+4A|\xi|^{2}}} \,\ud\tau}_{L^{2}},
\]
\[
L_{4}:=\norm{\int_{0}^{T_{\eta}}e^{-\lambda_{+}(T_{\eta}-\tau)}i\xi\cdot\widehat{n \nabla c}\frac{-2}{1+\sqrt{1+4A|\xi|^2{}}}\frac{- A|\xi|^{2}}{\sqrt{1+4A|\xi|^{2}}}
\,\ud\tau}_{L^{2}}.
\]
By a direct computation, \eqref{LAMMLL}, and $\varepsilon<\frac{1}{2}\sqrt{\frac{A^{2}-1}{4A}}$ in \eqref{EPSTHETA}, we have
\[
 \begin{aligned}
L_{1}&=\norm{ e^{-\lambda_{-}T_{\eta}}\frac{1+\sqrt{1+4A|\xi|^{2}}}{2\sqrt{1+4A|\xi|^{2}}} \eta C_{\theta}\chi_{\Bigr{\{} \abs{ |\xi|- \sqrt{\frac{A^{2}-1}{4A}}  }\le \varepsilon\Bigr{\}}} }_{L^{2}}
\\& \ge\frac{1}{2}\eta e^{\theta\frac{(A-1)^{2}}{4A}T_{\eta}}
\Biggr{(}\sigma_{d}\int_{\sqrt{\frac{A^{2}-1}{4A}} -\varepsilon}^{\sqrt{\frac{A^{2}-1}{4A}} +\varepsilon}(1+r^{2})^{s}r^{d-1}\,\ud r\Biggr{)}^{-\frac{1}{2}}\Biggr{(}\sigma_{d}\int_{\sqrt{\frac{A^{2}-1}{4A}} -\varepsilon}^{\sqrt{\frac{A^{2}-1}{4A}} +\varepsilon}r^{d-1}\,\ud r\Biggr{)}^{\frac{1}{2}}
\\& \ge\frac{1}{2}\eta e^{\theta\frac{(A-1)^{2}}{4A}T_{\eta}}\bket{1+\biggr{(}\sqrt{\frac{A^{2}-1}{4A}}+\varepsilon\biggr{)}^{2}   }^{-\frac{s}{2}}
\\& \ge\frac{1}{2}\eta e^{\theta\frac{(A-1)^{2}}{4A}T_{\eta}}\bke{1+\frac{9(A^{2}-1)}{16A}}^{-\frac{s}{2}}.
\end{aligned}
\]
Then, by \eqref{THETATETA}, it follows that
\[
L_{1}\ge\frac{\eta}{4}e^{\frac{(A-1)^{2}}{4A}T_{\eta}}\bke{1+\frac{9(A^{2}-1)}{16A}}^{-\frac{s}{2}}
 =\frac{1}{4}\bke{1+\frac{9(A^{2}-1)}{16A}}^{-\frac{s}{2}}(C_{0}\eta_0).
\]
Using  a direct computation, $\lambda_{+}\ge1$   in \eqref{LAMML} and \eqref{NOHSNORM},  we have
\[
L_{2} 
 =\norm{ e^{-\lambda_{+}T_{\eta}}\frac{-2}{1+\sqrt{1+4A|\xi|^2{}}}\frac{-A|\xi|^{2}}{\sqrt{1+4A|\xi|^{2}}}\widehat{n_{0,\eta}} }_{L^{2}}
 \le \frac{1}{2}e^{-T_{\eta}}
\|n_{0,\eta}\|_{L^{2}} \le  \frac{\eta}{2} 
 \le \frac{\eta_{0}}{2}.
\]
The computations for $L_{k}$, $k=3,4$, are similar to those of $J^{**}_{k}$. Indeed, we have
\[
\begin{aligned}
L_{3}&\le\norm{\int_{0}^{T_{\eta}}e^{-\lambda_{-}(T_{\eta}-\tau)}|\xi| |\widehat{n \nabla c}|  \,\ud\tau}_{L^{2}(\{|\xi|\le\sqrt{2(A-1)}\}\cup \{|\xi|>\sqrt{2(A-1)}\})}
\\&\le   \sqrt{2(A-1)}  \int_{0}^{T_{\eta}} e^{\frac{(A-1)^{2}}{4A}(T_{\eta}-\tau)}   \|  n \nabla c   \|_{L^{2}}\,\ud\tau  
\\&\quad+\sqrt{2}\norm{ \int_{0}^{T_{\eta}} e^{-\lambda_{-}(T_{\eta}-\tau)}\sqrt{\frac{\lambda_{-}(T_{\eta}-\tau)}{2}} \frac{|\xi|  |\widehat{n \nabla c}|}{\sqrt{\lambda_{-}(T_{\eta}-\tau)}} \,\ud\tau  }_{L^{2}(|\xi|>\sqrt{2(A-1)})}
\\&\le  \sqrt{2(A-1)}  C_{d,s} \int_{0}^{T_{\eta}} e^{\frac{(A-1)^{2}}{4A}(T_{\eta}-\tau)}   \|  n \|_{L^{2}} 
 \|D  c\|_{L^{2}}^{\frac{2s-d}{2s}}\|   c\|_{H^{s+1}}^{\frac{d}{2s}}
\,\ud\tau  
\\&\quad+\frac{C_{2}}{\sqrt{e}}\norm{ \int_{0}^{T_{\eta}}   \frac{e^{-\frac{1}{2}\lambda_{0}(T_{\eta}-\tau)}}{\sqrt{ T_{\eta}-\tau }}   |\widehat{n \nabla c}| \,\ud\tau  }_{L^{2}(|\xi|>\sqrt{2(A-1)})}
\\&\le   \frac{ \sqrt{2(A-1)}  C_{d,s}}{\frac{(A-1)^{2}}{4A}\frac{2s-d}{2s}} (3\eta e^{\frac{(A-1)^{2}}{4A}T_{\eta}} )^{\frac{4s-d}{2s}}(2\sqrt{2}C_{0}\eta_{0})^{\frac{d}{2s}}+\frac{C_{2}}{\sqrt{e}} C_{1}  (3\eta e^{\frac{(A-1)^{2}}{4A}T_{\eta}} )^{2},
\\&\le   \frac{ \sqrt{2(A-1)}  C_{d,s}}{\frac{(A-1)^{2}}{4A}\frac{2s-d}{2s}} 3^{\frac{4s-d}{2s}}(2\sqrt{2}C_{0})^{\frac{d}{2s}}\eta_{0}^{2}+9\frac{C_{2}}{\sqrt{e}} C_{1}  \eta_{0}^{2},
\end{aligned}
\]
and
\[
\begin{aligned}
L_{4}&\le \frac{1}{2}   \norm{\int_{0}^{T_{**}}e^{-\lambda_{-}(T_{**}-\tau)}  |\xi||\widehat{n \nabla c}| \,\ud\tau}_{L^{2}(\{|\xi|\le\sqrt{2(A-1)}\}\cup \{|\xi|>\sqrt{2(A-1)}\})}
\\&\le \frac{1}{2}\Biggr{(}  \frac{ \sqrt{2(A-1)}  C_{d,s}}{\frac{(A-1)^{2}}{4A}\frac{2s-d}{2s}} 3^{\frac{4s-d}{2s}}(2\sqrt{2}C_{0})^{\frac{d}{2s}}\eta_{0}^{2}+9\frac{C_{2}}{\sqrt{e}} C_{1}  \eta_{0}^{2}\Biggr{)}.
\end{aligned}
\]
Combining the above estimates and using the definitions of $C_{0}$ and $\eta_{0}$ given in \eqref{DEFTMU}, we have
 \[
 \begin{aligned}
\|\widehat{n}(\cdot,T_{\eta})\|_{L^{2}}\ge  L_{1}-L_{2}-(L_{3}+L_{4})
\ge\frac{1}{4}\bke{1+\frac{9(A^{2}-1)}{16A}}^{-\frac{s}{2}}(C_{0}\eta_{0})-\frac{\eta_{0}}{2}-\frac{\eta_{0}}{4}
\ge\frac{\eta_0}{4}.
\end{aligned}
\]
\subsubsection{Lower Estimate for $c$.}
Next, we establish that  $\| c(\cdot,T_{\eta})\|_{L^{2}}\ge \frac{8\eta_{0}}{3}$.
Note from \eqref{PREdf_u}  that  
$
 \|\widehat{c}(\cdot,T_{\eta})\|_{L^{2}}\ge  M_{1}-M_{2}-M_{3}-M_{4}$,
 where
  \[
M_{1}:=  \|e^{-\lambda_{-}T_{\eta}} \langle \widehat{u_{0,\eta}},a_{-} \rangle \langle b_{-},e_{2} \rangle  \|_{L^{2}},
\qquad
M_{2}:=  \|e^{-\lambda_{+}T_{\eta}} \langle \widehat{u_{0,\eta}},a_{+} \rangle \langle b_{+},e_{2} \rangle  \|_{L^{2}},
\]
\[
M_{3}:=\norm{\int_{0}^{T_{\eta}}e^{-\lambda_{-}(T_{\eta}-\tau)} \frac{i\xi\cdot\widehat{n \nabla c}}{\sqrt{1+4A|\xi|^{2}}}
\,\ud\tau}_{L^{2}},\qquad 
M_{4}:=\norm{\int_{0}^{T_{\eta}}e^{-\lambda_{+}(T_{\eta}-\tau)} \frac{i\xi\cdot\widehat{n \nabla c}}{\sqrt{1+4A|\xi|^{2}}}
\,\ud\tau}_{L^{2}}.
\]
Similarly to $L_{k}$, $k=1,2,3,4$, we can compute 
 \[
 \begin{aligned}
M_{1}&=\norm{ e^{-\lambda_{-}T_{\eta}}\frac{1}{\sqrt{1+4A|\xi|^{2}}} \eta C_{\theta}\chi_{\Bigr{\{} \abs{ |\xi|- \sqrt{\frac{A^{2}-1}{4A}}  }\le \varepsilon\Bigr{\}}} }_{L^{2}}
\\& \ge
\frac{2}{\sqrt{9A^{2}-5}}
 \eta e^{\theta\frac{(A-1)^{2}}{4A}T_{\eta}}
\Biggr{(}\sigma_{d}\int_{\sqrt{\frac{A^{2}-1}{4A}} -\varepsilon}^{\sqrt{\frac{A^{2}-1}{4A}} +\varepsilon}(1+r^{2})^{s}r^{d-1}\,\ud r\Biggr{)}^{-\frac{1}{2}}\Biggr{(}\sigma_{d}\int_{\sqrt{\frac{A^{2}-1}{4A}} -\varepsilon}^{\sqrt{\frac{A^{2}-1}{4A}} +\varepsilon}r^{d-1}\,\ud r\Biggr{)}^{\frac{1}{2}}
\\& \ge \frac{2}{\sqrt{9A^{2}-5}}\bke{1+\frac{9(A^{2}-1)}{16A}}^{-\frac{s}{2}}(C_{0}\eta_0),
\end{aligned}
\]
\[
M_{2} 
 =\norm{ e^{-\lambda_{+}T_{\eta}}\frac{1}{\sqrt{1+4A|\xi|^{2}}}\widehat{n_{0,\eta}} }_{L^{2}}
 \le  e^{-T_{\eta}}
\|n_{0,\eta}\|_{L^{2}} \le   \eta  
 \le \eta_{0},
\]
\[
\begin{aligned}
M_{3}&\le\norm{\int_{0}^{T_{\eta}}e^{-\lambda_{-}(T_{\eta}-\tau)}|\xi| |\widehat{n \nabla c}|  \,\ud\tau}_{L^{2}(\{|\xi|\le\sqrt{2(A-1)}\}\cup \{|\xi|>\sqrt{2(A-1)}\})}
\\&\le   \frac{ \sqrt{2(A-1)}  C_{d,s}}{\frac{(A-1)^{2}}{4A}\frac{2s-d}{2s}} 3^{\frac{4s-d}{2s}}(2\sqrt{2}C_{0})^{\frac{d}{2s}}\eta_{0}^{2}+9\frac{C_{2}}{\sqrt{e}} C_{1}  \eta_{0}^{2},
\end{aligned}
\] 
\[
\begin{aligned}
M_{4}&\le   \norm{\int_{0}^{T_{**}}e^{-\lambda_{-}(T_{**}-\tau)}  |\xi||\widehat{n \nabla c}| \,\ud\tau}_{L^{2}(\{|\xi|\le\sqrt{2(A-1)}\}\cup \{|\xi|>\sqrt{2(A-1)}\})}
\\&\le  \Biggr{(}  \frac{ \sqrt{2(A-1)}  C_{d,s}}{\frac{(A-1)^{2}}{4A}\frac{2s-d}{2s}} 3^{\frac{4s-d}{2s}}(2\sqrt{2}C_{0})^{\frac{d}{2s}}\eta_{0}^{2}+9\frac{C_{2}}{\sqrt{e}} C_{1}  \eta_{0}^{2}\Biggr{)}.
\end{aligned}
\]
Thus, 
 \[
 \|\widehat{c}(\cdot,T_{\eta})\|_{L^{2}}\ge  \frac{2}{\sqrt{9A^{2}-5}}\bke{1+\frac{9(A^{2}-1)}{16A}}^{-\frac{s}{2}}(C_{0}\eta_0)-\eta_{0}-\frac{1}{3}\eta_{0}\ge\frac{8\eta_{0}}{3}.
 \] 
\subsubsection{Lower Estimate for $\nabla c$.} 
 Last, we show  $\|\nabla  c(\cdot,T_{\eta})\|_{L^{2}}\ge \frac{5\eta_{0}}{6}$.
Note from \eqref{PREdf_u}  that  
$
 \||\xi|\widehat{c}(\cdot,T_{\eta})\|_{L^{2}}\ge  N_{1}-N_{2}-N_{3}-N_{4}$,
 where
  \[
N_{1}:=  \|e^{-\lambda_{-}T_{\eta}}|\xi| \langle \widehat{u_{0,\eta}},a_{-} \rangle \langle b_{-},e_{2} \rangle  \|_{L^{2}},
\qquad
N_{2}:=  \|e^{-\lambda_{+}T_{\eta}} |\xi|\langle \widehat{u_{0,\eta}},a_{+} \rangle \langle b_{+},e_{2} \rangle  \|_{L^{2}},
\]
\[
N_{3}:=\norm{\int_{0}^{T_{\eta}}e^{-\lambda_{-}(T_{\eta}-\tau)} \frac{|\xi|i\xi\cdot\widehat{n \nabla c}}{\sqrt{1+4A|\xi|^{2}}}
\,\ud\tau}_{L^{2}},\qquad 
N_{4}:=\norm{\int_{0}^{T_{\eta}}e^{-\lambda_{+}(T_{\eta}-\tau)} \frac{|\xi|i\xi\cdot\widehat{n \nabla c}}{\sqrt{1+4A|\xi|^{2}}}
\,\ud\tau}_{L^{2}}.
\]
Similarly to $M_{k}$, $k=1,2,3,4$, we have 
\[
 \begin{aligned}
N_{1}&=\norm{ e^{-\lambda_{-}T_{\eta}}\frac{|\xi|}{\sqrt{1+4A|\xi|^{2}}} \eta C_{\theta}\chi_{\Bigr{\{} \abs{ |\xi|- \sqrt{\frac{A^{2}-1}{4A}}  }\le \varepsilon\Bigr{\}}} }_{L^{2}}
\\& \ge
\frac{\sqrt{A-A^{-1}}}{2\sqrt{9A^{2}-5}}
 \eta e^{\theta\frac{(A-1)^{2}}{4A}T_{\eta}}
\Biggr{(}\sigma_{d}\int_{\sqrt{\frac{A^{2}-1}{4A}} -\varepsilon}^{\sqrt{\frac{A^{2}-1}{4A}} +\varepsilon}(1+r^{2})^{s}r^{d-1}\,\ud r\Biggr{)}^{-\frac{1}{2}}\Biggr{(}\sigma_{d}\int_{\sqrt{\frac{A^{2}-1}{4A}} -\varepsilon}^{\sqrt{\frac{A^{2}-1}{4A}} +\varepsilon}r^{d-1}\,\ud r\Biggr{)}^{\frac{1}{2}}
\\& \ge \frac{\sqrt{A-A^{-1}}}{2\sqrt{9A^{2}-5}}\bke{1+\frac{9(A^{2}-1)}{16A}}^{-\frac{s}{2}}(C_{0}\eta_0),
\end{aligned}
\]
\[
N_{2} 
 =\norm{ e^{-\lambda_{+}T_{\eta}}\frac{|\xi|}{\sqrt{1+4A|\xi|^{2}}}\widehat{n_{0,\eta}} }_{L^{2}}
 \le  e^{-T_{\eta}}
\|n_{0,\eta}\|_{H^{s}} \le   \eta  
 \le \eta_{0},
\]
\[
\begin{aligned}
N_{3}&\le\frac{1}{2\sqrt{A}}\norm{\int_{0}^{T_{\eta}}e^{-\lambda_{-}(T_{\eta}-\tau)}|\xi| |\widehat{n \nabla c}|  \,\ud\tau}_{L^{2}(\{|\xi|\le\sqrt{2(A-1)}\}\cup \{|\xi|>\sqrt{2(A-1)}\})}
\\&\le  \frac{1}{2\sqrt{A}} \Biggr{(}\frac{ \sqrt{2(A-1)}  C_{d,s}}{\frac{(A-1)^{2}}{4A}\frac{2s-d}{2s}} 3^{\frac{4s-d}{2s}}(2\sqrt{2}C_{0})^{\frac{d}{2s}}\eta_{0}^{2}+9\frac{C_{2}}{\sqrt{e}} C_{1}  \eta_{0}^{2} \Biggr{)},
\end{aligned}
\] 
\[
\begin{aligned}
N_{4}&\le  \frac{1}{2\sqrt{A}} \norm{\int_{0}^{T_{**}}e^{-\lambda_{-}(T_{**}-\tau)}  |\xi||\widehat{n \nabla c}| \,\ud\tau}_{L^{2}(\{|\xi|\le\sqrt{2(A-1)}\}\cup \{|\xi|>\sqrt{2(A-1)}\})}
\\&\le  \frac{1}{2\sqrt{A}}\Biggr{(}  \frac{ \sqrt{2(A-1)}  C_{d,s}}{\frac{(A-1)^{2}}{4A}\frac{2s-d}{2s}} 3^{\frac{4s-d}{2s}}(2\sqrt{2}C_{0})^{\frac{d}{2s}}\eta_{0}^{2}+9\frac{C_{2}}{\sqrt{e}} C_{1}  \eta_{0}^{2}\Biggr{)}.
\end{aligned}
\]
Thus,
 \[
 \||\xi|\widehat{c}(\cdot,T_{\eta})\|_{L^{2}}\ge  \frac{\sqrt{A-A^{-1}}}{2\sqrt{9A^{2}-5}}\bke{1+\frac{9(A^{2}-1)}{16A}}^{-\frac{s}{2}}(C_{0}\eta_0)-\eta_{0}-\frac{1}{6}\eta_{0}\ge\frac{5\eta_{0}}{6}.
 \]
\qed 
\section{Conclusion}
In this work, we identified the critical threshold $A_{\mathrm{crit}}=1$ for the fully parabolic Keller--Segel system near constant steady states and established a complete dichotomy between the nonlinearly stable ($A \le A_{\mathrm{crit}}$) and unstable ($A > A_{\mathrm{crit}}$) regimes. As a byproduct, we also proved that in the setting $A=0$ and $d=2$, the system with $\gamma>0$ does not possess a critical mass $M_{\mathrm{crit}}$ that forces radial solutions to blow up in finite time, in contrast to the $\gamma=0$ case. Because our analytical approach appears broadly applicable to other variants of the model, we anticipate that these insights into critical parameters will help advance the theoretical understanding of Keller--Segel-type systems.
Future work may address whether the smallness assumptions imposed on the initial perturbations can be weakened, whether sharper asymptotic behavior can be derived, and whether the present stability--instability framework can be extended to bounded or periodic domains, applied near other steady states, or adapted for Keller--Segel systems with more general structures.
\appendix
\section{Appendix}
This appendix contains the proofs of several lemmas deferred from the main text.
\subsection{Proof of Lemma~\ref{REVLEM31}}\label{secA1}
Let $j=0$ or $1$.
To obtain \eqref{est_NABLAn_L2+}--\eqref{est_NABLAc_L2+}, we claim that
\begin{equation}\label{est_n_L2+}
	\begin{aligned}
		&\int_{\R^{d}} \||\xi| \widehat{n} \|_{L^2(0,T)} \,\ud \xi \leq   \| \widehat{n_0} \|_{L^1} +   \int_{\R^{d}}\bigr{(} \| \widehat{n \nabla c} \|_{L^2(0,T)}+A\| |\xi|\widehat{c}  \|_{L^2(0,T)}\bigr{)} \,\ud \xi,\\
		&\int_{\R^{d}} \| |\xi| \widehat{c} \|_{L^2(0,T)} \,\ud \xi \leq    \| |\xi|^{j}\widehat{c_0} \|_{L^1} +   \int_{\R^{d}} \| |\xi|\widehat{n} \|_{L^2(0,T)}\,\ud \xi.
	\end{aligned}
\end{equation}
Indeed, 
using $\eqref{FTROFEQ}_{1}$,  we can compute
\[
	\begin{aligned}
		\| |\xi|\widehat{n} \|_{L^2(0,T)} \leq \| e^{-|\xi|^2 t} |\xi|  \widehat{n_0}  \|_{L^2(0,T)} + \left\| \int_0^t e^{-|\xi|^2(t-\tau)} |\xi|^2 (|\widehat{n \nabla c}|+A|\xi|| \widehat{c} |) \,\ud \tau \right\|_{L^2(0,T)}.
	\end{aligned}
\]
We further compute
the first term on the right-hand side as
\[
\| e^{-|\xi|^2 t} |\xi|  \widehat{n_0}  \|_{L^2(0,T)} \le  |\widehat{n_0}|.
\]
Using Young's convolution inequality, we compute the rightmost term as
\[
 \left\| \int_0^t e^{-|\xi|^2(t-\tau)} |\xi|^2 (|\widehat{n \nabla c}|+A|\xi|| \widehat{c} |) \,\ud \tau \right\|_{L^2(0,T)}
\le   \| \widehat{n \nabla c} \|_{L^2(0,T)}+A\| |\xi|\widehat{c}  \|_{L^2(0,T)}.
 \]
Combining the above estimates and integrating over $\xi$, we have $\eqref{est_n_L2+}_{1}$. Similarly,  from $\eqref{FTROFEQ}_{2}$ we have
\[
	\begin{aligned}
		\| |\xi|\widehat{c} \|_{L^2(0,T)} \leq \| e^{-(1+|\xi|^2) t} |\xi|  \widehat{c_0}  \|_{L^2(0,T)}+ \left\| \int_0^t e^{-(1+|\xi|^2)(t-\tau)}  |\xi||\widehat{n} |  \,\ud \tau \right\|_{L^2(0,T)}.
	\end{aligned}
\]
A direct computation yields  
\[
\| e^{-(1+|\xi|^2) t} |\xi| \widehat{c_0}  \|_{L^2(0,T)}\le 
\left\{
\begin{aligned}  & \| e^{-|\xi|^2 t} |\xi| \widehat{c_0} \|_{L^2(0,T)}\le  |\widehat{c_0}|,\\
& \| e^{-t} |\xi| \widehat{c_0} \|_{L^2(0,T)}\le  |\xi||\widehat{c_0}|,
\end{aligned}
\right.
\]
and Young's convolution inequality gives
\[
\begin{aligned}
 \left\| \int_0^t e^{-(1+|\xi|^2)(t-\tau)}  |\xi||\widehat{n} |  \,\ud \tau \right\|_{L^2(0,T)}  \le   \| e^{-t} \|_{L^1(0,\infty)} \||\xi| \widehat{n} \|_{L^2(0,T)}
 \le  \||\xi| \widehat{n} \|_{L^2(0,T)}.
 \end{aligned}
\]
Combining the above estimates and integrating over $\xi$ shows $\eqref{est_n_L2+}_{2}$. Then,  \eqref{est_n_L2+} gives
\eqref{est_NABLAn_L2+}--\eqref{est_NABLAc_L2+}.

Next, to show \eqref{est_nc_L2+}, we denote  $*$ the convolution with respect to $\xi$, and note that 
\[
\widehat{n \nabla c} = \widehat{n} * (i\xi\widehat{c})=:J_1 + J_2 + J_3 + J_4,
\]
where 
\begin{equation*}
	\begin{aligned}
		J_1 &:= \left( e^{-|\xi|^2 t} \widehat{n_0} \right) * \left( e^{-(1+|\xi|^2) t} i\xi\widehat{c_0} \right), \\
		J_2 &:= \left( e^{-|\xi|^2 t} \widehat{n_0} \right) * \left( \int_0^t e^{-(1+|\xi|^2) (t-\tau)} i\xi\widehat{n} \,\ud \tau \right), \\
		J_3 &:= - \left( \int_0^t e^{-|\xi|^2(t-\tau)} i \xi \cdot (\widehat{n \nabla c}+A\xi\widehat{c}) \,\ud \tau \right) * \left( e^{-(1+|\xi|^2) t}i\xi \widehat{c_0} \right), \\
		J_4 &:= - \left( \int_0^t e^{-|\xi|^2(t-\tau)} i \xi \cdot (\widehat{n \nabla c}+A\xi\widehat{c}) \,\ud \tau \right) * \left( \int_0^t e^{-(1+|\xi|^2) (t-\tau)} i\xi\widehat{n} \,\ud \tau \right).
	\end{aligned}
\end{equation*}
A direct computation shows
\begin{equation*}
	\begin{aligned}
		\| J_1 \|_{L^2(0,T)} 
		&\le |\widehat{n_0}| *  \| e^{-(1+|\xi|^2) t} |\xi| \widehat{c_0}  \|_{L^2(0,T)} \\ 
		&\le  |\widehat{n_0}| * \left\{
\begin{aligned}  & \| e^{-|\xi|^2 t} |\xi| \widehat{c_0} \|_{L^2(0,T)}\le  |\widehat{c_0}|,\\
& \| e^{-t} |\xi| \widehat{c_0} \|_{L^2(0,T)}\le  |\xi||\widehat{c_0}|,
\end{aligned}
\right.
	\end{aligned}
\end{equation*} 
and Young's convolution inequality yields
\begin{equation*}
	\begin{aligned}
		\| J_2 \|_{L^2(0,T)} &\leq |\widehat{n_0}| * \left\| \int_0^t e^{-(t-\tau)} |\xi||\widehat{n}| \,\ud \tau \right\|_{L^2(0,T)} \leq  |\widehat{n_0}| * \| |\xi|\widehat{n} \|_{L^2(0,T)}. 
	\end{aligned}
\end{equation*}
Similarly, we can compute
\begin{equation*}
	\begin{aligned}
		\| J_3 \|_{L^2(0,T)} &\leq \left\| \int_0^t e^{-|\xi|^2(t-\tau)}  |\xi|(|\widehat{n \nabla c}|+A|\xi||\widehat{c}|) \,\ud \tau \right\|_{L^{\infty}(0,T)} *  \| e^{-(1+|\xi|^2) t} |\xi| |\widehat{c_0}|  \|_{L^2(0,T)} \\
		&\leq  (\| \widehat{n \nabla c} \|_{L^2(0,T)}+A\||\xi| \widehat{c} \|_{L^2(0,T)}) * \left\{
\begin{aligned}  & \| e^{-|\xi|^2 t} |\xi| \widehat{c_0} \|_{L^2(0,T)}\le  |\widehat{c_0}|,\\
& \| e^{-t} |\xi| \widehat{c_0} \|_{L^2(0,T)}\le  |\xi||\widehat{c_0}|,
\end{aligned}
\right.
	\end{aligned}
\end{equation*} and 
\begin{equation*}
	\begin{aligned}
		\| J_4 \|_{L^2(0,T)} &\leq \left\| \int_0^t e^{-|\xi|^2(t-\tau)} |\xi|   (|\widehat{n \nabla c}|+A|\xi||\widehat{c}|) \,\ud \tau \right\|_{L^{\infty}(0,T)} * \left\| \int_0^t e^{-(t-\tau)} |\xi||\widehat{n}| \,\ud \tau \right\|_{L^2(0,T)} \\
		&\leq  (\| \widehat{n \nabla c} \|_{L^2(0,T)}+A\| |\xi |\widehat{c}  \|_{L^2(0,T)}) * \| |\xi|\widehat{  n} \|_{L^2(0,T)}. 
	\end{aligned}
\end{equation*}
Combining the above estimates, integrating with respect to $\xi$, and using Young's convolution inequality,
we have \eqref{est_nc_L2+}.

Next, we use $\eqref{FTROFEQ}_{1}$ to see
\[
 \|\widehat{n}(\cdot,t) \|_{L^{1}}
\le \|\widehat{n_0}\|_{L^{1}} + \int_{\R^{d}}\int_0^t e^{-|\xi|^2(t-\tau)}  |\xi|  (|\widehat{n \nabla c}|+A|\xi||\widehat{  c}|) \,\ud \tau \,\ud \xi,
\]
and apply H\"older's inequality to compute the rightmost term as
\[
\begin{aligned}
\int_0^t e^{-|\xi|^2(t-\tau)} & |\xi|  (|\widehat{n \nabla c}|+A|\xi||\widehat{  c}|) \,\ud \tau
\\& \le \|e^{-|\xi|^2(t-\tau)}  |\xi| \|_{L^2(0,t)} ( \| \widehat{n \nabla c} \|_{L^2(0,t)}+A \||\xi| \widehat{c}  \|_{L^2(0,t)})
\\& \le | \widehat{n \nabla c} \|_{L^2(0,t)}+A \||\xi| \widehat{c}  \|_{L^2(0,t)}.
\end{aligned}
\]
Combining the above estimates gives \eqref{est_L1+}.

Next, to show $\eqref{est_L1++}$, from $\eqref{FTROFEQ}_{2}$  we observe that
\[
	\begin{aligned}
 \|\widehat{c}(\cdot,t)\|_{L^{1}} &\le \| e^{-(1+|\xi|^{2})t}\widehat{c_0}\|_{L^{1}} + \int_{\R^{d}}\int_0^t e^{-(1+|\xi|^2)(t-\tau)} |\widehat{ n}| \,\ud \tau \,\ud \xi
\\
&\le \|\widehat{c_0}\|_{L^{1}} + \int_{\R^{d}}\int_0^t e^{-(t-\tau)} |\widehat{ n}| \,\ud \tau \,\ud \xi
\\
 &\le \|\widehat{c_0}\|_{L^{1}} +  \sup_{\tau \le t}\|\widehat{n}(\cdot,\tau) \|_{L^{1}},
	\end{aligned}
\]
where we used Fubini's theorem and H\"older's inequality in the last line. This shows $\eqref{est_L1++}_{1}$. 

Similarly, we use $\eqref{FTROFEQ}_{2}$ to compute
\[
	\begin{aligned}
 \||\xi|\widehat{c}(\cdot,t)\|_{L^{1}} &\le \| e^{-(1+|\xi|^{2})t}|\xi||\widehat{c_0}|\|_{L^{1}} + \int_{\R^{d}}\int_0^t e^{-(1+|\xi|^2)(t-\tau)} |\xi||\widehat{ n}| \,\ud \tau \,\ud \xi
\\
&\le \||\xi|\widehat{c_0}\|_{L^{1}} + \int_{\R^{d}}  \| e^{-(t-\tau)} \|_{L^{2}(0,t)} \| |\xi| \widehat{ n}\|_{L^{2}(0,t)} \,\ud \xi
\\
&\le \||\xi|\widehat{c_0}\|_{L^{1}} + \int_{\R^{d}}  \| |\xi| \widehat{ n}\|_{L^{2}(0,t)} \,\ud \xi,
	\end{aligned}
\]
which yields $\eqref{est_L1++}_{2}$. \qed
\subsection{Proof of Remark~\ref{RMKL1}}\label{secA2}
Let $A=0$, and  $n_0$, $c_0$ are in $L^{1}(\R^{d})$ and nonnegative. Then, $n\ge0$, $c\ge0$ by the maximum principle. The time evolution formula for $\|n(\cdot,t)\|_{L^{1}}$ and $\|c(\cdot,t)\|_{L^{1}}$ are obtained by integrating the $n$- and $c$-equations over $\R^{d}$.
Using $(n,c)\in (L^{\infty}(0,\infty;L^{1}(\R^{d})))^{2}$ and
 $\| D^{m-1} f \|_{L^2}\lesssim \| f \|_{L^1}^{\frac{2}{d+2m}}\| D^{m} f \|_{L^2}^{1-\frac{2}{d+2m}}$ for all $f\in (L^{1}\cap H^{m})(\mathbb{R}^{d})$ instead of  \eqref{SEP121}, we can replace \eqref{SEP122} by $\| D^{m-1} n (\cdot,t)\|_{L^2} + \| D^{m}c (\cdot,t)\|_{L^2} \lesssim t^{-\frac{ 2(m-1)+d}{4}}$ for all $t>1$. Then, in a similar manner as in
  the part of the proof of Theorem~\ref{thm1}, one can obtain the desired result in Remark~\ref{RMKL1}. \qed

\subsection{Proof of Lemma~\ref{LEM42}}\label{secA3}
Let $t\le T<T_{\rm max}$.
Note from \eqref{PREdf_u} that 
\[
\hat{n}(\xi,t)=\sum_{j=+,-}e^{-\lambda_{j}t} \brk{ \widehat{u_0},a_{j} } \brk{ b_{j}, e_{1}}+  \int_0^t e^{-\lambda_{j}(t-\tau)} \brk{ i\xi \cdot\widehat{  (n \nabla c)}e_1, a_{j} } \brk{ b_{j},e_{1}} \,\ud \tau,
\]
where
\[
e^{-\lambda_{\pm}t} \brk{ \widehat{u_0},a_{\pm} } \brk{ b_{\pm}, e_{1}}=e^{-\lambda_{\pm}t} \bke{ \widehat{n_0}\frac{-2}{1\pm\sqrt{1+4|\xi|^{2}}} +\widehat{c_0} }\bke{\frac{\mp|\xi|^{2}}{\sqrt{1+4|\xi|^{2}}} }, 
\]
\[
\begin{aligned}
  \int_0^t e^{-\lambda_{\pm}(t-\tau)} &\brk{ i\xi \cdot\widehat{  (n \nabla c)}e_1, a_{\pm} } \brk{ b_{\pm},e_{1}} \,\ud \tau\\
&= \int_0^t e^{-\lambda_{\pm}(t-\tau)} \bke{ i\xi \cdot\widehat{  (n \nabla c)}\frac{-2}{1\pm\sqrt{1+4|\xi|^{2}}}  }\bke{\frac{\mp|\xi|^{2}}{\sqrt{1+4|\xi|^{2}}} }\,\ud \tau.
\end{aligned}
\]
A direct computation gives
\begin{equation}\label{nL1LABS}
\begin{aligned}
\abs{e^{-\lambda_{\pm}t} \brk{ \widehat{u_0},a_{\pm} } \brk{ b_{\pm}, e_{1}}} &=\abs{e^{-\lambda_{\pm}t} \bke{ \widehat{n_0}\frac{-2}{1\pm\sqrt{1+4|\xi|^{2}}} +\widehat{c_0} }\bke{\frac{\mp|\xi|^{2}}{\sqrt{1+4|\xi|^{2}}} }}
\\
&\lesssim  e^{-\lambda_{\pm}t} \bke{ |\widehat{n_0}| + ||\xi|\widehat{c_0}| },
\end{aligned}
\end{equation}
and
\begin{equation}\label{nL1NLABS}
\begin{aligned}
 &\abs{ \int_0^t e^{-\lambda_{\pm}(t-\tau)} \brk{ i\xi \cdot\widehat{  (n \nabla c)}e_1, a_{\pm} } \brk{ b_{\pm},e_{1}} \,\ud \tau}
 \\&\qquad=
\abs{\int_0^t e^{-\lambda_{\pm}(t-\tau)} \bke{ i\xi \cdot\widehat{  (n \nabla c)}\frac{-2}{1\pm\sqrt{1+4|\xi|^{2}}}  }\bke{\frac{\mp|\xi|^{2}}{\sqrt{1+4|\xi|^{2}}} }\,\ud \tau} 
\\&\qquad \lesssim
 \int_0^t e^{-\lambda_{\pm}(t-\tau)}   |\xi| |\widehat{  (n \nabla c)}|\,\ud \tau.
 \end{aligned}
\end{equation}
Now, we compute the linear part. Due to $\lambda_{\pm}\ge0$, we can see that
\begin{equation}\label{NL1L}
\begin{aligned}
\|e^{-\lambda_{\pm}t} \bke{ |\widehat{n_0}| + ||\xi|\widehat{c_0}| } \|_{L^{1}}
\lesssim \norm{\widehat{n_0}}_{L^{1}}+\norm{|\xi|\widehat{c_0}}_{L^{1}}.
\end{aligned}
\end{equation}
On the other hand, using
\begin{equation}\label{EXPDEC}
e^{-\lambda_{+}t}\le  e^{-\frac{|\xi|^{2}}{3}t},\qquad \quad
e^{-\lambda_{-}t}\le \left\{
\begin{aligned}
e^{-\frac{|\xi|^{4}}{3}t}\quad&\mbox{if}\quad |\xi| < 1,\\
e^{-\frac{|\xi|^{2}}{3}t}\quad&\mbox{if}\quad |\xi| > 1,
\end{aligned}
\right.
\end{equation}
H\"older's inequality and a direct computation, we obtain
\[
\begin{aligned}
\|e^{-\lambda_{\pm}t} \bke{ |\widehat{n_0}| + ||\xi|\widehat{c_0}|  } \|_{L^{1}}
 \lesssim \sum_{k=1,2}\|e^{-\frac{|\xi|^{2k}}{3}t} \bke{ |\widehat{n_0}| + ||\xi|\widehat{c_0}| } \|_{L^{1}}
  \lesssim \sum_{k=1,2}
t^{-\frac{d}{2k}}\bke{\norm{    \widehat{n_0}  }_{L^{\infty}}  + \norm{|\xi|\widehat{c_0} }_{L^{\infty}}}
\end{aligned}
\]
and thus, combining it with \eqref{NL1L}, it follows from \eqref{nL1LABS} that 
\begin{equation}\label{nL1LR}
\|e^{-\lambda_{\pm}t} \brk{ \widehat{u_0},a_{\pm} } \brk{ b_{\pm}, e_{1}}\|_{L^{1}} \lesssim (1+t)^{-\frac{d}{4}}\bke{\norm{    \widehat{n_0}  }_{L^{1}\cap L^{\infty}}  + \norm{|\xi|\widehat{c_0} }_{L^{1}\cap L^{\infty}}}.
\end{equation}
Similarly, we have
\begin{equation}\label{nLINFL}
\begin{aligned}
\|e^{-\lambda_{\pm}t} \brk{ \widehat{u_0},a_{\pm} } \brk{ b_{\pm}, e_{1}}\|_{L^{\infty}} &\lesssim 
\|e^{-\lambda_{\pm}t} \bke{ |\widehat{n_0}| + ||\xi|\widehat{c_0}| } \|_{L^{\infty}}
  \lesssim\bke{\norm{    \widehat{n_0}  }_{ L^{\infty}}  + \norm{|\xi|\widehat{c_0} }_{  L^{\infty}}}.
\end{aligned}
\end{equation}
Next, we compute the nonlinear part. By \eqref{EXPDEC}, we have that 
\[
\begin{aligned}
\norm {\int_0^t e^{-\lambda_{\pm}(t-\tau)}   |\xi| |\widehat{  (n \nabla c)}|\,\ud \tau}_{L^{1}}&\lesssim \sum_{k=1,2}\norm {\int_0^t e^{-\frac{|\xi|^{2k}}{3}(t-\tau)}   |\xi| |\widehat{  (n \nabla c)}|\,\ud \tau}_{L^{1}}.
\end{aligned}
\]
Using  H\"older's inequality, Young's convolution inequality, and a direct computation, we further compute the right-hand side as
\[
\begin{aligned}
 &\sum_{k=1,2}\norm{\int_0^\frac{t}{2} e^{-\frac{|\xi|^{2k}}{3}(t-\tau)}   |\xi| |\widehat{  (n \nabla c)}|\,\ud \tau}_{L^{1}}
\\&\lesssim  \sum_{k=1,2}\int_0^\frac{t}{2} \|e^{-\frac{|\xi|^{2k}}{3}(t-\tau)}   |\xi|\|_{L^{1}} \norm{\hat{n} }_{L^{\infty}}\norm{|\xi| \hat{c}}_{L^{1}}\,\ud \tau
\\&\lesssim \sum_{k=1,2} \norm{\hat{n} }_{L^{\infty}_{\xi,t}}\|(1+t)^{\frac{d+1}{4}}|\xi| \hat{c}(t)\|_{L^{1}_{\xi}L^{\infty}_{t}}\int_0^\frac{t}{2} (t-\tau)^{-\frac{d+1}{2k}}(1+\tau)^{-\frac{d+1}{4}}\,\ud \tau
\\&\lesssim \sum_{k=1,2} t^{-\frac{d}{2k}}\norm{\hat{n} }_{L^{\infty}_{\xi,t}}\|(1+t)^{\frac{d+1}{4}}|\xi| \hat{c}(t)\|_{L^{1}_{\xi}L^{\infty}_{t}}\int_0^\frac{t}{2} (t-\tau)^{-\frac{1}{2k}}(1+\tau)^{-\frac{d+1}{4}}\,\ud \tau
\\&\lesssim  \sum_{k=1,2}  t^{-\frac{d}{2k}}\norm{\hat{n} }_{L^{\infty}_{\xi,t}}\|(1+t)^{\frac{d+1}{4}}|\xi| \hat{c}(t)\|_{L^{1}_{\xi}L^{\infty}_{t}}
\end{aligned}
\]
and
 \[
\begin{aligned}
 &\sum_{k=1,2}\norm{\int_\frac{t}{2}^{t} e^{-\frac{|\xi|^{2k}}{3}(t-\tau)}   |\xi| |\widehat{  (n \nabla c)}|\,\ud \tau}_{L^{1}}
\\&\lesssim  \sum_{k=1,2}\int_\frac{t}{2}^{t} \|e^{-\frac{|\xi|^{2k}}{3}(t-\tau)}   |\xi|\|_{L^{\infty}} \norm{\hat{n} }_{L^{1}}\norm{|\xi| \hat{c}}_{L^{1}}\,\ud \tau
\\&\lesssim  \sum_{k=1,2} \|(1+t)^{\frac{d}{4}} \hat{n}(t)\|_{L^{1}_{\xi}L^{\infty}_{t}}\|(1+t)^{\frac{d+1}{4}}|\xi| \hat{c}(t)\|_{L^{1}_{\xi}L^{\infty}_{t}}\int_\frac{t}{2}^{t} (t-\tau)^{-\frac{1}{2k}}(1+\tau)^{-\frac{2d+1}{4}}\,\ud \tau
\\&\lesssim \sum_{k=1,2} t^{-\frac{d}{4}}\|(1+t)^{\frac{d}{4}} \hat{n}(t)\|_{L^{1}_{\xi}L^{\infty}_{t}}\|(1+t)^{\frac{d+1}{4}}|\xi| \hat{c}(t)\|_{L^{1}_{\xi}L^{\infty}_{t}}\int_\frac{t}{2}^{t} (t-\tau)^{-\frac{1}{2k}}(1+\tau)^{-\frac{d+1}{4}}\,\ud \tau
\\&\lesssim  t^{-\frac{d}{4}}\|(1+t)^{\frac{d}{4}} \hat{n}(t)\|_{L^{1}_{\xi}L^{\infty}_{t}}\|(1+t)^{\frac{d+1}{4}}|\xi| \hat{c}(t)\|_{L^{1}_{\xi}L^{\infty}_{t}}.
\end{aligned}
\]
Combining it with 
\[
\begin{aligned}
 &\sum_{k=1,2}\norm{\int_0^t e^{-\frac{|\xi|^{2k}}{3}(t-\tau)}   |\xi| |\widehat{  (n \nabla c)}|\,\ud \tau}_{L^{1}}
\\&\lesssim  \sum_{k=1,2}\int_0^t \|e^{-\frac{|\xi|^{2k}}{3}(t-\tau)}   |\xi|\|_{L^{\infty}} \norm{\hat{n} }_{L^{1}}\norm{|\xi| \hat{c}}_{L^{1}}\,\ud \tau
\\&\lesssim \sum_{k=1,2} \|(1+t)^{\frac{d}{4}} \hat{n}(t)\|_{L^{1}_{\xi}L^{\infty}_{t}}\|(1+t)^{\frac{d+1}{4}}|\xi| \hat{c}(t)\|_{L^{1}_{\xi}L^{\infty}_{t}}\int_0^t (t-\tau)^{-\frac{1}{2k}}(1+\tau)^{-\frac{2d+1}{4}}\,\ud \tau
\\&\lesssim  \|(1+t)^{\frac{d}{4}} \hat{n}(t)\|_{L^{1}_{\xi}L^{\infty}_{t}}\|(1+t)^{\frac{d+1}{4}}|\xi| \hat{c}(t)\|_{L^{1}_{\xi}L^{\infty}_{t}}
\end{aligned}
\]  
and \eqref{nL1NLABS}, we have
\begin{equation}\label{nL1NL}
\begin{aligned}
&\norm{ \int_0^t e^{-\lambda_{\pm}(t-\tau)} \brk{ i\xi \cdot\widehat{  (n \nabla c)}e_1, a_{\pm} } \brk{ b_{\pm},e_{1}} \,\ud \tau}_{L^{1}} 
\\&\qquad\lesssim (1+t)^{-\frac{d}{4}}\bke{\norm{\hat{n} }_{L^{\infty}_{\xi,t}}+\|(1+t)^{\frac{d}{4}} \hat{n}(t)\|_{L^{1}_{\xi}L^{\infty}_{t}}}\|(1+t)^{\frac{d+1}{4}}|\xi| \hat{c}(t)\|_{L^{1}_{\xi}L^{\infty}_{t}}.
\end{aligned}
\end{equation}
Similarly, we can compute
\[
\begin{aligned}
\norm {\int_0^t e^{-\lambda_{\pm}(t-\tau)}   |\xi| |\widehat{  (n \nabla c)}|\,\ud \tau}_{L^{\infty}}&\lesssim \sum_{k=1,2}\norm {\int_0^t e^{-\frac{|\xi|^{2k}}{3}(t-\tau)}   |\xi| |\widehat{  (n \nabla c)}|\,\ud \tau}_{L^{\infty}},
\end{aligned}
\]
and
\[
\begin{aligned}
 &\sum_{k=1,2}\norm{\int_0^t e^{-\frac{|\xi|^{2k}}{3}(t-\tau)}   |\xi| |\widehat{  (n \nabla c)}|\,\ud \tau}_{L^{\infty}}
\\&\lesssim  \sum_{k=1,2}\int_0^t \|e^{-\frac{|\xi|^{2k}}{3}(t-\tau)}   |\xi|\|_{L^{\infty}} \norm{\hat{n} }_{L^{\infty}}\norm{|\xi| \hat{c}}_{L^{1}}\,\ud \tau
\\&\lesssim \sum_{k=1,2}\norm{\hat{n} }_{L^{\infty}_{\xi,t}}\|(1+t)^{\frac{d+1}{4}}|\xi| \hat{c}(t)\|_{L^{1}_{\xi}L^{\infty}_{t}}\int_0^t (t-\tau)^{-\frac{1}{2k}}(1+\tau)^{-\frac{d+1}{4}} \,\ud \tau 
\\&\lesssim \norm{\hat{n} }_{L^{\infty}_{\xi,t}}\|(1+t)^{\frac{d+1}{4}}|\xi| \hat{c}(t)\|_{L^{1}_{\xi}L^{\infty}_{t}}
\end{aligned}
\]
and thus,
\begin{equation}\label{nLINFNL}
 \norm{ \int_0^t e^{-\lambda_{\pm}(t-\tau)} \brk{ i\xi \cdot\widehat{  (n \nabla c)}e_1, a_{\pm} } \brk{ b_{\pm},e_{1}} \,\ud \tau}_{L^{\infty}}
  \lesssim   \norm{\hat{n} }_{L^{\infty}_{\xi,t}}\|(1+t)^{\frac{d+1}{4}}|\xi| \hat{c}(t)\|_{L^{1}_{\xi}L^{\infty}_{t}}.
\end{equation}
Combining \eqref{nL1LR}, \eqref{nLINFL}, \eqref{nL1NL}, and  \eqref{nLINFNL}, we have the desired results. 
\qed
\subsection{Proof of Lemma~\ref{LEM43}}\label{secA4}
Let $t\le T<T_{\rm max}$.
Note that  
\[
|\xi|\hat{c}(\xi,t)=|\xi|\sum_{j=+,-}\brk{ \hat{u},a_{j} } \brk{ b_{j} , e_{2}},
\]
\[
|\xi|e^{-\lambda_{\pm}t} \brk{ \widehat{u_0},a_{\pm} } \brk{ b_{\pm}, e_{2}}=e^{-\lambda_{\pm}t} \bke{ \widehat{n_0}\frac{-2|\xi|}{1\pm\sqrt{1+4|\xi|^{2}}} +|\xi|\widehat{c_0} }\bke{\frac{1\pm\sqrt{1+4|\xi|^{2}}}{\pm2\sqrt{1+4|\xi|^{2}}} }, 
\]
and
\[
\begin{aligned}
  |\xi|\int_0^t &e^{-\lambda_{\pm}(t-\tau)} \brk{ i\xi \cdot\widehat{  (n \nabla c)}e_1, a_{\pm} } \brk{ b_{\pm},e_{2}} \,\ud \tau\\
&= \int_0^t e^{-\lambda_{\pm}(t-\tau)} |\xi|\bke{ i\xi \cdot\widehat{  (n \nabla c)}\frac{-2}{1\pm\sqrt{1+4|\xi|^{2}}}  }\bke{\frac{1\pm\sqrt{1+4|\xi|^{2}}}{\pm2\sqrt{1+4|\xi|^{2}}} }\,\ud \tau.
\end{aligned}
\]
First, we compute the linear part. Since $\lambda_{\pm}\ge0$, we  have
\begin{equation}\label{CL1L}
 \norm{ e^{-\lambda_{\pm}t} \bke{ \widehat{n_0}\frac{-2|\xi|}{1\pm\sqrt{1+4|\xi|^{2}}} +|\xi|\widehat{c_0} }\bke{\frac{1\pm\sqrt{1+4|\xi|^{2}}}{\pm2\sqrt{1+4|\xi|^{2}}} }}_{L^{1}}
 \lesssim \norm{\widehat{n_0}}_{L^{1}}+\norm{|\xi|\widehat{c_0}}_{L^{1}}.
\end{equation}
Moreover, using H\"older's inequality and \eqref{EXPDEC}, we have
\[
\begin{aligned}
&\norm{ e^{-\lambda_{+}t} \bke{ \widehat{n_0}\frac{-2|\xi|}{1+\sqrt{1+4|\xi|^{2}}} +|\xi|\widehat{c_0} }\bke{\frac{1+\sqrt{1+4|\xi|^{2}}}{2\sqrt{1+4|\xi|^{2}}} }}_{L^{1}}
\\&\qquad\qquad\qquad\qquad\lesssim \|e^{-\frac{|\xi|^{2}}{3}t}\|_{L^{1}}(\norm{\widehat{n_0}}_{L^{\infty}}+\norm{|\xi|\widehat{c_0}}_{L^{\infty}})
\\&\qquad\qquad\qquad\qquad\lesssim t^{-\frac{d}{2}}(\norm{\widehat{n_0}}_{L^{\infty}}+\norm{|\xi|\widehat{c_0}}_{L^{\infty}}),
\end{aligned}
\]
\[
\begin{aligned}
&\norm{ e^{-\lambda_{-}t} \bke{ \widehat{n_0}\frac{-2|\xi|}{1-\sqrt{1+4|\xi|^{2}}} +|\xi|\widehat{c_0} }\bke{\frac{1-\sqrt{1+4|\xi|^{2}}}{-2\sqrt{1+4|\xi|^{2}}} }}_{L^{1}} 
\\&\qquad\qquad\qquad\qquad\lesssim \sum_{k=1,2}\|e^{-\frac{|\xi|^{2k}}{3}t}|\xi|\|_{L^{1}}(\norm{\widehat{n_0}}_{L^{\infty}}+\norm{|\xi|\widehat{c_0}}_{L^{\infty}})
\\&\qquad\qquad\qquad\qquad \lesssim \sum_{k=1,2}t^{-\frac{d+1}{2k}}(\norm{\widehat{n_0}}_{L^{\infty}}+\norm{|\xi|\widehat{c_0}}_{L^{\infty}}).
\end{aligned}
\]
Thus, combining it with \eqref{CL1L}, we have
\begin{equation}\label{CL1LR}
\||\xi|e^{-\lambda_{\pm}t} \brk{ \widehat{u_0},a_{\pm} } \brk{ b_{\pm}, e_{2}}\|_{L^{1}}\lesssim (1+t)^{-\frac{d+1}{4}}(\norm{\widehat{n_0}}_{L^{1}\cap L^{\infty}}+\norm{|\xi|\widehat{c_0}}_{L^{1}\cap L^{\infty}}).
\end{equation}
Next, we consider the nonlinear part. Note that
\[
\begin{aligned}
\abs{\int_0^t e^{-\lambda_{\pm}(t-\tau)} |\xi|\bke{ i\xi \cdot\widehat{  (n \nabla c)}\frac{-2}{1\pm\sqrt{1+4|\xi|^{2}}}  }\bke{\frac{1\pm\sqrt{1+4|\xi|^{2}}}{\pm2\sqrt{1+4|\xi|^{2}}} }\,\ud \tau}
\\\lesssim  \int_0^t e^{-\lambda_{\pm}(t-\tau)}   \frac{|\xi|^{2}}{\sqrt{1+4|\xi|^{2}}} |\widehat{  (n \nabla c)}|\,\ud \tau.
\end{aligned}
\]
 Due to \eqref{EXPDEC}, we have that 
\[
\begin{aligned}
\norm {\int_0^t e^{-\lambda_{\pm}(t-\tau)}   \frac{|\xi|^{2}}{\sqrt{1+4|\xi|^{2}}} |\widehat{  (n \nabla c)}|\,\ud \tau}_{L^{1}}&\lesssim \sum_{k=1,2}\norm {\int_0^t e^{-\frac{|\xi|^{2k}}{3}(t-\tau)} \frac{|\xi|^{2}}{\sqrt{1+4|\xi|^{2}}} |\widehat{  (n \nabla c)}|\,\ud \tau}_{L^{1}}.
\end{aligned}
\]
Using  H\"older's inequality, Young's convolution inequality, and a direct computation, we further compute the right-hand-side as
\[
\begin{aligned}
&\sum_{k=1,2}\norm{\int_0^\frac{t}{2} e^{-\frac{|\xi|^{2k}}{3}(t-\tau)}   \frac{|\xi|^{2}}{\sqrt{1+4|\xi|^{2}}} |\widehat{  (n \nabla c)}|\,\ud \tau}_{L^{1}}
\\&\qquad\lesssim  \sum_{k=1,2}\int_0^\frac{t}{2} \|e^{-\frac{|\xi|^{2k}}{3}(t-\tau)}   |\xi|^{2}\|_{L^{1}} \norm{\hat{n} }_{L^{\infty}}\norm{|\xi| \hat{c}}_{L^{1}}\,\ud \tau
\\&\qquad\lesssim \sum_{k=1,2} \norm{\hat{n} }_{L^{\infty}_{\xi,t}}\|(1+t)^{\frac{d+1}{4}}|\xi| \hat{c}(t)\|_{L^{1}_{\xi}L^{\infty}_{t}}\int_0^\frac{t}{2} (t-\tau)^{-\frac{d+2}{2k}}(1+t)^{-\frac{d+1}{4}}\,\ud \tau
\\&\qquad\lesssim \sum_{k=1,2} t^{-\frac{d+1}{2k}}\norm{\hat{n} }_{L^{\infty}_{\xi,t}}\|(1+t)^{\frac{d+1}{4}}|\xi| \hat{c}(t)\|_{L^{1}_{\xi}L^{\infty}_{t}}\int_0^\frac{t}{2} (t-\tau)^{-\frac{1}{2k}}(1+t)^{-\frac{d+1}{4}}\,\ud \tau
\\&\qquad\lesssim  \sum_{k=1,2}  t^{-\frac{d+1}{2k}}\norm{\hat{n} }_{L^{\infty}_{\xi,t}}\|(1+t)^{\frac{d+1}{4}}|\xi| \hat{c}(t)\|_{L^{1}_{\xi}L^{\infty}_{t}},
\end{aligned}
\]
and 
 \[
\begin{aligned}
&\sum_{k=1,2}\norm{\int_\frac{t}{2}^{t} e^{-\frac{|\xi|^{2k}}{3}(t-\tau)}   \frac{|\xi|^{2}}{\sqrt{1+4|\xi|^{2}}} |\widehat{  (n \nabla c)}|\,\ud \tau}_{L^{1}}
\\&\qquad\lesssim  \sum_{k=1,2}\int_\frac{t}{2}^{t} \|e^{-\frac{|\xi|^{2k}}{3}(t-\tau)}   |\xi|^{k}\|_{L^{\infty}} \norm{\hat{n} }_{L^{1}}\norm{|\xi| \hat{c}}_{L^{1}}\,\ud \tau
\\&\qquad\lesssim    \|(1+t)^{\frac{d}{4}} \hat{n}(t)\|_{L^{1}_{\xi}L^{\infty}_{t}}\|(1+t)^{\frac{d+1}{4}}|\xi| \hat{c}(t)\|_{L^{1}_{\xi}L^{\infty}_{t}}\int_\frac{t}{2}^{t} (t-\tau)^{-\frac{1}{2}}(1+\tau)^{-\frac{2d+1}{4}}\,\ud \tau
\\&\qquad\lesssim   t^{-\frac{d+1}{4}}\|(1+t)^{\frac{d}{4}} \hat{n}(t)\|_{L^{1}_{\xi}L^{\infty}_{t}}\|(1+t)^{\frac{d+1}{4}}|\xi| \hat{c}(t)\|_{L^{1}_{\xi}L^{\infty}_{t}}\int_\frac{t}{2}^{t} (t-\tau)^{-\frac{1}{2}}(1+\tau)^{-\frac{d}{4}}\,\ud \tau
\\&\qquad\lesssim  t^{-\frac{d+1}{4}}\|(1+t)^{\frac{d}{4}} \hat{n}(t)\|_{L^{1}_{\xi}L^{\infty}_{t}}\|(1+t)^{\frac{d+1}{4}}|\xi| \hat{c}(t)\|_{L^{1}_{\xi}L^{\infty}_{t}}.
\end{aligned}
\]
Combining it with 
\[
\begin{aligned}
 &\sum_{k=1,2}\norm{\int_0^t e^{-\frac{|\xi|^{2k}}{3}(t-\tau)}   \frac{|\xi|^{2}}{\sqrt{1+4|\xi|^{2}}} |\widehat{  (n \nabla c)}|\,\ud \tau}_{L^{1}}
\\&\qquad\lesssim  \sum_{k=1,2}\int_0^t \|e^{-\frac{|\xi|^{2k}}{3}(t-\tau)}   |\xi|\|_{L^{\infty}} \norm{\hat{n} }_{L^{1}}\norm{|\xi| \hat{c}}_{L^{1}}\,\ud \tau
\\&\qquad\lesssim \sum_{k=1,2} \|(1+t)^{\frac{d}{4}} \hat{n}(t)\|_{L^{1}_{\xi}L^{\infty}_{t}}\|(1+t)^{\frac{d+1}{4}}|\xi| \hat{c}(t)\|_{L^{1}_{\xi}L^{\infty}_{t}}\int_0^t (t-\tau)^{-\frac{1}{2k}}(1+t)^{-\frac{2d+1}{4}}\,\ud \tau
\\&\qquad\lesssim  \|(1+t)^{\frac{d}{4}} \hat{n}(t)\|_{L^{1}_{\xi}L^{\infty}_{t}}\|(1+t)^{\frac{d+1}{4}}|\xi| \hat{c}(t)\|_{L^{1}_{\xi}L^{\infty}_{t}},
\end{aligned}
\]  
we have
\begin{equation}\label{CL1NL}
\begin{aligned}
&\norm{\int_0^t e^{-\lambda_{\pm}(t-\tau)}   \frac{|\xi|^{2}}{\sqrt{1+4|\xi|^{2}}} |\widehat{  (n \nabla c)}|\,\ud \tau}_{L^{1}} \\
&\qquad\lesssim (1+t)^{-\frac{d+1}{4}}\bke{\norm{\hat{n} }_{L^{\infty}_{\xi,t}}+\|(1+t)^{\frac{d}{4}} \hat{n}(t)\|_{L^{1}_{\xi}L^{\infty}_{t}}}\|(1+t)^{\frac{d+1}{4}}|\xi| \hat{c}(t)\|_{L^{1}_{\xi}L^{\infty}_{t}}.
\end{aligned}
\end{equation}
Therefore, by \eqref{CL1LR} and \eqref{CL1NL},  we can conclude the desired estimate. 
\qed
\subsection{Proof of Lemma~\ref{LININSTA}}\label{secA5}
For given $\eta\in(0,\eta_0)$, we   take  $ \theta\in(0,1)$ close to $1$ so that \eqref{THETATETA} holds. 
Due to the continuity of $\lambda_{-}(|\xi|)$ and \eqref{LAMML}, we can find 
\begin{equation}\label{EPSTHETA}
\varepsilon=\varepsilon(\theta)\in\biggr{(}0, \frac{1}{2}\sqrt{\frac{A^{2}-1}{4A}}\biggr{)}
\end{equation} such that
\begin{equation}\label{LAMMLL}
\lambda_{-}(|\xi|)\le -\theta\frac{ (A-1)^2 }{4A  }\quad\mbox{ if }\quad  \biggr{|}|\xi|- \sqrt{\frac{A^{2}-1}{4A}} \biggr{|}\le \varepsilon.
\end{equation}
We denote $\sigma_{d}=|\partial B_{1}(0)|$.
With such $\varepsilon$, we set 
\[
n_{0,\eta}=\eta C_{\theta}\mathscr{F}^{-1}\biggr{(}\chi_{\Bigr{\{} \abs{ |\xi|- \sqrt{\frac{A^{2}-1}{4A}}  }\le \varepsilon\Bigr{\}}}\biggr{)},\qquad C_{\theta}= \displaystyle\Biggr{(}\sigma_{d}\int_{\sqrt{\frac{A^{2}-1}{4A}} -\varepsilon}^{\sqrt{\frac{A^{2}-1}{4A}} +\varepsilon}(1+r^{2})^{s}r^{d-1}\,\ud r\Biggr{)}^{-\frac{1}{2}},
\]
and 
$c_{0,\eta}\equiv 0$.
Then, 
\begin{equation}\label{NOHSNORM}
\begin{aligned}
\|n_{0,\eta}\|_{H^{s}}&=\eta C_{\theta}\Bigr{(}\int_{\R^{d}}(1+|\xi|^{2})^{s}|\widehat{n_{0,\eta}}(\xi)|^{2}\,\ud\xi\Bigr{)}^{\frac{1}{2}}=\eta,
\end{aligned}
\end{equation}
and $\|c_{0,\eta}\|_{H^{s+1}}=0$. Moreover, we can compute the lower and upper bounds of 
\[
\begin{aligned}
\|e^{-\lambda_{-}t} \langle \widehat{u_{0,\eta}},a_{-} \rangle \langle b_{-},e_{1} \rangle  \|_{L^{2}}&=\norm{ e^{-\lambda_{-}t}\frac{-2}{1-\sqrt{1+4A|\xi|^2{}}}\frac{A|\xi|^{2}}{\sqrt{1+4A|\xi|^{2}}}\eta C_{\theta}\chi_{\Bigr{\{} \abs{ |\xi|- \sqrt{\frac{A^{2}-1}{4A}}  }\le \varepsilon\Bigr{\}}}
 }_{L^{2}}
 \\&=\norm{ e^{-\lambda_{-}t}\frac{1+\sqrt{1+4A|\xi|^{2}}}{2\sqrt{1+4A|\xi|^{2}}}\eta C_{\theta}\chi_{\Bigr{\{} \abs{ |\xi|- \sqrt{\frac{A^{2}-1}{4A}}  }\le \varepsilon\Bigr{\}}}
 }_{L^{2}}
\\
\end{aligned}
\] 
as
\[
\begin{aligned}
e^{\frac{\theta(A-1)^{2}}{4A}}\| \langle \widehat{u_{0,\eta}},a_{-} \rangle \langle b_{-},e_{1} \rangle  \|_{L^{2}}
&=\norm{ e^{\theta\frac{(A-1)^{2}}{4A}t}\frac{1+\sqrt{1+4A|\xi|^{2}}}{2\sqrt{1+4A|\xi|^{2}}}\eta C_{\theta}\chi_{\Bigr{\{} \abs{ |\xi|- \sqrt{\frac{A^{2}-1}{4A}}  }\le \varepsilon\Bigr{\}}}
 }_{L^{2}}
\\& \le \|e^{-\lambda_{-}t} \langle \widehat{u_{0,\eta}},a_{-} \rangle \langle b_{-},e_{1} \rangle  \|_{L^{2}}
\\& \le \norm{ e^{\frac{(A-1)^{2}}{4A}t}\frac{1+\sqrt{1+4A|\xi|^{2}}}{2\sqrt{1+4A|\xi|^{2}}}\eta C_{\theta}\chi_{\Bigr{\{} \abs{ |\xi|- \sqrt{\frac{A^{2}-1}{4A}}  }\le \varepsilon\Bigr{\}}}
 }_{L^{2}}
\\&  = 
 e^{\frac{(A-1)^{2}}{4A}}\| \langle \widehat{u_{0,\eta}},a_{-} \rangle \langle b_{-},e_{1} \rangle  \|_{L^{2}},
 \end{aligned}
\]
where we used \eqref{LAMML} and \eqref{LAMMLL}. This completes the proof.
\qed

\ackn

JA acknowledges support of the National Research Foundation (NRF) of Korea (Grant No. 
RS-2024-00336346 and RS-2025-24523482). The work of JK was supported by the National Research Foundation of Korea(NRF) grant funded by the Korea government(MSIT) (No. RS-2024-00360798).

\end{document}